\documentclass[11pt]{amsart}
\usepackage{amssymb,amsmath,mathtools,bm}
\usepackage{a4wide}
\newtheorem{theorem}{Theorem}[section]
\newtheorem{lemma}[theorem]{Lemma}
\newtheorem{proposition}[theorem]{Proposition}

\newtheorem{claim}{Claim}[section]
\theoremstyle{remark}
\newtheorem{remark}{Remark}[section]
\numberwithin{equation}{section}
\newcommand{\R}{\mathbb{R}}

\newcommand{\N}{\mathbb{N}}
\newcommand{\Z}{\mathbb{Z}}

\newcommand{\la}{\langle}
\newcommand{\ra}{\rangle}
\newcommand{\pd}{\partial}
\renewcommand{\a}{\alpha}
\newcommand{\eps}{\varepsilon}
\newcommand{\wP}{\widetilde{P}}
\newcommand{\wR}{\widetilde{R}}
\newcommand{\wS}{\widetilde{S}}
\newcommand{\cC}{\mathcal{C}}
\newcommand{\wC}{\widetilde{\mathcal{C}}}
\newcommand{\mF}{\mathcal{F}}
\newcommand{\mL}{\mathcal{L}}

\newcommand{\bS}{\bar{S}}
\newcommand{\bM}{\mathbb{M}}
\newcommand{\cN}{\mathcal{N}}
\newcommand{\fN}{\mathfrak{N}}

\newcommand{\ta}{\tilde{a}}
\newcommand{\tc}{\tilde{c}}

\newcommand{\tu}{\tilde{u}}
\newcommand{\tv}{\tilde{v}}
\newcommand{\tx}{\tilde{x}}
\newcommand{\bb}{\mathbf{b}}
\newcommand{\bd}{\mathbf{d}}
\newcommand{\tpsi}{\tilde{\psi}}

\DeclareMathOperator{\sech}{sech}

\DeclareMathOperator{\supp}{supp}
\DeclareMathOperator{\diag}{diag}

\begin{document}
\title[Stability of line solitons, II.]{Stability of line solitons \\ for the KP-II equation in $\R^2$, II.}
\author{Tetsu Mizumachi}
\email{tetsum@hiroshima-u.ac.jp}
\address{Department of Mathematics,
Graduate School of Science,
Hiroshima University\linebreak
1-7-1 Kagamiyama,
Higashi-Hiroshima 739-8521,
Japan}
\begin{abstract}
The KP-II equation was derived by Kadmotsev and Petviashvili \cite{KP}
to explain stability of line solitary waves of shallow water.
Recently, Mizumachi \cite{Mi} has proved nonlinear stability
of $1$-line solitons for exponentially localized perturbations.
In this paper,  we prove stability of $1$-line solitons
for perturbations in $(1+x^2)^{-1/2-0}H^1(\R^2)$ and perturbations
in $H^1(\R^2)\cap \pd_xL^2(\R^2)$.
\end{abstract}

\maketitle
\section{Introduction}
\label{sec:intro}
The KP-II equation
\begin{equation}\label{eq:KPII}
\partial_x(\pd_tu+\pd_x^3u+3\pd_x(u^2))+3\partial_y^2u=0
\quad\text{for $t>0$ and $(x,y)\in \R^2$,}
\end{equation}
is a generalization to two spatial dimensions of the KdV equation
\begin{equation}
  \label{eq:KdV}
\pd_tu+\pd_x^3u+3\pd_x(u^2)=0\,,
\end{equation}
and has been derived as a model in the study of the transverse stability
of solitary wave solutions to the KdV equation
with respect to two dimensional perturbation when the surface tension
is weak or absent.  See \cite{KP} for the derivation of \eqref{eq:KPII}.
\par
The global well-posedness of \eqref{eq:KPII} in $H^s(\R^2)$ ($s\ge0$)
on the background of line solitons has been studied by Molinet,
Saut and Tzvetkov \cite{MST} whose  proof is based on the work of
Bourgain \cite{Bourgain}.
For the other contributions on the Cauchy problem of the KP-II equation,
see e.g. \cite{GPS,Hadac,HHK,IM, Tak,TT,Tz,Ukai} and the references therein.
\par
Let
$$\varphi_c(x)\equiv c\cosh^{-2}\Big(\sqrt{\frac{c}{2}}\,x\Big),\quad c>0.
$$
Then $\varphi_c(x-2ct)$ is a solitary wave solution of the KdV equation
\eqref{eq:KdV} and a line soliton solution of \eqref{eq:KPII} as well.
\par
Let us briefly explain known results on stability of $1$-solitons for
the KdV equation first.
Stability of the $1$-soliton $\varphi_c(x-2ct)$ of \eqref{eq:KdV}
was proved by \cite{Benjamin,Bona,We} 
using the fact that $\varphi_c$ is a minimizer of the Hamiltonian
on the manifold $\{u\in H^1(\R)\mid \|u\|_{L^2(\R)}=\|\varphi_c\|_{L^2(\R)}\}$.
As is well known, a solitary wave of the KdV equation travels
at a speed faster than the maximum group velocity of linear waves
and the larger solitary wave moves faster to the right.
Using this property, Pego and Weinstein \cite{PW} prove asymptotic stability
of solitary wave solutions of \eqref{eq:KdV} in an exponentially
weighted space. Later, Martel and Merle established the Liouville theorem
for the generalized KdV equations by using a virial type identity
and prove the asymptotic stability of solitary waves in $H^1_{loc}(\R)$
(see e.g. \cite{MM}). For stability of multi-solitons
of the generalized KdV equations, see \cite{MMT}.
\par
For the KP-II equation, its Hamiltonian is infinitely indefinite and
the variational approach such as \cite{GSS} is not applicable.  Hence
it seems natural to study stability of line solitons using strong
linear stability of line solitons.  Spectral transverse stability of
line solitons of \eqref{eq:KPII} has been studied by
\cite{APS,Burtsev}.  See also \cite{Haragus} for transverse linear
stability of cnoidal waves.  Alexander \textit{et al.} \cite{APS}
proved that the spectrum of the linearized operator in $L^2(\R^2)$
consists of the entire imaginary axis.  On the other hand, in an
exponentially weighted space where the size of perturbations are
biased in the direction of motion, the spectrum of the linearized
operator consists of a curve of resonant continuous eigenvalues which
goes through $0$ and the set of continuous spectrum which locates in
the stable half plane and is away from the imaginary axis
(see \cite{Burtsev, Mi}).  The former one appears because line solitons
are not localized in the transversal direction and $0$, which is
related to the symmetry of line solitons, cannot be an isolated
eigenvalue of the linearized operator.  Such a situation is common
with planer traveling wave solutions for the heat equation.  See
e.g. \cite{Kapitula, Le-Xin, Xin}.
\par
Using the inverse scattering method, Villarroel and Ablowitz \cite{VA}
studied solutions of around line solitons for \eqref{eq:KPII}.
Recently, Mizumachi \cite{Mi} has proved transversal stability of 
line soliton solutions of \eqref{eq:KPII} for exponentially localized
perturbations. The idea is to use the exponential decay property
of the linearized equation satisfying a secular term condition and
describe variations of local amplitudes and local inclinations of
the crest of modulating line solitons by a system of Burgers equations.
\par

The purpose of the present paper is to prove transverse stability 
of the line soliton solutions for perturbations which are the
$x$-derivative of $L^2(\R^2)$ functions and for polynomially localized
perturbations. Now let us introduce our results.
\begin{theorem}
  \label{thm:stability}
  Let $c_0>0$ and $u(t,x,y)$ be a solution of
\eqref{eq:KPII} satisfying  $u(0,x,y)=\varphi_{c_0}(x)+v_0(x,y)$.
There exist positive constants $\eps_0$ and $C$
  satisfying the following: if
 $v_0\in H^{1/2}(\R^2)\cap \pd_xL^2(\R^2)$ and
$\|v_0\|_{L^2(\R^2)}+\||D_x|^{1/2}v_0\|_{L^2}+\||D_x|^{-1/2}|D_y|^{1/2}v_0\|_{L^2(\R^2)} <\eps_0$
then there exist $C^1$-functions $c(t,y)$ and $x(t,y)$
such that for every $t\ge0$ and $k\ge0$,
\begin{align}
\label{OS} &
\|u(t,x,y)-\varphi_{c(t,y)}(x-x(t,y))\|_{L^2(\R^2)}\le C\|v_0\|_{L^2}\,,\\
\label{phase-sup} & 
\left\|c(t,\cdot)-c_0\right\|_{H^k(\R)}+\left\|\pd_yx(t,\cdot)\right\|_{H^k(\R)}
+\|x_t(t,\cdot)-2c(t,\cdot)\|_{H^k(\R)}\le C\|v_0\|_{L^2}\,,\\
\label{phase2} &  
\lim_{t\to\infty}\left(\left\|\pd_yc(t,\cdot)\right\|_{H^k(\R)}+
\left\|\pd_y^2x(t,\cdot)\right\|_{H^k(\R)}\right)=0\,,
\end{align}
and for any$R>0$,
\begin{equation}
  \label{AS}
\lim_{t\to\infty}
\left\|u(t,x+x(t,y),y)-\varphi_{c(t,y)}(x)\right\|_{L^2((x>-R)\times\R_y)}=0\,.
\end{equation}
\end{theorem}

\begin{theorem}
  \label{thm:poly}
Let $c_0>0$ and $s>1$. Suppose that $u$ is a solutions of 
\eqref{eq:KPII} satisfying  $u(0,x,y)=\varphi_{c_0}(x)+v_0(x,y)$.
Then there exist positive constants $\eps_0$ and $C$ such that 
if $\|\la x\ra^sv_0\|_{H^1(\R^2)} <\eps_0$,
there exist $c(t,y)$ and $x(t,y)$ satisfying
\eqref{phase2}, \eqref{AS} and
\begin{align}
\label{OS'} &
\|u(t,x,y)-\varphi_{c(t,y)}(x-x(t,y))\|_{L^2(\R^2)}
\le C\|\la x\ra^sv_0\|_{H^1(\R^2)}\,,\\
\label{phase-sup'} & 
\left\|c(t,\cdot)-c_0\right\|_{H^k(\R)}+\left\|\pd_yx(t,\cdot)\right\|_{H^k(\R)}
+\|x_t(t,\cdot)-2c(t,\cdot)\|_
{H^k(\R)}\le C\|\la x\ra^sv_0\|_{H^1(\R^2)}
\end{align}
for every $t\ge0$ and $k\ge0$.
\end{theorem}

\begin{remark}
By \eqref{phase-sup} and \eqref{phase2},
$$\lim_{t\to\infty}\sup_{y\in\R}(|c(t,y)-c_0|+|x_y(t,y)|)=0\,,$$
and as $t\to\infty$, the modulating line soliton $\varphi_{c(t,y)}(x-x(t,y))$
converges to a $y$-independent modulating line soliton
$\varphi_{c_0}(x-x(t,0))$ in $L^2(\R_x\times (|y|\le R))$ for any $R>0$.
Hence it follows from \eqref{AS} that
$$\lim_{t\to\infty}
\left\|u(t,x+x(t,0),y)-\varphi_{c_0}(x)
\right\|_{L^2((x>-R)\times(|y|\le R))}=0\,.$$
We remark that the phase shift $x(t,y)$ in \eqref{OS} and \eqref{AS}
cannot be uniform in $y$ because of the variation of the local phase shift
around $y=\pm2\sqrt{2c_0}t+O(\sqrt{t})$.
See Theorems~1.4 and 1.5 in \cite{Mi}.
\end{remark}

\begin{remark}
The KP-II equation has no localized solitary waves (see \cite{dBM,dBS}).
On the other hand, the KP-I equation has stable ground states
(see \cite{dBS,LiuW}) and line solitons of the KP-I equation are unstable
(see \cite{RT1,RT2,Z}). See e.g. \cite{Kl-Saut} and the references therein
for numerical studies of KP-type equations.
\end{remark}

\begin{remark}
Following the idea of Merle and Vega \cite{MV},
Mizumachi and Tzvetkov \cite{MT} 
used the Miura transformation to prove stability of line soliton
solutions to the perturbations which are periodic in the transverse directions.
They prove that the Miura transformation gives a local isomorphism
between solutions around a $1$-line soliton and solutions
around the null solution of KP-II via solutions around a kink of MKP-II.
\par
The argument in \cite{MT} fails for localized perturbations because in
view of the resonant continuous eigenvalues of MKP-II in
$L^2(\R^2;e^{2\a x}dxdy)$ with $\a\in(0,\sqrt{2c_0})$ (see Lemma~2.5
in \cite{Mi}), the motion of waves along the crest of modulating line
kink of MKP-II is expected to be unilateral, whereas the wave motion
along the crest of a modulating line soliton for the KP-II equation is
bidirectional (see Theorem~1.5 in \cite{Mi}).
\end{remark}
Now let us explain our strategy of the proof.
To prove stability of line solitons in \cite{Mi}, we rely on the fact that
solutions of the linearized equation decay exponentially in exponentially
weighted norm as $t\to\infty$ if data are orthogonal to the adjoint
resonant continuous eigenmodes.
To describe the behavior of solutions around a line soliton, we
represent them by using an ansatz
\begin{equation}
  \label{eq:ansatz0}
u(t,x,y)=\varphi_{c(t,y)}(z)-\psi_{c(t,y)}(z+3t)+v(t,z,y)\,,\quad
z=x-x(t,y)\,,
\end{equation}
where $c(t,y)$ and $x(t,y)$ are the local amplitude and the local
phase shift of the modulating line soliton $\varphi_{c(t,y)}(x-x(t,y))$
at time $t$ along the line parallel to the $x$-axis and
$\psi_{c(t,y)}$ is an auxiliary function so that
$$\int_\R v(t,z,y)\,dz=\int_\R v(0,z,y)\,dz\quad\text{for any $y\in\R$.}$$
One of the key step is to prove $\|v(t)\|_{L^2_{loc}}$
is square integrable in time. 
In \cite{Mi}, we impose a non secular condition on $v(t)$ such that
the perturbation $v(t)$ is orthogonal to the adjoint resonant
eigenfunctions in order to apply the strong linear stability property
of line solitons (see Proposition~\ref{prop:semigroup-est}
in Section~\ref{sec:preliminaries})to $v$.
Since the adjoint resonant eigenfunctions grow exponentially as
$x\to\infty$, the secular term condition is not feasible for $v(t)$
which is not exponentially localized as $x\to\infty$.  
Following the idea of \cite{Mi1,MPQ,MT}, we split the perturbation $v(t)$
into a sum of a small solution $v_1(t)$ of \eqref{eq:KPII} satisfying
$v_1(0)=v_0$ and the remainder part $v_2(t)$.
As is the same with other long wave models, the solitary wave part
moves faster than the freely propagating freely propagating perturbations
and the localized $L^2$-norms of $v_1$ are square integrable in time
thanks to the virial identity. 
The remainder part $v_2(t)$ is exponentially localized
as $x\to\infty$ and is mainly driven by the interaction between
$v_1$ and the line soliton. We impose the secular term condition on $v_2$
to apply the linear stability estimate.
Using the linear stability estimate as well as a virial type identity,
we have the square integrability
of $\|e^{\a z}v_2(t)\|_{L^2}$ in time for small $\a>0$.
\par
For Boussinesq equations, Pedersen \cite{Ped} heuristically observed that
modulation of line solitary waves are described by a system of Burgers
equations. We expect the method presented in this paper is applicable
to the other $2$-dimensional long wave models.
\par

Our plan of the present paper is as follows.
In Section~\ref{sec:preliminaries}, we recollect strong linear stability
property of line solitons that are proved in \cite{Mi}.
In Section~\ref{sec:decomp}, we decompose a solution around line solitons
into a sum of the modulating line soliton $\varphi_{c(t,y)}(z)$,
a small freely propagating part $v_1$, an exponentially localized remainder
part $v_2$ and an auxiliary function $\psi_{c(t,y)}$.
In Section~\ref{sec:modulation}, we compute the time derivative
of the secular term condition on $v_2$ and derive a system of Burgers equations
that describe the local amplitude $c(t,y)$ and the local phase shift
$x(t,y)$.
In Section~\ref{sec:apriori}, we estimate $\tc(t):=c(t)-c_0$ and $x_y(t)$.
In the present paper, $\tc(t)$ and $x_y(t)$ are not necessarily
pseudo-measures and we are not able to estimate $\mathcal{F}^{-1}L^\infty-L^2$
estimates for $\tc$ and $x_y$. Instead, we use the monotonicity formula
to obtain time global bounds for $\tc(t)$ and $x_y(t)$.
Since the terms related to $v_1(t)$ are merely square integrable in time
and cubic terms that appear in the energy identity are not necessarily
integrable in time, we use a change of variables to eliminate these
terms to obtain time global estimates.
In Section~\ref{sec:L2norm}, we estimate the $L^2$-norm of
the remainder term $v$.
In Section~\ref{sec:v1}, we introduce several estimates
for $v_1$ which is a small solution of \eqref{eq:KPII}.
First, we show that a virial identity by \cite{dBM} ensures that
localized norm of $v_1$ is square integrable in time.
Then, we explain that the nonlinear scattering theory in \cite{HHK}
gives a time global bound for $L^p$-norms with $p>2$ if $v_1(0)=v_0\in
|D_x|^{1/2}L^2(\R^2)$ and $v_0$ is sufficiently smooth.
In Section~\ref{sec:X-norm}, we estimate the exponentially weighted norm of
$v_2$ following the lines of \cite{Mi}. We use the semigroup estimate
introduced in Section~\ref{sec:preliminaries} to estimate the low frequencies
in $y$ and apply a virial type estimate to estimate high frequencies in $y$
to avoid a loss of derivatives.
Since we split the perturbation $v$ into two parts $v_1$ and $v_2$,
we cannot cancel the derivative of the nonlinear term by integration 
by parts and we need a time global bound of $\|v_1(t)\|_{L^3}$
to estimate the exponentially localized energy norm of $v_2(t)$
by using the virial identity.
In Sections~\ref{sec:thm1} and \ref{sec:poly},
we prove Theorems~\ref{thm:stability} and \ref{thm:poly}.

\par
Finally, let us introduce several notations. 
For Banach spaces $V$ and $W$, let $B(V,W)$ be the space of all
linear continuous operators from $V$ to $W$ and let
$\|T\|_{B(V,W)}=\sup_{\|x\|_V=1}\|Tu\|_W$ for $T\in B(V,W)$.
We abbreviate $B(V,V)$ as $B(V)$.
For $f\in \mathcal{S}(\R^n)$ and $m\in \mathcal{S}'(\R^n)$, let 
\begin{gather*}
(\mathcal{F}f)(\xi)=\hat{f}(\xi)
=(2\pi)^{-n/2}\int_{\R^n}f(x)e^{-ix\xi}\,dx\,,\\
(\mathcal{F}^{-1}f)(x)=\check{f}(x)=\hat{f}(-x)\,,\quad
(m(D_x)f)(x)=(2\pi)^{-n/2}(\check{m}*f)(x)\,.
\end{gather*}
We use $a\lesssim b$ and $a=O(b)$ to mean that there exists a
positive constant such that $a\le Cb$. 
Various constants will be simply denoted
by $C$ and $C_i$ ($i\in\mathbb{N}$) in the course of the
calculations. We denote $\la x\ra=\sqrt{1+x^2}$ for $x\in\R$.
\bigskip

\section*{Acknowledgment}
The author would like to express his gratitude to
Professor~Nikolay Tzvetkov for helpful discussions and his hospitality
during the author's visit to University of Cergy-Pontoise.
This research is supported by Grant-in-Aid for Scientific
Research (No. 21540220).
\bigskip

\section{Preliminaries}
\label{sec:preliminaries}
In this section, we recollect decay estimates of the semigroup
generated by the linearized operator around a $1$-line soliton
in exponentially weighted spaces.
\par
Since \eqref{eq:KPII} is invariant under the scaling
$u\mapsto \lambda^2 u(\lambda^3t,\lambda x,\lambda^2y)$,
we may assume  $c_0=2$ in Theorems~\ref{thm:stability} and \ref{thm:poly}
without loss of generality. 
Let $$\varphi=\varphi_2\,,\quad
\mL=-\pd_x^3+4\pd_x-3\pd_x^{-1}\pd_y^2-6\pd_x(\varphi \cdot)\,.$$
We remark that $e^{t\mL}$ is a $C^0$-semigroup on $X:=L^2(\R^2;e^{2\a x}dxdy)$
for any $\a>0$ because $\mL_0:=-\pd_x^3+4\pd_x-3\pd_x^{-1}\pd_y^2$
is $m$-dissipative on $X$ and  $\mL-\mL_0$ is infinitesimally small
with respect to $\mL_0$.
\par
We have the following exponential decay estimates for $e^{t\mL_0}$ on $X$.
\begin{lemma}\emph{(\cite[Lemma~3.4]{Mi})}
  \label{lem:free-semigroup}
Suppose $\a>0$. Then there exists a positive constant
$C$ such that for every $f\in C_0^\infty(\R^2)$ and $t>0$,
  \begin{gather*}
\|e^{t\mL_0}f\|_X \le Ce^{-\a(4-\a^2)t}\|f\|_X\,,\\
\|e^{t\mL_0}\pd_x f\|_X +\|e^{t\mL_0}\pd_x^{-1}\pd_y f\|_X 
\le C(1+t^{-1/2})e^{-\a(4-\a^2)t}\|f\|_X\,,\\
\|e^{t\mL_0}\pd_x f\|_X 
\le C(1+t^{-3/4})e^{-\a(4-\a^2)t}\|e^{ax}f\|_{L^1_xL^2_y}\,.
  \end{gather*}
\end{lemma}

Solutions of $\pd_tu=\mL u$ satisfying a \textit{secular term condition}
decay like solutions to the free equation $\pd_tu=\mL_0u$.
To be more precise, let us introduce a family of
continuous resonant eigenvalues near $0$ and the corresponding
continuous eigenfunctions of the linearized operator $\mL$.
Let 
\begin{gather*}
\beta(\eta)=\sqrt{1+i\eta}\,,\quad \lambda(\eta)=4i\eta\beta(\eta)\,,\\
g(x,\eta)=\frac{-i}{2\eta\beta(\eta)}
\pd_x^2(e^{-\beta(\eta)x}\sech x),\quad
g^*(x,\eta)=\pd_x(e^{\beta(-\eta)x}\sech x)\,.
\end{gather*}
Then
$$\mL(\eta)g(x,\pm\eta)=\lambda(\pm\eta)g(x,\pm\eta)\,,\quad
\label{eq:lem-kp2}
\mL(\eta)^*g^*(x,\pm\eta)=\lambda(\mp\eta)g^*(x,\pm\eta)\,.$$
Now we define a spectral projection to the resonant eigenmodes
$\{g_\pm(x,\eta)\}$. 
Let 
\begin{gather*}
g_1(x,\eta)=2\Re g(x,\eta)\,,\quad g_2(x,\eta)=-2\eta\Im g(x,\eta)\,,\\
g_1^*(x,\eta)=\Re g^*(x,\eta)\,,\quad g_2^*(x,\eta)=-\eta^{-1}\Im g^*(x,\eta)\,,
\end{gather*}
and $P_0(\eta_0)$ be a projection
to resonant modes defined by
\begin{equation*}
P_0(\eta_0)f(x,y)=\frac{1}{2\pi}\sum_{k=1,\,2}
\int_{-\eta_0}^{\eta_0}a_k(\eta)g_k(x,\eta)e^{iy\eta}\,d\eta\,,  
\end{equation*}
\begin{align*}
 a_k(\eta)=&\int_\R \lim_{M\to\infty}\left(\int_{-M}^M f(x_1,y_1)e^{-iy_1\eta}
\,dy_1\right)\overline{g_k^*(x_1,\eta)}\,dx_1
\\=& \sqrt{2\pi}\int_\R (\mF_yf)(x,\eta)\overline{g_k^*(x,\eta)}\,dx\,.
  \end{align*}
For $\eta_0$ and $M$ satisfying $0<\eta_0\le M\le \infty$, let
\begin{gather*}
P_1(\eta_0, M)u(x,y):=\frac{1}{2\pi}\int_{\eta_0\le |\eta|\le M}
\int_\R  u(x,y_1)e^{i\eta(y-y_1)}dy_1d\eta\,,
\\ P_2(\eta_0,M):= P_1(0,M)-P_0(\eta_0)\,.
\end{gather*}
Then we have the following.
\begin{proposition}\emph{(\cite[Proposition~3.2 and Corollary~3.3]{Mi})}
\label{prop:semigroup-est}
Let $\a\in (0,2)$ and $\eta_1$ be a positive number satisfying
$\Re\beta(\eta_1)-1<\a$.  Then there exist positive constants
$K$ and  $b$ such that for any $\eta_0\in(0,\eta_1]$, $M\ge \eta_0$,
$f\in X$ and $t\ge 0$,
$$ \|e^{t\mL}P_2(\eta_0,M)f\|_X\le Ke^{-bt}\|f\|_X\,.$$
Moreover, there exist positive constants $K'$ and $b'$ such that for $t>0$,
\begin{gather*}
\|e^{t\mL}P_2(\eta_0,M)\pd_xf\|_X \le
K'e^{-b' t}t^{-1/2}\|e^{ax}f\|_X\,,\\
\|e^{t\mL}P_2(\eta_0,M)\pd_xf\|_X \le
K'e^{-b' t}t^{-3/4}\|e^{ax}f\|_{L^1_xL^2_y}\,.  
\end{gather*}
\end{proposition}
\bigskip

\section{Decomposition of the perturbed line soliton}
\label{sec:decomp}
Let us decompose a solution around a line soliton solution
$\varphi(x-4t)$ into a sum of a modulating line soliton
and a non-resonant dispersive
part plus a small wave which is caused by amplitude changes of the line
soliton:
\begin{equation}
  \label{eq:decomp}
u(t,x,y)=\varphi_{c(t,y)}(z)-\psi_{c(t,y),L}(z+3t)+v(t,z,y)\,,\quad
z=x-x(t,y)\,,
\end{equation}
where
$\psi_{c,L}(x)=2(\sqrt{2c}-2)\psi(x+L)$,
$\psi(x)$ is a nonnegative function such that
$\psi(x)=0$ if  $|x|\ge1$ and that $\int_\R \psi(x)\,dx=1$
and $L>0$ is a large constant to be fixed later.
The modulation parameters $c(t_0,y_0)$ and $x(t_0,y_0)$ denote
the maximum height and the phase shift of the modulating line soliton
$\varphi_{c(t,y)}(x-x(t,y))$ along the line $y=y_0$ at the time $t=t_0$,
and $\psi_{c,L}$ is an auxiliary function such that
\begin{equation}
  \label{eq:0mean}
\int_\R \psi_{c,L}(x)\,dx=\int_\R(\varphi_c(x)-\varphi(x))\,dx\,.
\end{equation}
Since a localized solution to KP-type equations
satisfies $\int_\R u(t,x,y)\,dx=0$ for any $y\in\R$ and $t>0$
(see \cite{MST_contr}), it is natural to expect small perturbations
appear in the rear of the solitary wave if the solitary wave is amplified.
\par
To utilize exponential linear stability of line solitons for solutions
that are not exponentially localized in space,
we further decompose $v$ into a small solution of \eqref{eq:KPII}
and an exponentially localized part following the idea of \cite{Mi1}
(see also \cite{MPQ,MT2}).
Let $\tv_1$ be a solution of
\begin{equation}
\label{eq:tv1}
\left\{\begin{aligned}
& \pd_t\tv_1+\pd_x^3\tv_1+3\pd_x(\tv_1^2)+3\pd_x^{-1}\pd_y^2\tv_1=0\,,\\
& \tv_1(0,x,y)=v_0(x,y)\,,    
  \end{aligned}\right.
\end{equation}
and
\begin{equation}
  \label{eq:decomp2}
v_1(t,z,y)=\tv_1(t,z+x(t,y),y)\,,\quad v_2(t,z,y)=v(t,z,y)-v_1(t,z,y)\,.
\end{equation}
Obviously, we have $v_2(0)=0$ and $v_2(t)\in X:=L^2(\R^2;e^{2\a z}dzdy)$
for $t\ge0$ as long as the decomposition \eqref{eq:decomp} persists. Indeed,
we have the following.
\begin{lemma}
\label{lem:diff}
Let $v_0\in H^{1/2}(\R^2)$ and $\tv_1(t)$ be a solution of \eqref{eq:tv1}.
Suppose $u(t)$ is a solution of \eqref{eq:KPII} satisfying
$u(0,x,y)=\varphi(x)+v_0(x,y)$. 
Let 
$w(t,x,y)=u(t,x+4t,y)-\varphi(x)-\tv_1(t,x+4t,y)$.
Then for any $\a\in [0,1)$,
\begin{gather}
  \label{eq:diff-utv1}
w \in  C([0,\infty);X)\,,\\
\label{eq:winX3}
\pd_xw\,,\enskip \pd_x^{-1}\pd_yw\in L^2(0,T;X)
\quad\text{for every $T>0$.}
\end{gather}
Moreover if $v_0\in  \pd_xL^2(\R^2)$ in addition, then
\begin{equation}
  \label{eq:pd-1bu}
\pd_x^{-1}\left(u(t,x,y)-\varphi(x-4t)\right)\in 
C([0,\infty);L^2(\R^2))\,.  
\end{equation}
\end{lemma}
We remark that by \cite{MST}, 
$\pd_xw$, $\pd_x^{-1}\pd_yw \in L^\infty_xL^2([-T,T]\times \R_y)$ for
any $T>0$ provided $v_0\in L^2(\R^2)$.
To prove Lemma~\ref{lem:diff}, we use the following imbedding
inequalities.
\begin{claim}
  \label{cl:aniso}
Let $p_n(x)=e^{2\a nx}(1+\tanh\a(x-n))$. 
There exists a positive constant $C$ such that for every $n\in \N$,
\begin{equation}
  \label{eq:aniso1}
\begin{split}
& \int_{\R^2} p_n'(x)^3w^6(s,x,y)\,dxdy
\\ \le & C\left[
\int_{\R^2} p_n'(x)
\left\{(\pd_xw)^2+(\pd_x^{-1}\pd_yw)^2+w^2\right\}(s,x,y)\,dxdy\right]^3\,.
\end{split}
\end{equation}
Moreover for any $p\in[2,6]$, 
\begin{equation}
  \label{eq:aniso2}
  \|e^{\a x}u\|_{L^p}\le C_1\|u\|_X^{\frac{3}{p}-\frac12}
(\|\pd_xu\|_X+\|\pd_x^{-1}\pd_yu\|_X+\|u\|_X)^{\frac{3}{2}-\frac{3}{p}}\,.
\end{equation}
\end{claim}
\begin{proof}
First, we remark 
\begin{equation}
  \label{eq:pn}
0<p_n'(x)\le 2\a p_n(x)\le 4\a e^{2\a x}\,,\quad
|p_n''(x)|\le 2\a p_n'(x)\,,\quad  |p_n'''(x)|\le 4\a^2p_n'(x)\,.
\end{equation}
Using \eqref{eq:pn}, we have \eqref{eq:aniso1} in the same way
as the proof of \cite[Lemma~2]{MST_KPI} and \cite[Claim~5.1]{MT}.
\par
Eq.~\eqref{eq:aniso2} is obvious if $p=2$.
For $p=6$, we have \eqref{eq:aniso2} with $p=6$
by passing the limit to $n\to\infty$ in \eqref{eq:aniso1}
because $p_n'(x)>0$ for every $x\in\R$ and $p_n'(x)$
is monotone increasing in $n$.
Thus we have \eqref{eq:aniso2} by interpolation.
\end{proof}

\begin{proof}[Proof of Lemma~\ref{lem:diff}]
First, we prove \eqref{eq:diff-utv1} assuming that $v_0\in H^3(\R^2)$
and $v_0\in \pd_x H^2(\R^2)$. Then it follows from \cite{Bourgain,MST}
that $\tv_1$, $w\in C(\R;H^3(\R^2))$ and
$\pd_x^{-1}\tv_1$, $\pd_x^{-1}w\in C(\R;H^2(\R^2))$.
Since $\mL_0\varphi=3\pd_x\varphi^2$ and $u$ and $\tv_1$ are
solutions of \eqref{eq:KPII},
\begin{equation}
  \label{eq:diff-uv1}
\left\{
  \begin{aligned}
& \pd_t w=\mL_0w-\pd_x\fN_1\,,\\
& w(0,x,y)=0\,,
  \end{aligned}\right.
\end{equation}
where $\fN_1=6\varphi(w+\bar{v}_1)+3w(w+2\bar{v}_1)$.
Multiplying \eqref{eq:diff-uv1} by $2p_n(x)w(t,x,y)$
and integrating the resulting equation by parts, we have
\begin{equation}
  \label{eq:weid}
  \begin{split}
& \frac{d}{dt}\int_{\R^2} p_n(x)w^2(t,x,y)\,dxdy +
\int_{\R^2}p_n'(x)\left\{\mathcal{E}(w)-4w^3\right\}(t,x,y)\,dxdy
\\ =& 6\int_{\R^2}\left\{p_n'(x)(\bar{v}_1(t,x,y)+\varphi(x))
-p_n(x)(\pd_x\bar{v}_1(t,x,y)+\varphi'(x))\right\}w(t,x,y)^2\,dxdy
\\ & -12\int_{\R^2}p_n(x)w(t,x,y)\pd_x(\varphi(x)\bar{v}_1(t,x,y))\,dxdy
+\int_{\R^2}p_n'''(x)w^2(t,x,y)\,dxdy\,,
  \end{split}
\end{equation}
where $\mathcal{E}(w)=3(\pd_zw)^2+3(\pd_z^{-1}\pd_yw)^2+4w^2$.
By Claim~\ref{cl:aniso},
  \begin{align*}
& \left|\int_{\R^2}p_n'(x)w^3(t,x,y)\,dxdy\right|
\\ \lesssim & \|w(t)\|_{L^2}
\left(\int_{\R^2}p_n'(x)w^2(t,x,y)\,dxdy\right)^{1/4}
\left(\int_{\R^2} p_n'(x)\mathcal{E}(w)(t,x,y)\,dxdy\right)^{3/4}
\,,  \end{align*}  
and it follows from \eqref{eq:pn} and the above that there exist
positive constants $\nu$ and $C_1$ such that for any $n\in\N$,
$T\ge0$ and $t\in[0,T]$,
\begin{align*}
&  \int_{\R^2} p_n(x)w^2(t,x,y)\,dxdy +\nu \int_0^t \int_{\R^2} p_n'(x) \mathcal{E}(w)(s,x,y)\,dxdyds
\\ \le & C_1T\sup_{t\in[0,T]}\|\bar{v}_1(t)\|_{H^1}^2 
\\  &
+C_1\sup_{t\in[0,T]}(1+\|\bar{v}_1(t)\|_{H^3}+\|w(t)\|_{L^2}^4)
\int_0^t\int_{\R^2} p_n(x)w^2(s,x,y)\,dxdyds\,.
\end{align*}
By Gronwall's inequality, we have for $t\in[0,T]$,
$$\int_{\R^2}p_n(x)w^2(t,x,y)\,dxdy\le C_2\sup_{t\in[0,T]}\|\bar{v}_1(t)\|_{H^1}^2\,,$$
where $C_2$ is a constant independent of $n$.
By passing the limit to $n\to\infty$, we have
\begin{equation*}
\|w(t)\|_X^2 \le   C_2\sup_{t\in[0,T]}\|\bar{v}_1(t)\|_{H^1}^2
\quad\text{for $t\in[0,T]$.}
\end{equation*}
since $0<p_n(x)\uparrow 2e^{2\a x}$ as $n\to\infty$.
Thus we prove $w\in L^\infty(0,T;X)$ and
$\pd_xw$, $\pd_x^{-1}\pd_y\in L^2(0,T;X)$ for every $T\ge0$
provided $v_0\in H^3(\R^2)\cap \pd_xH^2(\R^2)$.
\par
Let $p(x)=e^{2\a x}$. Integrating by parts the second and the third terms
of the right hand side of \eqref{eq:weid}, integrating the resulting
over $[0,t]$ and passing the limit to $n\to\infty$, we have 
\begin{align*}
& \int_{\R^2} p(x)w^2(t,x,y)\,dxdy +
\int_0^t\int_{\R^2}p'(x)\left\{\mathcal{E}(w)-4w^3\right\}(s,x,y)\,dxdyds
\\ =& 12\int_{\R^2}(\bar{v}_1(s,x,y)+\varphi(x))
\left\{p'(x)w^2(s,x,y)+p(x)(w\pd_xw)(s,x,y)\right\}\,dxdy
\\ & +12\int_0^t\int_{\R^2}\pd_x\{p(x)w(s,x,y)\}\varphi(x)\bar{v}_1(s,x,y)\,dxdyds
\\ & +\int_0^t\int_{\R^2}p'''(x)w^2(s,x,y)\,dxdyds\,.
 \end{align*}
By the H\"older inequality and Claim~\ref{cl:aniso},
\begin{align*}
& \left|\int_{\R^2}
\left\{p'(x)w^2(s,x,y)+p(x)(w\pd_xw)(s,x,y)\right\}
(\bar{v}_1(s,x,y)+\varphi(x))\,dxdy\right|
\\ \lesssim & (\|\pd_xw(s)\|_X+\|w(s)\|_X)\|e^{\a x }w(s)\|_{L^4}
\|\tv_1(s)+\varphi\|_{L^4}
\\ \lesssim & 
\|\tv_1(s)+\varphi\|_{L^4}\|w(s)\|_X^{\frac14}
\|\mathcal{E}(w(s))^{1/2}\|_X^{\frac{7}{4}} \,,
\end{align*}
\begin{align*}
\left|\int_{\R^2}\pd_x\{p(x)w(s,x,y)\}\varphi(x)\bar{v}_1(s,x,y)\,dxdy\right|
 \lesssim & \|\bar{v}_1(s)\|_{L^2}\|\mathcal{E}(w(s))^{1/2}\|_X\,,
\end{align*}
and
\begin{equation*}
\left|\int_{\R^2}p(x)w^3(s,x,y)\,dxdy\right|
\lesssim \|w(s)\|_{L^2}\|w(s)\|_X^{1/2}\|\mathcal{E}(w(s))^{1/2}\|_X^{3/2}\,.
\end{equation*}
Combining the above, we have for $t\in[0,T]$,
  \begin{equation}
\label{eq:winX2}
    \begin{split}
& \|w(t)\|_X^2+\nu\int_0^t\|\mathcal{E}(w(s))^{1/2}\|_X^2\,ds
\\ \lesssim & 
\int_0^t \left\{\|\bar{v}_1(s)\|_{L^2}^2+
(\|\bar{v}_1(s)\|_{L^2}^2+\|w(s)\|_{L^2}^4+\|\varphi+\bar{v}_1(s)\|_{L^4}^8)
\|w(s)\|_X^2\right\}\,ds\,.     
    \end{split}
  \end{equation}
where $\nu$ is a positive constant independent of $T$.
Since $\|\bar{v}_1(t)\|_{L^2}=\|v_0\|_{L^2}$ for every $t\in\R$ and $H^{1/2}(\R^2)\subset L^4(\R^2)$,
it follows from the Gronwall's inequality that
\begin{equation}
  \label{eq:winX}
\|w(t)\|_X^2 \le   C_3Te^{C_4t}\|v_0\|_{L^2}^2
\quad\text{for $t\in[0,T]$,}
\end{equation}
where $C_3$ and $C_4$ are positive constants depending only on
$\|v_1(t)\|_{H^{1/2}}$ and $\|w(t)\|_{L^2}$.
By a standard limiting argument, we have \eqref{eq:winX} and
\eqref{eq:winX3} for every $v_0\in H^{1/2}(\R^2)$.
\par

Next, we will show that $w\in C([0,\infty);X)$.
By Claim~\ref{cl:aniso}, \eqref{eq:winX} and \eqref{eq:winX3}
that
$$\|e^{ax}w\|_{L^4}\lesssim \|w\|_X^{1/4}\|\mathcal{E}(w)^{1/2}\|_X^{3/4}
\in L^{8/3}(0,T;X)\,,$$
\begin{align*}
\|\fN_1\|_X\lesssim & \|w\|_X+\|\bar{v}_1\|_{L^2}
+(\|w\|_{L^4}+\|\bar{v}_1\|_{L^4})\|e^{ax}w\|_{L^4}\in L^{8/3}(0,T;X)\,.
\end{align*}
By the variation of constants formula,
\begin{equation}
  \label{eq:varw}
w(t)=-\int_0^t e^{(t-s)\mL_0}\pd_x\fN_1\,.
\end{equation}
By Lemma~\ref{lem:free-semigroup}, \eqref{eq:varw} and the fact that
$\fN_1\in L^{8/3}(0,T;X)$, we have for $h>0$,
\begin{align*}
\|w(t+h)-w(t)\|_X \le 
\left\|(e^{h\mL_0}-I)\int_0^te^{(t-s)\mL_0}\pd_x\fN_1(s)\,ds\right\|_X
+O(h^{1/8})\,.
\end{align*}
Since $e^{t\mL_0}$ is a $C^0$-semigroup on $X$, 
it follows that $w\in C([0,\infty);X)$.
\par
Finally, we will show \eqref{eq:pd-1bu}.
Let $\bar{u}(t,x,y):=u(t,x+4t,y)-\varphi(x)$.
Then by the variation of constants formula,
\begin{equation}
  \label{eq:baru-int}
\bar{u}(t)=e^{t\mL_0}v_0-3\pd_x\int_0^te^{(t-s)\mL_0}
\left(2\varphi \bar{u}(s)+\bar{u}^2(s)\right)\,ds\,.
\end{equation}
Since $e^{t\mL_0}$ is unitary on $L^2(\R^2)$, $\pd_x^{-1}v_0\in L^2(\R^2)$
and $\bar{u}(t)\in C(\R;H^{1/2}(\R^2))$,
we easily see that \eqref{eq:pd-1bu} follows from \eqref{eq:baru-int}.
Thus we complete the proof.
\end{proof}

Next, we will show the continuity of 
 $H^{1/2}(\R^2)\ni v_0 \mapsto u-\tv_1-\varphi(x-4t)\in X$.
 \begin{lemma}
\label{lem:diff-conti}
Let $v_0\in H^{1/2}(\R^2)$ and $v_{0,n}\in H^{1/2}(\R^2)$ for
$n\in\N$.  Suppose $\tv_1$, $\tv_{1,n}$, $u$ and $u_n$
   be solutions of \eqref{eq:KPII} satisfying
$\tv_1(0,x,y)=v_0(x,y)$, $\tv_{1,n}(0,x,y)=v_{0,n}(x,y)$, 
$u(0,x,y)=\varphi(x)+v_0(x,y)$ and  $u_n(0,x,y)=\varphi(x)+v_{0,n}(x,y)$.
If $\lim_{n\to\infty}\|v_{0,n}-v_0\|_{H^{1/2}(\R^2)}=0$, then for any $T\in(0,\infty)$,
\begin{equation*}
\lim_{n\to\infty}\sup_{t\in[0,T]}\|u(t)-\tv_1(t)-u_n(t)+\tv_{1,n}(t)\|_X=0\,.
\end{equation*}
 \end{lemma}
 \begin{proof}
Let $\bar{v}_{1,n}(t,x,y)=\tv_{1,n}(t,x+4t,y)$,
$w_n(t,x,y)=u_n(t,x+4t,y)-\varphi(x)-\bar{v}_{1,n}(t,x,y)$ and
$\tilde{w}_n=w-w_n$.
Then
\begin{equation}
  \label{eq:twn}
\left\{
  \begin{aligned}
& \pd_t\tilde{w}_n=\mL_0\tilde{w}_n-\pd_x(\fN_2+\fN_3)\,,\\
& \tilde{w}_n(0,x,y)=0\,,
  \end{aligned}\right.
\end{equation}
 where
$$\fN_2(t)=3(2\varphi+2\bar{v}_{1,n}(t)+w(t)+w_n(t))\tilde{w}_n(t)\,,
\quad \fN_3(t)=6(\varphi+w(t))(\bar{v}_{1,n}(t)-\bar{v}_1(t))\,.$$
Multiplying \eqref{eq:twn} by $2e^{2\a x}\tilde{w}_n$ and integrating the
resulting equation over $\R^2\times[0,t]$, we have
\begin{equation}
\label{eq:twn-est}
\begin{split}
& \|\tilde{w}_n(t)\|_X^2
+2\a\int_0^t\|\mathcal{E}(\tilde{w}_n(s))^{1/2}\|_X^2\,ds
\\ = & -2\int_0^t\int_{\R^2} e^{2\a x}
\tilde{w}_n(s)\pd_x(\fN_2(s)+\fN_3(s))\,dxdyds\,.
\end{split}
\end{equation}
Using Claim~\ref{cl:aniso} and the fact that $L^4(\R^2)\subset H^{1/2}(\R^2)$
\begin{align*}
& \left|\int_{\R^2} e^{2\a x}
\tilde{w}_n\pd_x\fN_2\,dxdy\right|
\\ \lesssim  & \|e^{ax}\tilde{w}_n\|_{L^4}
(\|\pd_x\tilde{w}_n\|_X+\|\tilde{w}_n\|_X)
(1+\|\bar{v}_{1,n}\|_{H^{1/2}}+\|w+w_n\|_{H^{1/2}})
\\ \lesssim & 
(1+\|\bar{v}_{1,n}\|_{H^{1/2}}+\|w+w_n\|_{H^{1/2}})
\|\tilde{w}_n\|_X^{1/4}\|\mathcal{E}(\tilde{w}_n)^{1/2}\|_X^{7/4}\,,
\end{align*}
\begin{align*}
 \left|\int_{\R^2} e^{2\a x}
\tilde{w}_n\pd_x\fN_3\,dxdy\right|
\lesssim & (1+\|e^{ax}w\|_{L^4})\|\bar{v}_{1,n}-\bar{v}_1\|_{L^4}
(\|\tilde{w}_n\|_X+\|\pd_x\tilde{w}_n\|_X)
\\ \lesssim &
(1+\|\mathcal{E}(w)^{1/2}\|_X)
\|\bar{v}_{1,n}-\bar{v}_1\|_{H^{1/2}}\|\mathcal{E}(\tilde{w}_n)^{1/2}\|_X\,.
\end{align*}
Combining the above with \eqref{eq:twn-est}, we have
\begin{equation}
  \label{eq:twn-est2}
  \begin{split}
& \|\tilde{w}_n(t)\|_X^2\lesssim 
(T+\|\mathcal{E}(w_n)^{1/2}\|_{L^2(0,T;X)})
\sup_{t\in[0,T]}\|v_{1,n}(t)-v_1(t)\|_{H^{1/2}}^2
\\ & +
\sup_{t\in[0,T]}(1+\|\bar{v}_{1,n}(t)\|_{H^{1/2}}+\|w_n(t)+w(t)\|_{H^{1/2}})^8
\int_0^t \|\tilde{w}_n(s)\|_X^2\,ds\,.
  \end{split}
\end{equation}
Thanks to the wellposedness of \eqref{eq:KPII} (e.g. \cite{Bourgain,MST}),
$$\lim_{n\to\infty}\sup_{t\in[0,T]}\|v_{1,n}(t)-v_1(t)\|_{H^{1/2}}=0\,,\quad
\lim_{n\to\infty}\sup_{t\in[0,T]}\|\tilde{w}_n(t)\|_{H^{1/2}}=0\,.$$
Thus by  \eqref{eq:winX}, \eqref{eq:winX3} and \eqref{eq:twn-est2},
we have for $t\in[0,T]$,
\begin{equation}
  \label{eq:twn-est3}
\|\tilde{w}_n(t)\|_X^2\le
C_1\sup_{t\in[0,T]}\|v_{1,n}(t)-v_1(t)\|_{H^{1/2}}^2
+C_2\int_0^t \|\tilde{w}_n(s)\|_X^2\,ds,,
\end{equation}
where $C_1$ and $C_2$ are positive constants independent of $n$.
Applying Gronwall's inequality to \eqref{eq:twn-est2}, we obtain
Lemma~\ref{lem:diff-conti}. Thus we complete the proof.
 \end{proof}

To fix the decomposition \eqref{eq:decomp}, we impose that 
$v_2(t,z,y)$ is symplectically orthogonal to low frequency resonant modes.
More precisely, we impose the constraint that for $k=1$, $2$,
\begin{equation}
  \label{eq:orth}
\lim_{M\to\infty}\int_{-M}^M \int_\R
v_2(t,z,y)\overline{g_k^*(z,\eta,c(t,y))}e^{-iy\eta}\,dzdy=0
\quad\text{in $L^2(-\eta_0,\eta_0)$,}
\end{equation}
where 
$g_1^*(z,\eta,c)=cg_1^*(\sqrt{c/2}z,\eta)$ and
$g_2^*(z,\eta,c)=\frac{c}{2}g_2^*(\sqrt{c/2}z,\eta)$.
\par
We will show that the decomposition \eqref{eq:decomp} with
\eqref{eq:decomp2} and \eqref{eq:orth} is well defined as long as
$v_2$ remains small in the exponentially weighted space $X$.
\par

Next, we introduce functionals to prove the existence of
the representation \eqref{eq:decomp}, \eqref{eq:decomp2} that satisfies 
the orthogonality condition \eqref{eq:orth}.
\par
Now let us introduce the subspaces of $L^2(\R)$ to analyze modulation
parameters $c(t,y)$ and $x(t,y))$.
For an $\eta_0>0$, let $Y$ and $Z$ be closed subspaces of $L^2(\R)$ defined by
$$Y=\mF^{-1}_\eta Z\,,\quad
Z=\{f\in L^2(\R)\mid\supp f \subset[-\eta_0,\eta_0]\}\,.$$
Let $Y_1=\mF^{-1}_\eta Z_1$ and
$Z_1=\{f\in Z \mid\|f\|_{Z_1}:=\|f\|_{L^\infty}<\infty\}$.
\begin{remark}
  \label{rem:smoothness}
We have
\begin{equation}
  \label{eq:H^s-Y}
\|f\|_{\dot{H}^s}\le \eta_0^s\|f\|_{L^2}
\quad\text{for any $s\ge0$ and $f\in Y$,}  
\end{equation}
since $\hat{f}$ is $0$ outside of $[-\eta_0,\eta_0]$.
We have $\|f\|_{L^\infty}\lesssim \|f\|_{L^2}$ for any $f\in Y$.
\par
Let $\wP_1$ be a projection defined by
$\wP_1 f=\mF_\eta^{-1}\mathbf{1}_{[-\eta_0,\eta_0]}\mF_yf$,
where $\mathbf{1}_{[-\eta_0,\eta_0]}(\eta)=1$ for $\eta\in[-\eta_0,\eta_0]$
and $\mathbf{1}_{[-\eta_0,\eta_0]}(\eta)=0$ for $\eta\not\in[-\eta_0,\eta_0]$.
Then $\|\wP_1f\|_{Y_1}\le (2\pi)^{-1/2}\|f\|_{L^1}(\R)$ for any $f\in L^1(\R)$.
In particular, for any $f$, $g\in Y$,
\begin{equation}
  \label{eq:Y1-L1}
\|\wP_1(fg)\|_{Y_1}\le (2\pi)^{-1/2}\|fg\|_{L^1}\le (2\pi)^{-1/2}\|f\|_Y\|g\|_Y\,.
\end{equation}
\end{remark}

For $\tu\in X$ and $\gamma$, $\tc\in Y$ and $L\ge 0$, let 
$c(y)=2+\tc(y)$ and 
\begin{align*}
F_k[\tu,\tc,\gamma,L](\eta):=\mathbf{1}_{[-\eta_0,\eta_0]}&(\eta)
\lim_{M\to\infty}\int_{-M}^M\int_\R
\bigl\{\tu(x,y)+\varphi(x)- \varphi_{c(y)}(x-\gamma(y))
\\ & +\psi_{c(y),L}(x-\gamma(y))\bigr\}
\overline{g_k^*(x-\gamma(y),\eta,c(y))}e^{-iy\eta}\,dxdy\,.
\end{align*}
The mapping $F=(F_1,F_2)$ maps 
$X\times Y\times Y\times \R$ into $Z\times Z$.
\begin{lemma} \emph{(\cite[Lemma~5.1]{Mi})}
\label{lem:F_k}
Let $\a\in(0,2)$,  $\tu\in X$, $\tc$, $\gamma\in Y$ and $L\ge 0$.
Then there exists a $\delta>0$ such that if 
$\|\tc\|_Y+\|\gamma\|_Y\le \delta$, then
$F_k[\tu,\tc,\gamma,L]\in Z$ for $k=1$, $2$.
\end{lemma}

\begin{lemma} \emph{(\cite[Lemma~5.2]{Mi})}
  \label{lem:decomp}
Let $\a\in(0,2)$.
There exist positive constants $\delta_0$, $\delta_1$ and $L_0$ such that
if $\|\tu\|_X<\delta_0$ and $L\ge L_0$, then there exists a unique
$(\tc,\gamma)$ with $c=2+\tc$ satisfying
\begin{gather}
\label{eq:imp1}
  \|\tc\|_{Y}+\|\gamma\|_{Y} <\delta_1\,,\\
\label{eq:imp2}
F_1[\tu,\tc,\gamma,L]=F_2[\tu,\tc,\gamma,L]=0\,.
\end{gather}
Moreover, the mapping $\{\tu\in X\mid \|u\|_X<\delta_0\}
\ni \tu\mapsto (\tc,\gamma)=:\Phi(\tu)$ is $C^1$.
\end{lemma}

\begin{remark}
\label{rem:decomp}
Let $u$ be a solution of \eqref{eq:KPII} satisfying
$u(0,x,y)=\varphi(x)+v_0(x,y)$ and let $\tv_1$ be a solution of \eqref{eq:tv1}.
Suppose $v_0\in H^{1/2}(\R^2)$.
Since $\tv\in C([0,T);X)$ by Lemma~\ref{lem:diff} and $\|\tv(0)\|_X$ is small, 
we see from Lemma~\ref{lem:decomp} that there exists a $T>0$ such that
$$(v_2,\tc,\tx)\in C([0,T];X\times Y\times Y)\,.$$
Moreover, replacing $u$ in \cite[Remark~5.3]{Mi} by
$\tu=u-\tv_1$ and using Lemma~\ref{lem:diff},
we can see that there exists a $T>0$ such that 
$$(\tc(t),\tx(t))=\Phi(\tv(t))\in C([0,T];Y\times Y)
\cap C^1((0,T);Y\times Y)\,,$$
where $\tv(t,x,y)=\tu(t,x+4t,y)-\varphi(x)$.
Moreover, we have $v_2\in C([0,T];X)$ 
and $(\tv(0),\tc(0),\tx(0))=(0,0,0)$.
\end{remark}
\begin{remark}
  \label{rem:decomp-continuity}
Let $u$, $\tv_1$, $\tc$ and $\tx$  be as in Remark~\ref{rem:decomp}
and let $u_n$ and $\tv_{1,n}$ be as in Lemma~\ref{lem:diff-conti}.
By Lemmas~\ref{lem:diff} and \ref{lem:diff-conti},
\begin{gather*}
\tv_n(t,x,y):=u_n(t,x+4t,y)-\tv_{1,n}(t,x+4t,y)-\varphi(x)
\in C([0,\infty);X)\,,\\  
\lim_{n\to\infty}\|\tv_n(t)-\tv(t)\|_X=0\,,
\end{gather*}
and it follows from Lemma~\ref{lem:decomp} that
there exists a $T>0$ such that 
\begin{gather*}
(\tc_n(t),\tx_n(t)):=\Phi(\tv_n(t))\in C([0,T];Y\times Y)
\cap C^1((0,T);Y\times Y)\,,\\
\lim_{n\to\infty}\sup_{t\in[0,T]}
\left(\|\tc_n(t)-\tc(t)\|_Y+\|\tx_n(t)-\tx(t)\|_Y\right)=0\,.
\end{gather*}
Following the argument of \cite[Remark~5.3]{Mi}, we also have
$$\lim_{n\to\infty}\sup_{t\in[0,T]}
\left(\|\pd_t\tc_n(t)-\pd_t\tc(t)\|_Y
+\|\pd_t\tx_n(t)-\pd_t\tx(t)\|_Y\right)=0\,.$$
\end{remark}
\par
We use a continuation principle that ensures the existence of
\eqref{eq:decomp} as long as $\|v_2(t)\|_X$ and $\|\tc(t)\|_Y$ remain small.

\begin{proposition} 
\label{prop:continuation}
Let $\a\in(0,1)$ and let $\delta_0$ and $L$ be the same as in Lemma~\ref{lem:decomp} and
let $u(t)$ and $\tv_1(t)$ be as in Lemma~\ref{lem:diff}.
Then there exists a constant $\delta_2>0$
such that if \eqref{eq:decomp}, \eqref{eq:decomp2} and \eqref{eq:orth}
hold for $t\in[0,T)$
and $v_2(t,z,y)$, $\tc(t,y):=c(t,y)-2$ and $\tx(t,y):=x(t,y)-4t$ satisfy
\begin{gather}
\label{eq:C1cx}
(\tc,\tx)\in C([0,T);Y\times Y)\cap C^1((0,T);Y\times Y)\,,\\
\label{eq:bd-vc}
\sup_{t\in[0,T)}\|v_2(t)\|_X\le \frac{\delta_0}{2}\,,
\quad \sup_{t\in[0,T)}\|\tc(t)\|_Y<\delta_2\,,\quad
\sup_{t\in[0,T)}\|\tx(t)\|_Y<\infty\,,
\end{gather}
then either $T=\infty$ or $T$ is not the maximal time of
the decomposition \eqref{eq:decomp} satisfying \eqref{eq:orth},
\eqref{eq:C1cx} and \eqref{eq:bd-vc}.
\end{proposition}
\begin{proof}
Since
$u(t,x,y)-\varphi(x-4t)-\tv_1(t,x,y)\in C([0,\infty);X)$ by
Lemma~\ref{lem:diff}, 
we can prove Proposition~\ref{prop:continuation} in the same way as
\cite[Proposition~5.5]{Mi}.
\end{proof}
\bigskip

\section{Modulation equations}
\label{sec:modulation}
In this section, we will derive a system of PDEs which describe the motion
of  modulation parameters $c(t,y)$ and $x(t,y)$.
Substituting $\tv_1(t,x,y)=v_1(t,z,y)$ with $z=x-x(t,y)$ into \eqref{eq:KPII},
we have
\begin{equation}
  \label{eq:v1}
\pd_tv_1-2c\pd_zv_1+\pd_z^3v_1+3\pd_{z}^{-1}\pd_y^2v_1=\pd_z(N_{1,1}+N_{1,2})+N_{1,3}\,,
\end{equation}
where 
$N_{1,1}=-3v_1^2$, $N_{1,2}=\{x_t-2c-3(x_y)^2\}v_1$
and $N_{1,3}=6\pd_y(x_yv_1)-3x_{yy}v_1$. 
Substituting the ansatz \eqref{eq:decomp} into \eqref{eq:KPII},
we have 
\begin{equation}
  \label{eq:v}
\pd_tv=\mL_{c}v+\ell+\pd_z(N_1+N_2)+N_3\,,  
\end{equation}
where  
$\mL_cv=-\pd_z(\pd_z^2-2c+6\varphi_c)v-3\pd_z^{-1}\pd_y^2$,
$\ell=\ell_1+\ell_2$, $\ell_k=\ell_{k1}+\ell_{k2}+\ell_{k3}$
$(k=1, 2)$, $\tpsi_c(z)=\psi_{c,L}(z+3t)$ and
\begin{align*}
\ell_{11}=&(x_t-2c-3(x_y)^2)\varphi'_c-(c_t-6c_yx_y)\pd_c\varphi_c\,,\quad
\ell_{12}=3x_{yy}\varphi_c\,,\\
\ell_{13}=& 3c_{yy}\int_z^\infty\pd_c\varphi_c(z_1)dz_1
+3(c_y)^2\int_z^\infty \pd_c^2\varphi_{c}(z_1)dz_1\,,\\
\ell_{21}=& (c_t-6c_yx_y)\pd_c\tpsi_c-(x_t-4-3(x_y)^2)\tpsi_c'\,,\\
\ell_{22}=&(\pd_z^3-\pd_z)\tpsi_c-3\pd_z(\tpsi_c^2)
+6\pd_z(\varphi_c\tpsi_c)-3x_{yy}\tpsi_c\,,\\
\ell_{23}=&-3c_{yy}\int_z^\infty\pd_c\tpsi_c(z_1)dz_1
-3(c_y)^2\int_z^\infty \pd_c^2\tpsi_c(z_1)dz_1\,,
\end{align*}
\begin{align*}
& N_1=-3v^2\,,\quad N_2=\{x_t-2c-3(x_y)^2\}v+6\tpsi_cv\,,\\
& N_3=6x_y\pd_yv+3x_{yy}v=6\pd_y(x_yv)-3x_{yy}v\,.
\end{align*}
Here we use the fact that $\varphi_c$ is a solution of
\begin{equation}
  \label{eq:B}
  \varphi_c''-2c\varphi_c+3\varphi_c^2=0\,.
\end{equation}
We slightly change the definition of $\tpsi$ from \cite{Mi}
in order to apply the virial identity to
$\int_{\R^2}\tpsi_c(z) v_1^2(t,z,y)\,dzdy$.
\par
Subtracting \eqref{eq:v1} from \eqref{eq:v}, we have
\begin{equation}
  \label{eq:v2}
\pd_tv_2=\mL_{c}v_2+\ell+\pd_z(N_{2,1}+N_{2,2}+N_{2,4})+N_{2,3}\,,
\end{equation}
where
\begin{gather*}
N_{2,1}=-3(2v_1v_2+v_2^2)\,,\quad  N_{2,2}=\{x_t-2c-3(x_y)^2\}v_2+6\tpsi_cv_2\,,
\\ N_{2,3}=6\pd_y(x_yv_2)-3x_{yy}v_2\,,\quad N_{2,4}=6(\tpsi_c-\varphi_c)v_1\,.
\end{gather*}
\par

Let
\begin{align*}
& \bM_{c,x}(T)=\sup_{[0,T]}(\|\tc(t)\|_Y+\|x_y(t)\|_Y)
+\|c_y\|_{L^2(0,T;Y)}+\|x_{yy}\|_{L^2(0,T;Y)}\,,\\
& \bM_1(T)=\sup_{t\in[0,T]}\|v_1(t)\|_{L^2}
+\|\mathcal{E}(v_1)^{1/2}\|_{L^2(0,T;W(t))}\,,
\quad \bM_1'(T)=\sup_{t\in[0,T]}\|\tv_1(t)\|_{L^3}\,,
\\ & 
\bM_2(T)=\sup_{0\le t\le T}\|v_2(t)\|_X+\|\mathcal{E}(v_2)^{1/2}\|_{L^2(0,T;X)}\,,
\quad \bM_v(T)=\sup_{t\in[0,T]}\|v(t)\|_{L^2}\,,
\end{align*}
where
$\|v\|_{W(t)}=\|(e^{-\a|z|/2}+e^{-\a|z+3t+L|})v\|_{L^2(\R^2)}$,
$L$ is a large positive constant and
$$\pd_z^{-1}\pd_yv(t,z,y):=\mF_{\xi,\eta}^{-1}\left(\frac{\eta}{\xi}
\mF_{z,y}v(t,\xi,\eta)\right)\,.$$
By Lemma~\ref{lem:diff}, we have $v_2(t)\in X$ and
$$\pd_x^{-1}v_2(t,z,y)=-\int_z^\infty v_2(t,z_1,y)\,dz_1\in X$$
if $x(t,\cdot)\in L^\infty(\R)$.
\par

Now we will derive modulation equations of $c(t,y)$ and $x(t,y)$
from the orthogonality condition \eqref{eq:orth} assuming
the smallness of $\bM_{c,x}(T)$, $\bM_1(T)$ and $\bM_2(T)$.
It follows from \cite{MST} and \cite[Lemma~3.2]{KZ} that $\tv_1(t)$,
$\tv(t)\in C(\R;L^2(\R^2))$ and $\pd_x^{^1}\pd_y\tv_1$, 
$\pd_x^{-1}\pd_y\tv\in L_x^\infty L^2([-T,T]\times\R_y)$ for any $T>0$.
Moreover, Lemma~\ref{lem:diff} implies that $\tv(t)\in C([0,\infty);X)$
and $\pd_x^{-1}\pd_y\tv \in L^2(0,T;X)$.
If $\bM_{c,x}(T)$ and $\bM_2(T)$ are sufficiently small,
then we see from Remark~\ref{rem:decomp} and
Proposition~\ref{prop:continuation} that
the decomposition \eqref{eq:decomp} satisfying \eqref{eq:orth}
and \eqref{eq:C1cx} exists for $t\in[0,T]$.
Since $Y\subset \cap_{s\ge0}H^s(\R)$, we have
\begin{equation}
  \label{eq:diff-tv}
  \begin{split}
& v_2(t,z,y)-\tv(t,z+\tx(t,y),y)
\\=& \varphi(z+\tx(t,y))-\varphi_{c(t,y)}(z)
+\tpsi_{c(t,y)}(z)\in L^2(\R^2)\cap X\,,      
  \end{split}
\end{equation}
and we easily see that $v_2(t)\in C([0,T];X\cap L^2(\R^2))$.
Moreover, since
$$\int_\R\left\{\varphi(z+\tx(t,y))
-\varphi_{c(t,y)}(z)+\tpsi_{c(t,y)}(z)\right\}\,dz=0$$
for any $y\in\R$ by \eqref{eq:0mean}
and its integrand decays exponentially as $z\to\pm\infty$,
we have 
$$(\pd_z^{-1}\pd_yv_2)(t,z,y)\in L^2(0,T;X)
\cap L_x^\infty L^2([-T,T]\times \R_y)\,.$$
Approximating $g_k^*(z,\eta)$ by $C_0^4(\R)$-functions
in $L^2(\R;e^{-2\a z}dz)$ and using 
Proposition~\ref{prop:continuation} and Remark~\ref{rem:decomp},
we can justify the mapping
$$ t\mapsto \int_{\R^2}v_2(t,z,y)\overline{g_k^*(z,\eta,c(t,y))}
e^{-iy\eta}\,dzdy\in Z$$
is $C^1$ for $t\in[0,T]$ if we have \eqref{eq:C1cx} and \eqref{eq:bd-vc}.
Differentiating \eqref{eq:orth} with respect to $t$ and substituting
\eqref{eq:v2} into the resulting equation, we have in $L^2(-\eta_0,\eta_0)$
\begin{equation}
\label{eq:orth_t}
\begin{split}
& \frac{d}{dt}\int_{\R^2}v_2(t,z,y)\overline{g_k^*(z,\eta,c(t,y))}
e^{-iy\eta}\,dzdy
\\=& \int_{\R^2} \ell\overline{g_k^*(z,\eta,c(t,y))}e^{-iy\eta}dzdy
+\sum_{j=1}^6II^j_k(t,\eta)=0\,,
\end{split}  
\end{equation}
where
\begin{align*}
II^1_k=& \int_{\R^2} v_2(t,z,y)\mL_{c(t,y)}^*(\overline{g_k^*(t,z,c(t,y))
e^{iy\eta}})\, dzdy\,,\\
II^2_k=& -\int_{\R^2} N_{2,1}\overline{\pd_zg_k^*(z,\eta,c(t,y))}
e^{-iy\eta}\, dzdy\,,\\
II^3_k=& \int_{\R^2} N_{2,3}\overline{g_k^*(z,\eta,c(t,y))}e^{-iy\eta}dzdy
\\ & +6\int_{\R^2} v_2(t,z,y)c_y(t,y)x_y(t,y)
\overline{\pd_cg_k^*(z,\eta,c(t,y))}e^{-iy\eta}\, dzdy\,,\\
II^4_k=& \int_{\R^2} v_2(t,z,y)\left(c_t-6c_yx_y\right)(t,y)
\overline{\pd_cg_k^*(z,\eta,c(t,y))}e^{-iy\eta}\, dzdy\,,\\
II^5_k=& -\int_{\R^2} N_{2,2}\overline{\pd_zg_k^*(z,\eta,c(t,y))}e^{-iy\eta}\, dzdy\,,
\\
II^6_k=& -\int_{\R^2} N_{2,4}\overline{\pd_zg_k^*(z,\eta,c(t,y))}e^{-iy\eta}\, dzdy\,.
\end{align*}
\par

The modulation PDEs of $c(t,y)$ and $x(t,y)$ can be obtained by
computing the inverse Fourier transform of \eqref{eq:orth_t} in $\eta$.
The leading term of
$$\frac{1}{2\pi}
\int_{-\eta_0}^{\eta_0}
\int_{\R^2}\ell_1\overline{g_k^*(z,\eta,c(t,y_1))}e^{i\eta(y-y_1)}\,dzdy_1d\eta$$
is
\begin{equation}
  \label{eq:Gk}
G_k(t,y)=\int_\R \ell_1 \overline{g_k^*(z,0,c(t,y))}dz\,.
\end{equation}
Since $g_1^*(z,0,c)=\varphi_c(z)$ and
$g_2^*(z,0,c)=(c/2)^{3/2}\int_{-\infty}^z\pd_c\varphi_c$, we can compute $G_1$ and $G_2$
explicitly.
\begin{lemma} \emph{(\cite[Lemma~6.1]{Mi})}
  \label{LEM:G}
Let $\mu_1=\frac{1}{2}-\frac{\pi^2}{12}$ and
$\mu_2=\frac{\pi^2}{32}-\frac{3}{16}$. Then
  \begin{align*}
G_1=&16x_{yy}\left(\frac{c}{2}\right)^{3/2}
-2(c_t-6c_yx_y)\left(\frac{c}{2}\right)^{1/2}+6c_{yy}-\frac{3}{c}(c_y)^2\,,\\
G_2=& -2(x_t-2c-3(x_y)^2)\left(\frac{c}{2}\right)^{2}
+6x_{yy}\left(\frac{c}{2}\right)^{3/2}
-\frac12(c_t-6c_yx_y)\left(\frac{c}{2}\right)^{1/2}
\\ & +\mu_1c_{yy}+\mu_2(c_y)^2\left(\frac{c}{2}\right)^{-1}\,.
  \end{align*}
\end{lemma}
We remark that $(G_1,G_2)$ are the dominant part of the modulation
equations for $c$ and $x$. Now we will write the remainder part of
$\int_{\R^2}\ell_1\overline{g_k^*(z,\eta,c(t,y))}e^{-iy\eta}\,dzdy$ in the same way
as \cite{Mi}. 
For $q_c=\varphi_c$, $\varphi_c'$, $\pd_c\varphi_c$
and $\pd_z^{-1}\pd_c^m\varphi_c(z)=-\int_z^\infty \pd_c^m\varphi_c(z_1)\,dz_1$
($m\ge1$), let $S_k^1[q_c]$  and $S_k^2[q_c]$ be operators defined by
\begin{align*}
& S_k^1[q_c](f)(t,y)=\frac{1}{2\pi} \int_{-\eta_0}^{\eta_0}\int_{\R^2}
f(y_1)q_2(z)\overline{g_{k1}^*(z,\eta,2)}e^{i(y-y_1)\eta}dy_1dzd\eta\,,\\
& S_k^2[q_c](f)(t,y)=\frac{1}{2\pi} \int_{-\eta_0}^{\eta_0}\int_{\R^2}
f(y_1)\tc(t,y_1)\overline{g_{k2}^*(z,\eta,c(t,y_1))}
e^{i(y-y_1)\eta}dy_1dzd\eta\,,
\end{align*}
where 
\begin{align*}
& g_{k1}^*(z,\eta,c)=\frac{g_k^*(z,\eta,c)-g_k^*(z,0,c)}{\eta^2}\,,\quad
\delta q_c(z)=\frac{q_c(z)-q_2(z)}{c-2}\,,\\
& g_{k2}^*(z,\eta,c)=g_{k1}^*(z,\eta,2)\delta q_c(z)+
\frac{g_{k1}^*(z,\eta,c)-g_{k1}^*(z,\eta,2)}{c-2}q_c(z)\,.
\end{align*}
Note that $S^1_k\in B(Y)$ and $S^1_k$ are independent of $c(t,y)$
whereas  $\|S^2_k\|_{B(Y,Y_1)}\lesssim \|\tc\|_Y$.
See \cite[Claims~B.1 and B.2]{Mi}.
Using $S^j_k$ ($j$, $k=1$, $2$), we have
\begin{equation}
  \label{eq:der2}
\begin{split}
& \frac{1}{2\pi}\int_{-\eta_0}^{\eta_0}\int_{\R^2}
 \ell_1\left(\overline{g_k^*(z,\eta,c(t,y))-g_k^*(z,0,c(t,y))}\right)e^{-iy\eta}
\,dzdyd\eta
\\=& -\sum_{j=1,2}\pd_y^2\left(S^j_k[\varphi_c'](x_t-2c-3(x_y)^2)
 -S^j_k[\pd_c\varphi_c](c_t-6c_yx_y)\right)-\pd_y^2(R^1_k+R^2_k)\,,
\end{split}  
\end{equation}
\begin{align*}
& R^1_k=3S^1_k[\varphi_c](x_{yy})-3S^1_k[\pd_z^{-1}\pd_c\varphi_c](c_{yy})\,,
\\ & 
R^2_k=3S^2_k[\varphi_c](x_{yy})-3S^2_k[\pd_z^{-1}\pd_c\varphi_c](c_{yy})
-3\sum_{j=1,2}S^j_k[\pd_z^{-1}\pd_c^2\varphi_c](c_y^2)\,.
\end{align*}
We rewrite the linear term $R^1_k$ as
$$\begin{pmatrix}R^1_1 \\ R^1_2\end{pmatrix}
=\wS_0\begin{pmatrix}c_{yy} \\ x_{yy} \end{pmatrix}\,,\quad
\wS_0=3\begin{pmatrix}-S^1_1[\pd_z^{-1}\pd_c\varphi_c] & S^1_1[\varphi_c]
\\ -S^1_2[\pd_z^{-1}\pd_c\varphi_c] & S^1_2[\varphi_c]\end{pmatrix}\,.$$
\par

Next, we deal with
$$\frac{1}{2\pi}\int_{-\eta_0}^{\eta_0}
\int_{\R^2} \ell_2\overline{g_k^*(z,\eta,c(t,y_1))}e^{i(y-y_1)\eta}
\,dzdy_1d\eta\,.$$
Let $S^3_k[p]$ and $S^4_k[p]$ be operators defined by
\begin{equation*}
S^3_k[p](f)(t,y)=\frac{1}{2\pi}\int_{-\eta_0}^{\eta_0}\int_{\R^2}
f(y_1)p(z+3t+L)\overline{g_k^*(z,\eta)}e^{i(y-y_1)\eta}dy_1dzd\eta\,,  
\end{equation*}
\begin{multline*}
S^4_k[p](f)(t,y)=\frac{1}{2\pi}\int_{-\eta_0}^{\eta_0}\int_{\R^2}
f(y_1)\tc(t,y_1)p(z+3t+L)\\ \times\overline{g_{k3}^*(z,\eta,c(t,y_1))}
e^{i(y-y_1)\eta}dy_1dzd\eta\,,
\end{multline*}
where $g_{k3}^*(z,\eta,c)=(c-2)^{-1}(g_k^*(z,\eta,c)-g_k^*(z,\eta))$.
By the definition of $\tpsi_c$,
\begin{equation}
  \label{eq:l21F-1}
  \begin{split}
& \frac{1}{2\pi}\int_{-\eta_0}^{\eta_0}\int_{\R^2}
\ell_{21}\overline{g_k^*(z,\eta,c(t,y))}e^{-iy\eta}\,dzdyd\eta
\\=& (S^3_k[\psi]+S^4_k[\psi])\bigl(\sqrt{2/c}(c_t-6c_yx_y)\bigr)
\\ &-2\sqrt{2}(S^3_k[\psi']+S^4_k[\psi'])
\bigl((\sqrt{c}-\sqrt{2})(x_t-4-3(x_y)^2\bigr)\,.
  \end{split}
\end{equation}
The operator norms of $S^j_k[\psi]$, $S^j_k[\psi']$
($j=3\,,\,4$, $k=1\,,\,2$) decay exponentially as $t\to\infty$
because $g_k^*(z,\eta)$
and $g_k^*(z,\eta,c)$ are exponentially localized as $z\to-\infty$
and  $\psi\in C_0^\infty(\R)$. See \eqref{eq:S3p} and \eqref{eq:S4p}
in Appendix~\ref{sec:[B3,pdy]}.

\par

Next, we decompose
$$(2\pi)^{-1}\int_{-\eta_0}^{\eta_0}\int_{\R^2}(\ell_{22}+\ell_{23})
\overline{g_k^*(z,\eta,c(t,y))}e^{-iy\eta}\,dzdyd\eta$$
into a linear part and a nonlinear part with respect to $\tc$ and $\tx$. 
The linear part can be written as
\begin{equation}
  \label{eq:der4}
\frac{1}{2\pi}\int_{-\eta_0}^{\eta_0}\int_{\R^2}
\ell_{2,lin}(t,z,y_1)\overline{g_k^*(z,\eta)}e^{i(y-y_1)\eta}dy_1dzd\eta
=:\ta_k(t,D_y)\tc\,,  
\end{equation}
where 
\begin{align*}
\ell_{2,lin}(t,z,y)=&
\tc(t,y)\pd_z\left\{\pd_z^2-1+6\varphi(z)\right\}\psi(z+3t+L)
\\ & -3c_{yy}(t,y)\int_z^\infty\psi(z_1+3t+L)dz_1\,,  
\end{align*}
\begin{equation}
\label{eq:def-ak}
\begin{split}
\ta_k(t,\eta)=& \Bigl[
\int_\R\left\{\pd_z\left(\pd_z^2-1+6\varphi(z)\right)\psi(z+3t+L)\right\}
\overline{g_k^*(z,\eta)}dz\\
& +3\eta^2\int_\R\left(\int_z^\infty\psi(z_1+3t+L)dz_1\right)
\overline{g_k^*(z,\eta)}dz\Bigr]\mathbf{1}_{[-\eta_0,\eta_0]}(\eta)\,,
\end{split}
\end{equation}
and the nonlinear part is 
\begin{equation}
  \label{eq:der5}
\begin{split}
R^3_k(t,y):=& \frac{1}{2\pi}\int_{-\eta_0}^{\eta_0}
\int_\R (\ell_{22}+\ell_{23})\overline{g_k^*(z,\eta,c(t,y_1))}
e^{i(y-y_1)\eta}dzdy_1d\eta \\
& -\frac{1}{2\pi}\int_{-\eta_0}^{\eta_0}
\int_\R \ell_{2,lin}\overline{g_k^*(z,\eta)}e^{i(y-y_1)\eta}dzdy_1d\eta\,.
\end{split}
\end{equation}
\par

Next, we deal with  $II^j_k$ $(j=1\,,\,\cdots\,,\,6)$ in \eqref{eq:orth_t}.
Let 
\begin{align*}
II^3_{k1}=& -3\int_{\R^2}v_2(t,z,y)x_{yy}(t,y)
\overline{g_k^*(z,\eta,c(t,y))}e^{-iy\eta}\,dzdy\,,\\
II^3_{k2}=& 6\int_{\R^2}v_2(t,z,y)x_y(t,y)
\overline{g_k^*(z,\eta,c(t,y))}e^{-iy\eta}\,dzdy
\end{align*}
so that $II^3_k=II^3_{k1}+i\eta II^3_{k2}$.
For $k=1$ and $2$, let
\begin{equation}
\label{eq:der6}
\begin{split}
R^4_k(t,y)=& \frac{1}{2\pi}\int_{-\eta_0}^{\eta_0}
\left\{II^1_k(t,\eta)+II^2_k(t,\eta)+ II^3_{k1}(t,\eta)\right\}e^{iy\eta}d\eta\,,
\\
R^5_k(t,y)=& \frac{1}{2\pi}\int_{-\eta_0}^{\eta_0}
II^3_{k2}(t,\eta)e^{iy\eta}\,d\eta\,.
\end{split}
\end{equation}
\par
Let $S^5_k$ and $S^6_k$ be operators defined by
\begin{gather*}
S^5_k(f)(t,y)=\frac{1}{2\pi}\int_{-\eta_0}^{\eta_0}
\int_{\R^2} v_2(t,z,y_1)f(y_1)\overline{\pd_cg_k^*(z,\eta,c(t,y_1))}
e^{i(y-y_1)\eta}\,dzdy_1d\eta\,,\\
S^6_k(f)(t,y)=-\frac{1}{2\pi}\int_{-\eta_0}^{\eta_0}\int_{\R^2}
v_2(t,z,y_1)f(y_1)\overline{\pd_zg_k^*(z,\eta,c(t,y_1))}
e^{i(y-y_1)\eta}\,dzdy_1d\eta\,,
\end{gather*}
and
$$R^6_k=-\frac{3}{\pi}\int_{-\eta_0}^{\eta_0}\int_{\R^2}\psi_{c(t,y_1),L}(z+3t)
v_2(t,z,y_1)\overline{\pd_zg_k^*(z,\eta,c(t,y_1))}e^{i(y-y_1)\eta}dy_1dzd\eta\,.$$
Then
\begin{equation}
  \label{eq:der7}
\begin{split}
& \mathbf{1}_{[-\eta_0,\eta_0]}(\eta)II^4_k(t,\eta)=
\sqrt{2\pi}\mF_y(S^5_k(c_t-6c_yx_y))\,,\\
& \mathbf{1}_{[-\eta_0,\eta_0]}(\eta)II^5_k(t,\eta)
=\sqrt{2\pi}\mF_y\left\{S^6_k\left(x_t-2c-3(x_y)^2\right)+R^6_k\right\}\,,
\end{split}  
\end{equation}
Let  $R^{v_1}={}^t(R^{v_1}_1,R^{v_1}_2)$ and
\begin{equation}
\label{eq:derv1}
R^{v_1}_k(t,y)=\frac{1}{2\pi}\int_{-\eta_0}^{\eta_0}II^6_k(t,\eta)e^{iy\eta}\,d\eta
\quad\text{for $k=1$ and $2$.}
\end{equation}
\par
Using \eqref{eq:Gk}--\eqref{eq:derv1},
we can translate \eqref{eq:orth_t} as
\begin{equation}
  \label{eq:modpre}
  \begin{multlined}
\wP_1\begin{pmatrix}  G_1 \\ G_2\end{pmatrix}
-\left(\pd_y^2(\wS_1+\wS_2)-\wS_3-\wS_4-\wS_5\right)
\begin{pmatrix} c_t-6c_yx_y \\ x_t-2c-3(x_y)^2\end{pmatrix}
\\ 
+\widetilde{\mathcal{A}}_1(t)
\begin{pmatrix}  \tc \\ x\end{pmatrix}
-\pd_y^2R^1+\wR^1+\pd_y\wR^2+R^{v_1}=0\,,
  \end{multlined}
\end{equation}
where $R^j={}^t\!(R^j_1\,,\,R^j_2)$ for $j=1, \cdots, 6, v_1$ and
\begin{align*}
& \wS_j=\begin{pmatrix}-S^j_1[\pd_c\varphi_c]  & S^j_1[\varphi_c']
\\ -S^j_2[\pd_c\varphi_c]  & S^j_2[\varphi_c']\end{pmatrix}
\enskip\text{for $j=1$, $2$,}\quad
\wS_3=\begin{pmatrix}
S^3_1[\psi] & 0
\\ S^3_2[\psi] & 0\end{pmatrix}\,,
\\ &
\wS_4=\begin{pmatrix}
S^3_1[\psi]((\sqrt{2/c}-1)\cdot)+S^4_1[\psi](\sqrt{2/c}\cdot)
& -2(S^3_1[\psi']+S^4_1[\psi'])((\sqrt{2c}-2)\cdot) 
\\ S^3_2[\psi]((\sqrt{2/c}-1)\cdot)+S^4_2[\psi](\sqrt{2/c}\cdot)
& -2(S^3_2[\psi']+S^4_2[\psi'])((\sqrt{2c}-2)\cdot)
\end{pmatrix}\,,\\
& \wS_5=\begin{pmatrix}  S^5_1 & S^6_1 \\ S^5_2 & S^6_2 \end{pmatrix}\,,
\quad
\widetilde{\mathcal{A}}_1(t)=
\begin{pmatrix} \ta_1(t,D_y) & 0  \\ \ta_2(t,D_y) & 0\end{pmatrix}\,,
\\ &\wR^1=R^3+R^4+R^6+\wS_4\begin{pmatrix} 0 \\ 2\tc\end{pmatrix}\,,\quad
\wR^2=R^5-\pd_yR^2\,.
\end{align*}
To translate the nonlinear terms $6(c/2)^{1/2}c_yx_y$
and $16x_{yy}\{((c/2)^{3/2}-1\}$ in $G_1$ into a divergence form,
we will make use of the following change of variables. Let
\begin{align}
\label{eq:bdef}
& b(t,\cdot)=\frac{1}{3}\wP_1\left\{\sqrt{2}c(t,\cdot)^{3/2}-4\right\}\,,
\quad
\cC_1=\frac12\wP_1\left\{c(t,\cdot)^2-4\right\}\wP_1\,,\\
& \wC_1=\begin{pmatrix}0 & 0 \\ 0 & \cC_1\end{pmatrix}\,,\quad
B_1=\begin{pmatrix}  2 & 0 \\ \frac12 & 2\end{pmatrix}\,,\quad
B_2=\begin{pmatrix}  6 & 16 \\ \mu_1 & 6\end{pmatrix}\,. \notag
\end{align}
We remark that $b\simeq \tilde{c}=c-2$ if $c$ is close to $2$
(see \cite[Claim~D.6]{Mi}).
By \eqref{eq:bdef}, we have
$b_t=\wP_1(c/2)^{1/2}c_t$, $b_y=\wP_1(c/2)^{1/2}c_y$ and it follows from
Lemma~\ref{LEM:G} that 
\begin{equation}
  \label{eq:G-T}
\begin{split}
\wP_1\begin{pmatrix}G_1 \\ G_2 \end{pmatrix}=& 
-(B_1+\wC_1)\wP_1\begin{pmatrix} b_t-6(bx_y)_y \\ x_t-2c-3(x_y)^2 \end{pmatrix}
+B_2\begin{pmatrix}c_{yy} \\ x_{yy} \end{pmatrix} +\wP_1R^7\,,
\end{split}  
\end{equation}
where $R^7={}^t(R^7_1,R^7_2)$ and
\begin{equation}
  \label{eq:R6-def}
\begin{split}
R^7_1=& \left\{4\sqrt{2}c^{3/2}-16-12b\right\}x_{yy}
-6(2b_y-(2c)^{1/2}c_y)x_y-3c^{-1}(c_y)^2\,,\\
R^7_2=& 6\left\{\left(\frac{c}{2}\right)^{3/2}-1\right\}x_{yy}
+3\left(\frac{c}{2}\right)^{1/2}c_yx_y-3(bx_y)_y+\mu_2\frac{2}{c}(c_y)^2
\\ & + \frac{3}{2}(c^2-4)(I-\wP_1)(x_y)^2\,.
\end{split}  
\end{equation}
Let 
$\cC_2=\wP_1\left\{\left(\frac{c(t,\cdot)}{2}\right)^{1/2}-1\right\}\wP_1$,
$\wC_2=\begin{pmatrix}\cC_2 & 0 \\ 0 & 0 \end{pmatrix}$,
$\bS_j=\wS_j(I+\wC_2)^{-1}$ for $1\le j\le 5$ and
\begin{equation}
  \label{eq:def-B3}
B_3=B_1+\wC_1+\pd_y^2(\bS_1+\bS_2)-\bS_3-\bS_4-\bS_5\,.
\end{equation}
Note that $I+\wC_2$ is invertible as long as $\tc(t,\cdot)$ remains small in $Y$
and that $B_3$ is a bounded operator on $Y\times Y$ depending on $\tc$ and $v$.
Substituting \eqref{eq:G-T} into \eqref{eq:modpre}, we have
\begin{align*}
& B_3\wP_1\begin{pmatrix}b_t-6(bx_y)_y \\ x_t-2c-3(x_y)^2 \end{pmatrix}
\\ =& \left\{(B_2-\pd_y^2\wS_0)\pd_y^2+\widetilde{\mathcal{A}}_1(t)\right\}
\begin{pmatrix}b \\ x \end{pmatrix}
+\wP_1R^7+\wR^1+\wR^3+\pd_y(\wR^2+\wR^4)+R^{v_1}\,,
\end{align*}
where $\wR^3=R^9+R^{11}$, $\wR^4=R^8+R^{10}$ and
\begin{align*}
& R^8=6\pd_y(\bS_1+\bS_2)
\begin{pmatrix}(I+\mathcal{C}_2)(c_yx_y)-(bx_y)_y \\ 0\end{pmatrix}\,,\\
&R^9=-6\sum_{3\le j\le 5}\bS_j
\begin{pmatrix}(I+\mathcal{C}_2)(c_yx_y)-(bx_y)_y \\ 0\end{pmatrix}\,,\\
& R^{10}=(\pd_y^2\wS_0-B_2)\begin{pmatrix} b_y-c_y\\ 0\end{pmatrix}\,,\quad
R^{11}=\widetilde{\mathcal{A}}_1(t)\begin{pmatrix}\tc -b \\ 0 \end{pmatrix}\,.
\end{align*}

We have the following.
\begin{proposition}
\label{prop:modulation}
There exists a $\delta_3>0$ such that if 
$\bM_{c,x}(T)+\bM_2(T)+\eta_0+e^{-\a L}<\delta_3$ for a $T\ge0$, then  
\begin{align}
  \label{eq:modeq}
\begin{pmatrix}b_t \\ \tx_t\end{pmatrix}=\mathcal{A}(t)
\begin{pmatrix}b \\ \tx\end{pmatrix}\,+\sum_{i=1}^5 \cN^i\,,
\end{align}
where $B_4=B_1+\pd_y^2\wS_1-\wS_3=B_3|_{\tc=0\,,\,v_2=0}$,
\begin{align*}
& \mathcal{A}(t)=B_4^{-1}(B_2-\pd_y^2\wS_0)\pd_y^2
+B_3^{-1}\widetilde{\mathcal{A}}_1(t)
+\begin{pmatrix}  0 & 0 \\ 2 & 0\end{pmatrix}\,,\\
\cN^1=& \wP_1\begin{pmatrix}
6(b\tx_y)_y\\ 2(\tilde{c}-b)+3(\tx_y)^2
\end{pmatrix}\,,\quad \cN^2=\cN^{2a}+\cN^{2b}\,,\\
 \cN^{2a}=& B_3^{-1}\left(
\wP_1\begin{pmatrix}R^7_1 \\ 0\end{pmatrix}+\wR^1+\wR^3\right)\,,
\quad \cN^{2b}=B_3^{-1}\wP_1\begin{pmatrix}0 \\ R^7_2\end{pmatrix}\,,
\\ \cN^3=& B_3^{-1}\pd_y(\wR^2+\wR^4)\,,\quad
\cN^4=(B_3^{-1}-B_4^{-1})(B_2-\pd_y^2\wS_0)\pd_y
\begin{pmatrix} b_y \\ x_y\end{pmatrix}\,,
\\ \cN^5=& B_3^{-1}R^{v_1}\,.
\end{align*}
Moreover, if $v_2(0)=0$,
\begin{equation}
  \label{eq:modeq-init}
b(0,\cdot)=0\,,\quad x(0,\cdot)=0\,.
\end{equation}
\end{proposition}
\begin{proof}
  Proposition~\ref{prop:continuation} implies that the
  \eqref{eq:decomp} persists on $[0,T]$ if $\delta_3$ is sufficiently
  small. Moreover Claims~\ref{cl:(1+wT2)^{-1}}--\ref{cl:invB_3} below
  imply that $B_3$, $B_4$ and $I+\wC_k$ are invertible if
  $\|\tc(t)\|_Y$, $\|v(t)\|_X$, $\eta_0$ and $e^{-\a L}$ are
  sufficiently small.  Thus we have \eqref{eq:modeq}.
Since $v_2(0)=0$, we have \eqref{eq:modeq-init} from  Lemma~\ref{lem:decomp}.
This completes the proof of Proposition~\ref{prop:modulation}.
\end{proof}

\begin{claim}
  \label{cl:(1+wT2)^{-1}}
There exist positive constants $\delta$ and $C$ such that
if $\bM_{c,x}(T)\le \delta$, then for $s\in[0,T]$ and $k=1$, $2$,
\begin{align}
  \label{eq:T-bound1}
& \sup_{t\in[0,T]}\|\wC_k(t)\|_{B(Y)}+\|\wC_k\|_{L^4(0,T;B(Y))}\le  C\bM_{c,x}(T)\,,\\
  \label{eq:T-bound}
& \sup_{t\in[0,T]}\|\wC_k(t)\|_{B(Y,Y_1)}\le C\bM_{c,x}(T)\,,\\
\notag
& \|(I+\wC_k)^{-1}\|_{B(Y)}+\|(I+\wC_k)^{-1}\|_{B(Y_1)}\le C\,.
\end{align}
\end{claim}
Claim~\ref{cl:(1+wT2)^{-1}} follows from \cite[Claim~B.6]{Mi} and
the definition of $\bM_{c,x}(T)$.
\begin{claim}
\label{cl:invB_4}
There exist positive constants $C$ and $\delta$ such that
if $\eta_0^2+e^{-\a L}\le \delta$, then 
$$\|B_4^{-1}\|_{B(Y)}+\|B_4^{-1}\|_{B(Y_1)}\le C\,.$$
\end{claim}
\begin{claim}
  \label{cl:invB_3}
There exist positive constants $\delta$ and $C$ such that
if $\bM_{c,x}(T)+\bM_2(T)+\eta_0^2+e^{-\a L}\le \delta$, then for $t\in[0,T]$,
\begin{gather*}
\|B_3-B_4\|_{B(Y)}+\|B_3-B_4\|_{B(Y_1)}\le C(\bM_{c,x}(T)+\bM_2(T))\,,\\
\|B_3^{-1}\|_{B(Y)}+\|B_3^{-1}\|_{B(Y_1)}\le C\,.  
\end{gather*}
\end{claim}
The proof of Claims~\ref{cl:invB_4} and \ref{cl:invB_3} is exactly the same
as the proof of Claims~6.2 and 6.3 in \cite{Mi}.
\bigskip

\section{\`A priori estimates for the local speed and the local phase shift}
\label{sec:apriori}
In this section, we will estimate $\bM_{c,x}(T)$
assuming the smallness of $\bM_{c,x}(T)$, $\bM_i(T)$ ($i=1,2$), $\eta_0$
and $e^{-\a L}$.
\begin{lemma}
  \label{lem:M1-bound}
There exist positive constants $\delta_4$ and $C$ such that if
$\bM_{c,x}(T)+\bM_1(T)+\bM_2(T)+\eta_0+e^{-\a L}\le \delta_4$, then
\begin{equation}
  \label{eq:M1-bound}
\bM_{c,x}(T)\le C(\|v_0\|_{L^2(\R^2)}+\bM_1(T)+\bM_2(T)^2)\,.
\end{equation}
\end{lemma}

Before we start to prove Lemma~\ref{lem:M1-bound},
we  estimate the upper bound of $c_t$ and $x_t-2c-3(x_y)^2$.
\begin{lemma}
Let $\delta_3$ be as in Proposition~\ref{prop:modulation}. Suppose
$\bM_{c,x}(T)+\bM_1(T)+\bM_2(T)+\eta_0+e^{-\a L}<\delta_3$ for a $T\ge0$. Then
  \label{lem:cx_t-bound}
\end{lemma}
  \begin{align*}
& \|c_t\|_{L^\infty(0,T;Y)\cap L^2(0,T;Y)}+\|x_t-2c-3(x_y)^2\|_{L^\infty(0,T;L^2(\R))\cap L^2(0,T;L^2(\R))}
\\ \lesssim & \eta_0^{-1/2}\bM_{c,x}(T)^2+\bM_{c,x}(T)+\bM_1(T)+\bM_2(T)^2\,.
  \end{align*}

To begin with, we will estimate the nonlinear terms of \eqref{eq:modeq}.
\begin{claim}
\label{cl:N1-N4-est}
\begin{gather}
\label{eq:cN1-est}
\sup_{t\in[0,T]}\|b\tx_y\|_Y+\|(b\tx_y)_y\|_{L^2(0,T;Y)}\lesssim \bM_{c,x}(T)^2\,,
\\ \label{eq:cN2-esta}
\sup_{t\in[0,T]}\|\cN^{2a}(t)\|_Y+\|\cN^{2a}\|_{L^1(0,T;Y)}
\lesssim \bM_{c,x}(T)^2+\bM_1(T)^2+\bM_2(T)^2\,,
\\ \label{eq:cN2-estb}
\sup_{t\in[0,T]}\|\cN^{2b}(t)\|_Y+\|\cN^{2b}\|_{L^2(0,T;Y)}
\lesssim \bM_{c,x}(T)^2+\bM_1(T)^2+\bM_2(T)^2\,,
\\ \label{eq:cN3-est}
\sup_{t\in[0,T]}\|\cN^3(t)\|_Y+\|\cN^3\|_{L^2(0,T;Y)}
\lesssim \bM_{c,x}(T)^2+\bM_{c,x}(T)\bM_2(T)\,,
\\ \label{eq:cN4-est}
\sup_{t\in[0,T]}\|\cN^4(t)\|_Y+\|\cN^4\|_{L^2(0,T;Y)}
\lesssim \bM_{c,x}(T)^2+\bM_{c,x}(T)\bM_2(T)\,,
\\\label{eq:cN5-est}
\sup_{t\in[0,T]}\|\cN^5(t)\|_Y+\|\cN^5\|_{L^2(0,T;Y)}\lesssim  \bM_1(T)\,.
\end{gather}  
\end{claim}
\begin{proof}[Proof of Claim~\ref{cl:N1-N4-est}]
Eq.~\eqref{eq:cN1-est} follows from \cite[Claim~D.6]{Mi}
and the fact that $Y\subset H^1(\R)$.
Eqs.~\eqref{eq:cN2-esta}--\eqref{eq:cN3-est} follow from
Claims~\ref{cl:invB_3}, \ref{cl:R1-R2}, \ref{cl:R3}, 
\ref{cl:R4-R5}--\ref{cl:R8-R11}, \eqref{eq:S3p} and \eqref{eq:S4p}.
\par
Next, we will estimate $\cN^4$.
Let
$\bS'=\pd_y^2\bS_2+\bS_4+\bS_5$ and $\bS''=\pd_y^2(\wS_1-\bS_1)+\wS_3-\bS_3$.
Then $B_3^{-1}-B_4^{-1}=B_3^{-1}(\bS'+\bS'')B_4^{-1}$ and
\begin{equation}
  \label{eq:bS'-est}
 \sup_{t\in[0,T]}\|\bS'\|_{B(Y,Y_1)}\lesssim \bM_{c,x}(T)+\bM_2(T)  
\end{equation}
by \eqref{eq:wS2-est}, \eqref{eq:bS3} and \eqref{eq:bS5} and
\begin{equation}
  \label{eq:bS''-est}
 \sup_{t\in[0,T]}\|\bS''\|_{B(Y,Y_1)}\lesssim (\eta_0^2+e^{-\a L})\bM_{c,x}(T)  
\end{equation}
by \eqref{eq:wS1-est}, \eqref{eq:bS3} and Claim~\ref{cl:(1+wT2)^{-1}}.
Combining \eqref{eq:bS'-est}, \eqref{eq:bS''-est} with
Claims~\ref{cl:invB_4} and \ref{cl:invB_3}, we have \eqref{eq:cN4-est}.
We can prove \eqref{eq:cN5-est} in the same way as \eqref{eq:Rv1-12-est}
of Claim~\ref{cl:Rv1} in Appendix~\ref{sec:Rk}.
\end{proof}

\begin{proof}[Proof of Lemma~\ref{lem:cx_t-bound}]
Claims~\ref{cl:N1-N4-est} and \ref{cl:akbound},
\eqref{eq:modeq} and \cite[(D.12)]{Mi} imply
\begin{align*}
 & \|c_t\|_{L^\infty(0,T;Y)\cap L^2(0,T;Y)}
+\|x_t-2c-3\wP_1(x_y)^2\|_{L^\infty(0,T;Y)\cap L^2(0,T;Y)}
\\ \lesssim &
\|b_{yy}\|_Y+\|x_{yy}\|_Y+\|\widetilde{\mathcal{A}}_1(t)(b,\tx)\|_Y
+\|(b\tx_y)_y\|_Y+\sum_{2\le i\le 5}\|\cN^i\|_Y
\\ \lesssim & \bM_{c,x}(T)+\bM_1(T)+\bM_2(T)^2\,.    
\end{align*}
Since $\mF_y\{(I-\wP_1)(x_y^2)\}(t,\eta)=0$ for $\eta\in[-\eta_0,\eta_0]$,
we have
\begin{equation}
  \label{eq:wpxy}
\|(I-\wP_1)(x_y)^2\|_{L^2}\le \eta_0^{-1}\|\pd_y(x_y)^2\|_{L^2}\lesssim \eta_0^{-1/2}\|x_y\|_Y\|x_{yy}\|_Y\,,  
\end{equation}
whence 
$\|(I-\wP_1)(x_y)^2\|_{L^\infty(0,T;L^2)\cap L^2(0,T;L^2)}\lesssim
\eta_0^{-1/2}\bM_{c,x}(T)^2$. Thus we prove Lemma~\ref{lem:cx_t-bound}.
  \end{proof}

To prove Lemma~\ref{lem:M1-bound}, we need the following.
\begin{claim}
  \label{cl:[B3,pdy]}
There exist positive constants $\eta_1$, $\delta$ and $C$ such that
if $\eta_0\in(0,\eta_1]$ and $\bM_{c,x}(T)\le \delta$, then $[\pd_y,B_4]=0$,
\begin{gather*}
\|[\pd_y,B_3]f\|_{L^2(0,T;Y_1)}\le C(\bM_{c,x}(T)+\bM_2(T))\sup_{t\in[0,T]}\|f(t)\|_Y\,,\\
\|[\pd_y,B_3]f\|_{L^1(0,T;Y_1)}\le C(\bM_{c,x}(T)+\bM_2(T))\|f\|_{L^2(0,T;Y)}\,.
\end{gather*}
\end{claim}
The proof is given in Appendix~\ref{sec:[B3,pdy]}.

\begin{proof}[Proof of Lemma~\ref{lem:M1-bound}]
Let us translate \eqref{eq:modeq} into a system of $b$ and $x_y$.
Let
\begin{align*}
& A(t)=\diag(1,\pd_y)\mathcal{A}(t)\diag(1,\pd_y^{-1})\,,\quad
B_5=B_1+\pd_y^2\wS_1\,,\\
& A_0=\diag(1,\pd_y)\left\{B_5^{-1}(B_2-\pd_y^2\wS_0)\pd_y^2
+\begin{pmatrix}0 & 0 \\ 2 & 0\end{pmatrix}\right\}\diag(1,\pd_y)^{-1}\,,\\
& A_1(t,D_y)=\diag(1,\pd_y)(B_4^{-1}-B_5^{-1})(B_2-\pd_y^2\wS_0)
\diag(\pd_y^2,\pd_y)
\\ & \phantom{A_1(t,D_y)=}
+\diag(1,\pd_y)B_3^{-1}\widetilde{\mathcal{A}}_1(t)\,,
\end{align*}
where $\pd_y^{-1}=\mF_\eta^{-1}(i\eta)^{-1}\mF_y$.
Then $A(t)=A_0(D_y)+A_1(t,D_y)$. 
Note that $\widetilde{\mathcal{A}}_1(t)=\linebreak
\widetilde{\mathcal{A}}_1(t)\diag(1,\pd_y^{-1})$.
Multiplying \eqref{eq:modeq} by $\diag(1,\pd_y)$ from the left,
we can transform \eqref{eq:modeq} into
\begin{equation}
  \label{eq:modeq2}\left\{
  \begin{aligned}
& \pd_t\begin{pmatrix}b \\ x_y\end{pmatrix}=A(t)
\begin{pmatrix}b \\ x_y\end{pmatrix}\,+\sum_{i=1}^5\diag(1,\pd_y)\cN^i\,,
\\ & b(0,\cdot)=0\,,\quad x_y(0,\cdot)=0\,.
  \end{aligned}\right.
\end{equation}
Let
$A_0(\eta)$ be the Fourier transform of the operator $A_0$.
Then 
\begin{equation}
  \label{eq:A0A*}
\begin{split}
A_0(\eta)=& \begin{pmatrix}1 & 0\\ 0 & i\eta\end{pmatrix}
(B_1^{-1}+O(\eta^2))(B_2+O(\eta^2))
\begin{pmatrix}-\eta^2 & 0\\ 0 & i\eta\end{pmatrix}
+
\begin{pmatrix}  0 & 0 \\ 2i\eta & 0\end{pmatrix}
\\=& 
A_*(\eta)
+\begin{pmatrix} O(\eta^4) & O(\eta^3)\\ O(\eta^5) & O(\eta^4)\end{pmatrix}\,,
\end{split}  
\end{equation}
where 
$A_*(\eta)=
\begin{pmatrix}-3\eta^2 & 8i\eta \\ i\eta(2+\mu_3\eta^2) & -\eta^2
\end{pmatrix}$ and
$\mu_3=-\frac{\mu_1}{2}+\frac{3}{4}=\frac{1}{2}+\frac{\pi^2}{24}>1/8$.

Next, we will diagonalize $A_*(\eta)$, a lower order part of $A_0(\eta)$.
Let $\omega(\eta)=\sqrt{16+(8\mu_3-1)\eta^2}$,
$\lambda_*^\pm(\eta)=-2\eta^2\pm i\eta\omega(\eta)$ and
$$\Pi_*(\eta)=\frac{1}{4i}
\begin{pmatrix} 8i & 8i \\ \eta+i\omega(\eta) & \eta-i\omega(\eta)
\end{pmatrix}\,.$$
Then $$\Pi_*(\eta)^{-1}A_*(\eta)\Pi_*(\eta)=
\diag(\lambda^+_*(\eta)\,,\,\lambda^-_*(\eta))\,.$$
We remark that if $\mu_3$ is replaced by $1/8$, then $\omega(\eta)=4$ and
$e^{tA_*(D_y)}$ is a composition of the wave and heat kernels.
In out setting,
\begin{equation}
  \label{eq:omega-approx}
 |\omega(\eta)-4|\lesssim \eta^2\,.
\end{equation}
\par
By the change of variables
\begin{gather*}
\bb(t,y)=\begin{pmatrix}  b_1(t,y) \\ b_2(t,y)\end{pmatrix}\,,
\quad
\begin{pmatrix}b(t,\cdot) \\ x_y(t,\cdot)\end{pmatrix}=
\Pi_*(D_y) \begin{pmatrix}  b_1(t,\cdot)\\ b_2(t,\cdot)\end{pmatrix}\,,
\end{gather*}
we have
\begin{equation}
  \label{eq:bd}
  \begin{split}
\pd_t\bb=\{2\pd_y^2I+& \pd_y\omega(D_y)\sigma_3+A_2(D_y)+A_3(t,D_y)\}\bb
\\ & +\Pi_*^{-1}(D_y)\sum_{i=1}^5\diag(1,\pd_y)\cN^i\,,    
  \end{split}
\end{equation}
where
\begin{gather*}
\sigma_3=\begin{pmatrix} 1 & 0 \\ 0 &1\end{pmatrix}\,,\quad
A_2(\eta)=\Pi_*(\eta)^{-1}(A_0(\eta)-A_*(\eta))\Pi_*(\eta)\,,\\
A_3(t,\eta)=\Pi_*(\eta)^{-1}A_1(t,\eta)\Pi_*(\eta)\,.
\end{gather*}
For $\eta\in[-\eta_0,\eta_0]$,
\begin{equation}
\label{eq:Pi-est}
\left|\Pi_*(\eta)
-\begin{pmatrix}2 & 2\\ 1 & -1 \end{pmatrix}\right|
+\left|\Pi_*(\eta)^{-1}-\frac{1}{4}
\begin{pmatrix}1 & 2 \\ 1 & -2 \end{pmatrix}\right|
\lesssim |\eta|\,.
\end{equation}
Hence $\Pi_*(D_y)$ and $\Pi_*^{-1}(D_y)$ are bounded operator on $Y$
for sufficiently small $\eta_0$.
By \eqref{eq:Pi-est} and Plancherel's theorem,
\begin{equation}
  \label{eq:bx-bb} 
\left\|\begin{pmatrix}  b(t,\cdot) \\ x_y(t,\cdot)\end{pmatrix}
-\begin{pmatrix}2 & 2 \\ 1 & -1 \end{pmatrix}\bb(t,\cdot)
\right\|_Y
\lesssim \|\pd_y\bb(t,\cdot)\|_Y\,.
\end{equation}
By \eqref{eq:A0A*} and \eqref{eq:Pi-est},
\begin{equation}
  \label{eq:A2-size}
A_2(\eta)=O(\eta^3)\,.  
\end{equation}
Since $\|A_1(t,D_y)\|_{B(Y)}\lesssim e^{-\a(3t+L)}$ for $t\ge0$ by
Claim~\ref{cl:akbound}, 
\begin{equation}
  \label{eq:A3-size}
\|A_3(t,D_y)\|_{B(Y)}  \lesssim e^{-\a(3t+L)}\quad\text{for $t\ge0$.}
\end{equation}
\par

To obtain the energy estimate for $b_1$ and $b_2$,
we translate the nonlinear term as 
\begin{equation}
  \label{eq:nl}
\Pi_*^{-1}(D_y)\sum_{i=1}^5\diag(1,\pd_y)\cN^i=
\cN'+\pd_y(\cN^0+\cN'')-\pd_tK(t,y)  
\end{equation}
such that $\cN_0$ is cubic in $b_1$ and $b_2$, that
$\lim_{t\to\infty}\|K(t,\cdot)\|_Y=0$ and that
\begin{equation}
  \label{eq:N'N''}
  \begin{split}
& \sup_{t\in[0,T]}\|\cN'(t)\|_{Y}+\|\cN'(t)\|_{L^1(0,T;Y)}
\lesssim  (e^{-\a L}+\bM_{c,x}(T))\bM_{c,x}(T)
\\ & \phantom{\sup_{t\in[0,T]}\|\cN'(t)\|_{Y}+\|\cN'(t)\|_{L^1(0,T;Y)}\lesssim}
+\bM_1(T)^2+\bM_2(T)^2\,,\\
& \sup_{t\in[0,T]}\|\cN''(t)\|_{Y}+\|\cN''\|_{L^2(0,T;Y)}
\lesssim \bM_1(T)+\bM_{c,x}(T)(\bM_{c,x}(T)+\bM_2(T))\,.
  \end{split}
\end{equation}

To begin with, we will translate the dominant part of
$\Pi_*^{-1}(D_y)\diag(1,\pd_y)\cN^1$ in terms of $b_1$ and $b_2$.
Let
\begin{gather*}
\widetilde{\cN}^0=\Pi_*^{-1}(D_y)\wP_1
\begin{pmatrix}  6(bx_y) \\ 3(x_y)^2-\frac14b^2\end{pmatrix}
\,,\quad
\widetilde{\cN}^1=\Pi_*^{-1}(D_y)\wP_1
\begin{pmatrix} 0 \\ \frac{1}{4}b^2-2(b-\tc)\end{pmatrix}\,,\\
\cN^0=\wP_1\begin{pmatrix}4b_1^2-4b_1b_2-2b_2^2
\\ 2b_1^2+4b_1b_2-4b_2^2\end{pmatrix}\,,
\quad \widetilde{\cN}^2=\widetilde{\cN}^0-\cN^0\,.
\end{gather*}
Then $\Pi_*^{-1}(D_y)\diag(1,\pd_y)\cN^1
=\pd_y(\cN^0+\widetilde{\cN}^1+\widetilde{\cN}^2)$.
By \cite[(D.16)]{Mi} and the Sobolev inequality
$\|f\|_{L^\infty(\R)}^2\le 2\|f\|_{L^2(\R)}\|f'\|_{L^2(\R)}$,
\begin{equation}
  \label{eq:wcN1-est}
 \sup_{t\in[0,T]}\|\widetilde{\cN}^1(t)\|_Y
+\|\widetilde{\cN}^1\|_{L^2(0,T;Y)}\lesssim \bM_{c,x}(T)^3\,.
\end{equation}
It follows from \eqref{eq:Pi-est} and \eqref{eq:bx-bb} that
$\|\widetilde{\cN}^2(t,\cdot)\|_Y
\lesssim \|\bb(t,\cdot)\|_Y\|\pd_y\bb(t,\cdot)\|_Y$ and that
\begin{equation}
  \label{eq:wcN2-est}
\sup_{t\in[0,T]}\|\widetilde{\cN}^2(t,\cdot)\|_Y
+\|\widetilde{\cN}^2\|_{L^2(0,T;Y)}\lesssim \bM_{c,x}(T)^2\,.
\end{equation}
\par
Next, we will decompose $\diag(1,\pd_y)\cN^{2b}$ into a sum of an $L^1(0,T;Y)$ function
and a $y$-derivative of $L^2(0,T;Y)$ and read $\cN^2$ as
\begin{equation}
  \label{eq:N2-est}
  \begin{split}
& \diag(1,\pd_y)\cN^2=\diag(1,\pd_y)\cN^{21}+\cN^{22}\,,\\
& \sup_{t\in[0,T]}\|\cN^{21}\|_Y+\|\cN^{21}\|_{L^1(0,T;Y)}\lesssim \bM_{c,x}(T)^2+\bM_1(T)^2+\bM_2(T)^2\,,\\
& \sup_{t\in[0,T]}\|\cN^{22}\|_Y+\|\cN^{22}\|_{L^2(0,T;Y)}\lesssim \bM_{c,x}(T)^2\,.
  \end{split}
\end{equation}
By \eqref{eq:def-B3},
\begin{equation}
  \label{eq:invB3-invB1}
B_3^{-1}=B_1^{-1}-B_1^{-1}\left(\wC_1+\pd_y^2\sum_{j=1,2}\bS_j-\sum_{3\le j\le 5}\bS_j\right)B_3^{-1}\,.  
\end{equation}
Let $E_2=\begin{pmatrix}  0 & 0\\ 0 & 1\end{pmatrix}$.
Since
\begin{equation}
  \label{eq:lwcomp}
\diag(1,\pd_y)B_1^{-1}E_2=\frac12\pd_yE_2\,,\quad \diag(1,\pd_y)B_1^{-1}\wC_1=\frac12\pd_y\wC_1\,,
\end{equation}
we have $\diag(1,\pd_y)\cN^{2b}=\pd_y\cN^{2b1}+\diag(1,\pd_y)\cN^{2b2}$,
where
\begin{gather*}
\cN^{2b1}=\left\{\frac12(E_2-\wC_1B_3^{-1})+\diag(\pd_y,\pd_y^2)\sum_{j=1,2}B_1^{-1}\bS_jB_3^{-1}\right\}
\begin{pmatrix}0 \\ R^7_2\end{pmatrix}\,,
\\  
\cN^{2b2}=-\sum_{3\le j\le 5}B_1^{-1}\bS_jB_3^{-1}\begin{pmatrix}0 \\ R^7_2\end{pmatrix}\,.
\end{gather*}
By \eqref{eq:R6-est3}, \eqref{eq:bS3} and \eqref{eq:bS5},
\begin{gather*}
\sup_{t\in[0,T]}\|\cN^{2b1}\|_Y+\|\cN^{2b1}\|_{L^2(0,T;Y)}\lesssim \bM_{c,x}(T)^2\,,\\
\sup_{t\in[0,T]}\|\cN^{2b2}\|_Y+\|\cN^{2b2}\|_{L^1(0,T;Y)}\lesssim (\bM_{c,x}(T)+\bM_2(T))\bM_{c,x}(T)^2\,,
\end{gather*}
and it follows from Claim~\ref{cl:N1-N4-est} and the above that
$\cN^{21}:=\cN^{2a}+\cN^{2b2}$ and $\cN^{22}:=\cN^{2b1}$ satisfy \eqref{eq:N2-est}.
\par
Let
$$\cN^{31}=[B_3^{-1},\pd_y](\wR^2+\wR^4)\,,
\quad \cN^{32}=B_3^{-1}(\wR^2+\wR^4)\,.$$
Then $\cN^3=\cN^{31}+\pd_y\cN^{32}$ and we have
\begin{equation}
  \label{eq:cN31-32-est}
  \begin{split}
& \sup_{t\in[0,T]}\|\cN^{31}(t)\|_{Y_1}+\|\cN^{31}\|_{L^1(0,T;Y_1)}
\lesssim  \bM_{c,x}(T)(\bM_{c,x}(T)+\bM_2(T))^2\,,\\
& \sup_{t\in[0,T]}\|\cN^{32}(t)\|_Y+\|\cN^{32}\|_{L^2(0,T;Y)}
\lesssim  \bM_{c,x}(T)(\bM_{c,x}(T)+\bM_2(T))  
  \end{split}
\end{equation}
in exactly the same way as the proof of \eqref{eq:cN3-est}.
To prove the estimate for $\cN^{31}$, we use Claim~\ref{cl:[B3,pdy]}.
\par
Secondly, we estimate $\cN^4$.
Using \eqref{eq:def-B3}, we read $\cN^4$ as
\begin{align*}
\cN^4=& 
B_4^{-1}\left\{\wC_1+\sum_{j=1,2}\pd_y^2(\bS_j-\wS_j)
-\sum_{3\le j\le 5}(\bS_j-\wS_j)\right\}B_3^{-1}(\pd_y^2\wS_0-B_2)
\begin{pmatrix} b_{yy} \\ x_{yy} \end{pmatrix}\,.
\end{align*}
Using the fact that
$$B_4^{-1}=B_1^{-1}-B_1^{-1}\wS_3B_4^{-1}+\pd_y^2B_1^{-1}\wS_1B_4^{-1}\,,
\quad \diag(1,\pd_y)B_1^{-1}\wC_1=\frac12\pd_y\wC_1\,,$$ we have
\begin{align*}
& \diag(1,\pd_y)\cN^4=\diag(\cN^{41}+\pd_y\cN^{42})+\pd_y\cN^{43}\,,\\
& \cN^{41}=\left\{B_1^{-1}\wS_3B_4^{-1}\wC_1+B_4^{-1}\sum_{3\le j\le 5}(\bS_j-\wS_j)\right\}
B_3^{-1}(B_2-\pd_y^2\wS_0)\begin{pmatrix} b_{yy} \\ x_{yy} \end{pmatrix}\,,
\\ &
\cN^{42}=\left\{B_1^{-1}\pd_y\wS_1B_4^{-1}\wC_1
+B_4^{-1}\sum_{j=1,2}\pd_y(\bS_j-\wS_j)B_3^{-1}\right\}
B_3^{-1}(\pd_y^2\wS_0-B_2)\begin{pmatrix} b_{yy} \\ x_{yy} \end{pmatrix}\,,
\\ & \cN^{43}=\frac12\wC_1B_3^{-1}(\pd_y^2\wS_0-B_2)
\begin{pmatrix}b_{yy}\\ x_{yy}\end{pmatrix}\,.
\end{align*}
Note that $[B_4,\pd_y]=0$ and $[\wS_0,\pd_y]=0$.
By Claim~\ref{cl:(1+wT2)^{-1}}, we have
\begin{equation}
  \label{eq:bS-wS}
\|\bS_j-\wS_j\|_{B(Y)}\lesssim \|\tc\|_{L^\infty}\|\wS_j\|_{B(Y)}
\quad\text{for $1\le j\le 5$.}
\end{equation}
By \cite[Claim~B.1]{Mi}, we have $\|\wS_0\|_{B(Y)}\lesssim1$.
Using Claims~\ref{cl:invB_4}, \ref{cl:invB_3}, \eqref{eq:bS3}--\eqref{eq:bS5},
\eqref{eq:bS-wS} and the above, we have
\begin{equation}
  \label{eq:cN41-est}
\sup_{t\in[0,T]}\|\cN^{41}(t)\|_Y+\|\cN^{41}\|_{L^1(0,T;Y)}
\lesssim  \bM_{c,x}(T)(\bM_{c,x}(T)+\bM_2(T))\,.
\end{equation}
By Claim~\ref{cl:invB_4}, \eqref{eq:wS1-est}, \eqref{eq:wS2-est} and \eqref{eq:bS-wS},
\begin{equation}
  \label{eq:cN42-est}
  \begin{split}
& \sup_{t\in[0,T]}(\|\cN^{42}(t)\|_Y+\|\cN^{43}(t)\|_Y)
+\|\cN^{42}\|_{L^2(0,T;Y)}+\|\cN^{43}\|_{L^2(0,T;Y)}
\\ \lesssim & \bM_{c,x}(T)^2\,.
  \end{split}
\end{equation}

\par

A crude estimate $\|\cN^5\|_{L^2(0,T;Y)}\lesssim \bM_1(T)$ is
insufficient to obtain upper bounds of $\bM_{c,x}(T)$.
We decompose $II^6_1$ as $II^6_1=II^6_{11}+\eta^2II^6_{12}-II^6_{13}$,
where
\begin{gather*}
II^6_{11}= 6\int_{\R^2} v_1(t,z,y)\varphi_{c(t,y)}
\overline{\pd_zg_1^*(z,0,c(t,y))}e^{-iy\eta}\,dzdy\,,\\
II^6_{12}=6\int_{\R^2} v_1(t,z,y)\varphi_{c(t,y)}
\overline{\pd_zg_{11}^*(z,\eta,c(t,y))}e^{-iy\eta}\,dzdy\,,
\\ II^6_{13}=6\int_{\R^2} v_1(t,z,y)\tpsi_{c(t,y)}(z)
\overline{\pd_zg_1^*(z,\eta,c(t,y))}e^{-iy\eta}\,dzdy\,.
\end{gather*}
By the fact that $g_1^*(z,0,c)=\frac12\varphi_c$ and \eqref{eq:B},
\begin{align*}
II^6_{11}=\frac{1}{2}\int_{\R^2} \{(\pd_z^3-2c(t,y)\pd_z)v_1(t,z,y)\}
\varphi_{c(t,y)}(z)e^{-iy\eta}\,dzdy\,.
\end{align*}
Substituting \eqref{eq:v1} into the above, we have
\begin{align*}
&  II^6_{11}+\frac12\frac{d}{dt}\int_{\R^2}v_1(t,z,y)
\varphi_{c(t,y)}(z)e^{-iy\eta}\,dzdy
\\=&  -\frac{3}{2}\int_{\R^2} \pd_z^{-1}\pd_y^2v_1(t,z,y)
\varphi_{c(t,y)}(z)e^{-iy\eta}\,dzdy
\\ & -\frac12\int_{\R^2} (N_{1,1}+N_{1,2})\varphi_{c(t,y)}'(z)e^{-iy\eta}\,dzdy
\\ & +\frac12\int_{\R^2} N_{1,3}\varphi_{c(t,y)}(z)e^{-iy\eta}\,dzdy
+\frac12\int_{\R^2}v_1(t,z,y)c_t(t,y)\pd_c\varphi_{c(t,y)}(z)e^{-iy\eta}\,dzdy\,.
\end{align*}
Let
$$S^7_1[q_c](f)(t,y)
=\frac{1}{4\pi}\int_{-\eta_0}^{\eta_0}
\int_{\R^2} v_1(t,z,y_1)f(y_1)q_{c(t,y_1)}(z)e^{i(y-y_1)\eta}\,dzdy_1d\eta\,,$$
\begin{equation}
  \label{eq:der9}
k(t,y)=\frac{1}{4\pi}\int_{-\eta_0}^{\eta_0}\int_{\R^2}
v_1(t,z,y_1)\varphi_{c(t,y_1)}(z)e^{i(y-y_1)\eta}\,dzdy_1d\eta\,.
\end{equation}

By integration by parts, we have
\begin{equation}
\label{eq:der8}
  \begin{split}
& \mathbf{1}_{[-\eta_0,\eta_0]}(\eta) \left\{II^6_{11}(t,\eta)
+\frac12\frac{d}{dt}
\int_{\R^2}v_1(t,z,y)\varphi_{c(t,y)}(z)e^{-iy\eta}\,dzdy\right\}
\\=& \sqrt{2\pi}\mF_y\left\{S^7_1[\pd_c\varphi_c](c_t)
-S^7_1[\varphi_c']\left(x_t-2c-3(x_y)^2\right))\right\}
+II^6_{111}(t,\eta)+i\eta II^6_{112}(t,\eta)\,,    
  \end{split}
\end{equation}
where
\begin{multline*}
II^6_{111}(t,\eta)
= \frac32\int_{\R^2}v_1(t,z,y)^2\varphi_{c(t,y)}'(z)e^{-iy\eta}\,dzdy
\\  
+\frac32\int_{\R^2}(\pd_z^{-1}\pd_yv_1)(t,z,y)c_y(t,y)\pd_c\varphi_{c(t,y)}(z)
e^{-iy\eta}\,dzdy 
\\  -\frac32\int_{\R^2}v_1(t,z,y)\left\{x_{yy}(t,y)\varphi_{c(t,y)}(z)
+2(c_yx_y)(t,y)\pd_c\varphi_{c(t,y)}(z)\right\}e^{-iy\eta}\,dzdy\,,
\end{multline*}
\begin{align*}
II^6_{112}(t,\eta)= & -\frac32\int_{\R^2}(\pd_z^{-1}\pd_yv_1)(t,z,y)
\varphi_{c(t,y)}(z)e^{-iy\eta}\,dzdy
\\ & +3\int_{\R^2}v_1(t,z,y)x_y(t,y)\varphi_{c(t,y)}(z)e^{-iy\eta}\,dzdy\,.
\end{align*}
Let 
\begin{align*}
R^{v_1}_{11}=&
\frac{1}{2\pi}\int_{-\eta_0}^{\eta_0}\left\{II^6_{111}(t,\eta)-II^6_{13}(t,\eta)
\right\}e^{iy\eta}\,d\eta\,,\\
R^{v_1}_{12}=&
\frac{1}{2\pi}\int_{-\eta_0}^{\eta_0}\left\{II^6_{112}(t,\eta)-i\eta
II^6_{12}(t,\eta)\right\}e^{iy\eta}\,d\eta\,.
\end{align*}
Then 
\begin{align*}
R^{v_1}_1=& \frac{1}{2\pi}\int_{-\eta_0}^{\eta_0}II^6_1(t,\eta)e^{iy\eta}\,d\eta
\\=& S^7_1[\pd_c\varphi_c](c_t)-S^7_1[\varphi_c'](x_t-2c-3(x_y)^2)
-\pd_tk+R^{v_1}_{11}+\pd_yR^{v_1}_{12}\,.
\end{align*}
\par
Combining the above with \eqref{eq:invB3-invB1} and \eqref{eq:lwcomp}, we have
\begin{align*}
& \diag(1,\pd_y)\cN^5=\diag(1,\pd_y)(\cN^{51}+\pd_y\cN^{52})+\pd_y\cN^{53}\,,\\
& \cN^{51}=  B_3^{-1}
\begin{pmatrix}
R^{v_1}_{11}+S^7_1[\pd_c\varphi_c](c_t)-S^7_1[\varphi_c'](x_t-2c-3(x_y)^2) \\ 0
\end{pmatrix}
\\ &\enskip
+[B_3^{-1},\pd_y] \begin{pmatrix} R^{v_1}_{12} \\ 0\end{pmatrix}
+B_1^{-1}\sum_{3\le i\le 5}\bS_jB_3^{-1}
  \begin{pmatrix} 0 \\ R_2^{v_1}  \end{pmatrix}
+\left[\pd_t\,,\,B_3^{-1}\right]
\begin{pmatrix} k \\ 0 \end{pmatrix}
\,,\\ 
& \cN^{52}= B_3^{-1} \begin{pmatrix} R^{v_1}_{12} \\ 0\end{pmatrix}
-B_1^{-1}\pd_y(\bS_1+\bS_2)B_3^{-1}
\begin{pmatrix}0 \\ R^{v_1}_2\end{pmatrix}\,,
\\
& \cN^{53}= \frac12\left(E_2-\wC_1B_3^{-1}\right)
  \begin{pmatrix} 0 \\ R_2^{v_1}  \end{pmatrix}\,.
\end{align*}
Then
$$\diag\cN^5=\diag(1,\pd_y)\left\{\cN^{51}+\pd_y\cN^{52}
-\pd_tB_3^{-1}\begin{pmatrix}  k \\ 0 \end{pmatrix}
\right\}+\pd_y\cN^{53}\,.$$
By Lemma~\ref{lem:cx_t-bound} and Claim~\ref{cl:S7},
\begin{align*}
& \|S^7_1[\pd_c\varphi_c](c_t)\|_{L^1(0,T;Y_1)}
+\|S^7_1[\varphi_c'](x_t-2c-3(x_y)^2)\|_{L^1(0,T;Y_1)}
\\ \lesssim  & \|v_1\|_{L^2(0,T;W(t))}
\left(\|c_t\|_{L^2(0,T;L^2(\R))}+\|x_t-2c-3(x_y)^2\|_{L^2(0,T;L^2(\R))}\right)
\\ \lesssim & \bM_1(T)
(\bM_{c,x}(T)+\bM_1(T)+\bM_2(T)^2)\,,
\end{align*}
and
\begin{align*}
& \sup_{t\in[0,T]}\|S^7_1[\pd_c\varphi_c](c_t)\|_{Y_1}
+ \sup_{t\in[0,T]}\|S^7_1[\varphi_c'](x_t-2c-3(x_y)^2)\|_{Y_1}
\\ \lesssim  & \sup_{t\in[0,T]}\left\{
\|v_1(t)\|_{L^2(R^2)}
\left(\|c_t\|_{L^2(0,T;L^2(\R))}+\|x_t-2c-3(x_y)^2\|_{L^2(0,T;L^2(\R))}\right)\right\}
\\ \lesssim & \bM_1(T)(\bM_{c,x}(T)+\bM_1(T)+\bM_2(T)^2)\,.
\end{align*}
Combining the above with Claims~\ref{cl:invB_3}, \ref{cl:[B3,pdy]},
\ref{cl:B3-comm}, \ref{cl:Rv1}, \ref{cl:k-est}, \eqref{eq:bS3} and \eqref{eq:bS5},
we have
\begin{equation}
  \label{eq:cN51-est}
  \begin{split}
& \sup_{t\in[0,T]}\|\cN^{51}\|_{Y_1}+\|\cN^{51}\|_{L^1(0,T;Y_1)}
\\ \lesssim & (e^{-\a L}+\bM_{c,x}(T)+\bM_1(T)+\bM_2(T))\bM_1(T)\,,    
  \end{split}
\end{equation}
\begin{equation}
\label{eq:cN52-est}
 \sup_{t\in[0,T]}(\|\cN^{52}\|_Y+\|\cN^{53}\|_Y)
+\|\cN^{52}\|_{L^2(0,T;Y)}+\|\cN^{53}\|_{L^2(0,T;Y)}
\lesssim \bM_1(T)\,.
\end{equation}
\par

Let 
\begin{align*}
 \cN'= & \Pi_*^{-1}(D_y)\diag(1,\pd_y)\sum_{2\le j\le5}\cN^{j1}\,,
\\ \cN''= & \widetilde{\cN}^1+\widetilde{\cN}^2+
\Pi_*^{-1}(D_y)\diag(1,\pd_y)(\cN^{32}+\cN^{42}+\cN^{52})
\\ & +\Pi_*^{-1}(D_y)(\cN^{22}+\cN^{43}+\cN^{53})\,,
\\  K=&\begin{pmatrix}  K_1 \\ K_2\end{pmatrix}
=\Pi_*^{-1}(D_y)\diag(1,\pd_y)B_3^{-1}\begin{pmatrix} k\\ 0 \end{pmatrix}\,,
\\ \widetilde{\cN}''=&\cN''+\{(\omega(D_y)\sigma_3-4+\pd_y^{-1}A_2(D_y)\}\bb\,.
\end{align*}
Then we have from \eqref{eq:bd} and \eqref{eq:nl},
\begin{equation}
  \label{eq:bb-new}
\pd_t(\bb+K)=2\pd_y^2\bb+4\pd_y\sigma_3\bb+A_3(t,D_y)\bb+\cN'
+\pd_y(\cN^0+\widetilde{\cN}'')\,,
\end{equation}
and \eqref{eq:N'N''} follows from
\eqref{eq:wcN1-est}--\eqref{eq:N2-est}, \eqref{eq:cN31-32-est}, 
\eqref{eq:cN41-est}, \eqref{eq:cN42-est}, \eqref{eq:cN51-est} and \eqref{eq:cN52-est}.
Claims~\ref{cl:invB_3} and \ref{cl:k-est} imply
\begin{equation}
  \label{eq:K-est}
\sup_{t\in[0,T]}\|K(t,\cdot)\|_Y+\|K\|_{L^2(0,T;Y)}\lesssim \bM_1(T)\,,
\quad \lim_{t\to\infty}\|K(t,\cdot)\|_Y=0\,.
\end{equation}
By \eqref{eq:omega-approx}, \eqref{eq:A2-size} and \eqref{eq:N'N''},
\begin{equation}
  \label{eq:wcN''-est}
  \begin{split}
& \sup_{t\in[0,T]}\|\widetilde{\cN}''(t)\|_Y+
\|\widetilde{\cN}''\|_{L^2(0,T;Y)}
\\ \lesssim & \eta_0\bM_{c,x}(T)+\bM_1(T)+\bM_{c,x}(T)^2+\bM_2(T)^2\,.    
  \end{split}
\end{equation}
\par
Time global bound for $\|\bb(t)\|_Y$ does not follow directly
from the energy identity of \eqref{eq:bb-new} because the $L^2(\R)$-inner
product of $\pd_y\cN^0$ and $\bb$ is not necessarily integrable globally
in time for $v_0$ that is not strongly localized in space.
To eliminate cubic nonlinear terms in the energy identity,
we make use of the following change of variables.
\begin{equation}
  \label{eq:btod}
\bd=\begin{pmatrix} d_1 \\ d_2 \end{pmatrix}
=\bb-\frac{1}{2}(b_1+K_1)(b_2+K_2)\mathbf{e_1}+K\,,
\quad \mathbf{e_1}=\begin{pmatrix} 1 \\ 1 \end{pmatrix}\,.  
\end{equation}
By \eqref{eq:btod}, Eq.~\eqref{eq:bb-new} can be rewritten as
\begin{equation}
  \label{eq:bb-new2}
  \begin{split}
\pd_t\bd=& 2\pd_y^2\bb+4\pd_y\sigma_3\bb+A_3(t,D_y)\bb+\cN'
+\pd_y(\cN^0+\widetilde{\cN}'')
\\ & 
-\{2\la \pd_y\sigma_3\bb,\sigma_1\bb\ra+\mathcal{R}_1+\mathcal{R}_2\}
\mathbf{e}_1\,,
  \end{split}
\end{equation}
where $\la\cdot,\cdot\ra$ denotes the inner product of $\R^2$ and 
\begin{gather*}
\sigma_1=\begin{pmatrix}  0 & 1 \\ 1 & 0\end{pmatrix}\,,\quad   
\mathcal{R}_1=
\frac12\la \sigma_1\bb, \pd_y(\cN^0+\widetilde{\cN}'')\ra\,,
\\
\mathcal{R}_2=\frac12
\pd_t\left\{(b_1+K_1)(b_2+K_2)\right\}-2\la \pd_y\sigma_3\bb,\sigma_1\bb\ra
-\mathcal{R}_1,.
\end{gather*}
Taking the $L^2(\R)$-inner product of \eqref{eq:bb-new2} and $\bd$, we have
\begin{equation}
  \label{eq:energy-id1}
\begin{split}
&\frac12\frac{d}{dt}\|\bd(t)\|_{L^2(\R)}^2+2\|\pd_y\bb(t)\|_{L^2(\R)}^2
\\=& \int_\R \la \pd_y\sigma_3\bb,4\bb-2b_1b_2\mathbf{e_1}\ra\,dy
-2\int_\R \la \pd_y\sigma_3\bb,\sigma_1\bb\ra\la \bb,\mathbf{e_1}\ra\,dy
\\ & +\int_\R \la \pd_y\cN^0,\bb\ra\,dy 
+\mathfrak{R}_1+\mathfrak{R}_2+\mathfrak{R}_3
\\=& \mathfrak{R}_1+\mathfrak{R}_2+\mathfrak{R}_3\,,
\end{split} 
\end{equation}
where
\begin{align*}
& \mathfrak{R}_1=
\int_\R \left\{2\la \pd_y\bd,\pd_y(\bb-\bd)\ra+4\la \pd_y\sigma_3\bb,K\ra
+\pd_y\la \pd_y\bb,\sigma_1\bb\ra\right\}\,dy\,,\\
& \mathfrak{R}_2=
\int_\R \la A_3(t,D_y)\bb+\cN'-\mathcal{R}_2\mathbf{e_1},\bd\ra\,dy\,,\\
& \mathfrak{R}_3= 
\int_\R \left\{ \la \pd_y\cN^0,\bd-\bb\ra
+\la \pd_y\widetilde{\cN}'',\bd\ra
-2\la \pd_y\sigma_3\bb,\sigma_1\bb\ra\la \bd-\bb,\mathbf{e_1}\ra
-\mathcal{R}_1\la\mathbf{e_1},\bd\ra
\right\}\,dy\,.
\end{align*}
Since
\begin{align*}
& \la \pd_y\sigma_3\bb,4\bb-2b_1b_2\mathbf{e_1}\ra
-2\la \pd_y\sigma_3\bb,\sigma_1\bb\ra\la \mathbf{e_1},\bb\ra
+\la \pd_y\cN^0,\bb\ra
\\=& 2\pd_y\la \sigma_3\bb,\bb\ra+\pd_y\la \cN^0,\bb\ra
-\frac{4}{3}\pd_y(b_1^3-b_2^3)\,,
\end{align*}
it follows from \eqref{eq:energy-id1} 
\begin{equation}
  \label{eq:energy-ineq1}
\sup_{t\in[0,T]}\|\bd(t)\|_{L^2}^2+4\int_0^T\|\pd_y\bb(t)\|_Y^2\,dt
\lesssim \|v_0\|_{L^2}^2+\sum_{1\le j\le 3}\|\mathfrak{R}\|_{L^1(0,T)}\,.
\end{equation}
Here we use the fact that $\bb(0,\cdot)\equiv0$ and 
$\|\bd(0)\|_Y=O(\|K(0)\|_Y)=O(\|v_0\|_{L^2})$.
\par
Now we will estimate each term of the right hand side of
\eqref{eq:energy-ineq1}.
By Claim~\ref{cl:k-est} and the fact that
$\supp \widehat{b_i}(t,\eta)\subset[-\eta_0,\eta_0]$,
\begin{equation}
  \label{eq:d-b-1}
  \begin{split}
\sup_{t\in[0,T]}\|\bb(t)-\bd(t)\|_{L^2(\R)}
\lesssim &  \sup_{t\in[0,T]}\left(\|\bb(t)\|_Y^2+\|K(t)\|_Y\right)
\\ \lesssim & \bM_{c,x}(T)^2+\bM_1(T)\,,
  \end{split}
\end{equation}
and for $k\ge1$,
\begin{equation}
  \label{eq:d-b-2}
  \begin{split}
& \|\pd_y^k\bb-\pd_y^k\bd\|_{L^2(0,T;L^2(\R))}
\\ \lesssim &  \|\bb\|_{L^\infty(0,T;Y)}\|\pd_y\bb\|_{L^2(0,T;Y)}
+\|K(t)\|_{L^2(0,T;Y)}
\\ \lesssim & \bM_{c,x}(T)^2+\bM_1(T)\,.
  \end{split}
\end{equation}
In view of \eqref{eq:K-est} and \eqref{eq:d-b-2},
\begin{equation*}
\left\|\int_\R \la \pd_y\bd,\pd_y(\bd-\bb)\ra\,dy\right\|_{L^1(0,T)}    
\lesssim \bM_{c,x}(T)^3+\bM_{c,x}(T)\bM_1(T)+\bM_1(T)^2\,,
\end{equation*}
$$\left\|\la \pd_y\sigma_3\bb, K\ra\right\|_{L^1(0,T;Y)}
\lesssim \bM_{c,x}(T)\bM_1(T)\,,$$
and it follows that
\begin{equation}
  \label{eq:fR1-est}
\|\mathfrak{R}_1\|_{L^1(0,T)}\lesssim
\bM_{c,x}(T)^3+\bM_{c,x}(T)\bM_1(T)+\bM_1(T)^2\,.
\end{equation}
\par
Substituting \eqref{eq:bb-new} into $\mathcal{R}_2$, we see that
\begin{align*}
\|\mathcal{R}_2\|_{Y_1}
\lesssim & \|\pd_y\bb\|_Y^2+\|\bb\|_Y(\|A_3(t,D_y)\bb\|_Y+\|\cN'\|_Y)
\\ &
+\|K\|_Y(\|\pd_y\bb\|_Y+\|A_3(t,D_y)\bb\|_Y
+\|\cN^0\|_Y+\|\cN'\|_Y+\|\widetilde{\cN}''\|_Y)\,.
\end{align*}
Combining the above with \eqref{eq:A3-size}, \eqref{eq:N'N''},
\eqref{eq:K-est} and \eqref{eq:wcN''-est},
we have 
\begin{equation*}
\|\mathcal{R}_2\|_{L^1(0,T;Y_1)} \lesssim
\bM_{c,x}(T)^2+\bM_1(T)^2+(\bM_{c,x}(T)+\bM_1(T))\bM_2(T)^2\,,
\end{equation*}
and
\begin{equation}
  \label{eq:fR2-est}
\|\mathfrak{R}_2\|_{L^1(0,T)}\lesssim 
(e^{-\a L}+\bM_{c,x}(T))\bM_{c,x}(T)^2+\bM_1(T)^2+(\bM_{c,x}(T)+\bM_1(T))\bM_2(T)^2\,.
\end{equation}
\par

Using the Sobolev inequality,
we have for $j_1$, $j_2$, $j_3$, $j_4=1\,,\,2$,
\begin{equation}
  \label{eq:quart-L2}
\left\|\int_\R \pd_yb_{j_1}b_{j_2}b_{j_3}b_{j_4}\,dy\right\|_{L^1(0,T)}
\lesssim \|\pd_y\bb\|_{L^2(0,T)}^2\|\bb\|_{L^\infty(0,T;Y)}^2
\lesssim \bM_{c,x}(T)^4\,.
\end{equation}
By \eqref{eq:K-est} and \eqref{eq:quart-L2},
\begin{equation}
  \label{eq:N0,d-b}
\left\|\int_\R \la\pd_y\cN^0,\bd-\bb\ra\,dy \right\|_{L^1(0,T)}
\lesssim \bM_{c,x}(T)^4+\bM_{c,x}(T)^2\bM_1(T)\,.
\end{equation}
By \eqref{eq:K-est} and \eqref{eq:wcN''-est},
\begin{equation}
  \label{eq:cN''-d}
  \begin{split}
& \left\|\int_\R \la\pd_y\widetilde{\cN}'',\bd\ra\,dy\right\|_{L^1(0,T)}
= \left\|\int_\R \la\widetilde{\cN}'',\pd_y\bd\ra\,dy\right\|_{L^1(0,T)}
\\ \lesssim & \{\bM_1(T)+(\eta_0+\bM_{c,x}(T))\bM_{c,x}(T)+\bM_2(T)^2\}
(\bM_{c,x}(T)+\bM_1(T))\,,    
  \end{split}
\end{equation}
and
\begin{equation}
  \label{eq:R1-d}
  \begin{split}
& \left\|\int_\R \mathcal{R}_1 \la \mathbf{e_1},\bd\ra\,dy\right\|_{L^1(0,T)}
\\ \lesssim & \{\bM_1(T)+(\eta_0+\bM_{c,x}(T))\bM_{c,x}(T)+\bM_2(T)^2\}
\bM_{c,x}(T)(\bM_{c,x}(T)+\bM_1(T))\,.    
  \end{split}
\end{equation}
By \eqref{eq:K-est} and \eqref{eq:d-b-1},
\begin{equation}
\label{eq:bbd-b}
\left\|\int_\R \pd_y\sigma_3\bb,\sigma_1\bb\ra\la \bd-\bb,\mathbf{e_1}\ra\,dy \right\|_{L^1(0,T)}
\lesssim \bM_{c,x}(T)^2(\bM_{c,x}(T)+\bM_1(T))\,.
\end{equation}
It follows from \eqref{eq:N0,d-b}--\eqref{eq:bbd-b} that
\begin{equation}
  \label{eq:fR3-est}
  \begin{split}
  \|\mathfrak{R}_3\|_{L^1(0,T)}\lesssim & (e^{-\a L}+\bM_{c,x}(T))\bM_{c,x}(T)^2
+\bM_{c,x}(T)\bM_1(T)+\bM_1(T)^2
\\ & + (\bM_{c,x}(T)+\bM_1(T))\bM_2(T)^2\,.    
  \end{split}
\end{equation}
Combining \eqref{eq:energy-ineq1} with \eqref{eq:d-b-1},
\eqref{eq:fR1-est}, \eqref{eq:fR2-est} and \eqref{eq:fR3-est},
we obtain \eqref{eq:M1-bound}.
This completes the proof of Lemma~\ref{lem:M1-bound}.
\end{proof}
\bigskip

\section{The $L^2(\R^2)$ estimate}
\label{sec:L2norm}
In this section, we will estimate $\bM_v(T)$
assuming smallness of $\bM_{c,x}(T)$, $\bM_1(T)$ and $\bM_2(T)$.
\begin{lemma}
  \label{lem:M4-bound}
Let $\a\in(0,1)$ and $\delta_4$ be as in Lemma~\ref{lem:M1-bound}.
Then there exists a positive constant $C$ such that 
$$\bM_v(T)\le  C(\|v_0\|_{L^2(\R^2)}+\bM_{c,x}(T)+\bM_1(T)+\bM_2(T))\,.$$
\end{lemma}

To prove Lemma~\ref{lem:M4-bound}, we will show a variant of the
$L^2$ conservation law on $v$ as in \cite[Lemma~8.1]{Mi}.
\begin{lemma}
  \label{lem:L2conserve}
Let $\a\in(0,2)$ and $T>0$. Suppose $v_1\in C([0,T];L^2(\R^2))$,
$v_2\in C([0,T];X\cap L^2(\R^2))$ and that $v_2(t)$, $c(t)$ and $x(t)$
satisfy \eqref{eq:orth}, \eqref{eq:C1cx} and \eqref{eq:bd-vc}.
Then
$$Q(t,v):=\int_{\R^2}\left\{v(t,z,y)^2
-2\psi_{c(t,y),L}(z+3t)v(t,z,y)\right\}\,dzdy$$
satisfies for $t\in[0,T]$,
\begin{multline*}
Q(t,v)= Q(0,v)+
2\int_0^t\int_{\R^2}\left(\ell_{11}+\ell_{12}
+6\varphi_{c(s,y)}'(z)\tpsi_{c(s,y)}(z)\right)v(s,z,y)\,dzdyds
\\
-6\int_0^t\int_{\R^2}(\pd_z^{-1}\pd_yv)(s,z,y)c_y(s,y)\pd_c\varphi_{c(t,y)}(z)\,dzdy
\\-6\int_0^t\int_{\R^2}\varphi_{c(s,y)}'(z)v(s,z,y)^2\,dzdyds
 -2\int_0^t\int_{\R^2} \ell\psi_{c(s,y),L}(z+3s)\,dzdyds\,.
\end{multline*}
\end{lemma}
\begin{proof}
Let
$$
\ell_{13}^*=c_{yy}(s,y)\int_{-\infty}^z\pd_c\varphi_{c(s,y)}(z_1)\,dz_1
+c_y(s,y)^2\int_{-\infty}^z \pd_c^2\varphi_{c(s,y)}(z_1)\,dz_1\,.$$
If $v_0\in X$ in addition, then
$$\int_{\R^2} v(t,z,y)\ell_{13}^*\,dzdy
=\int_{\R^2}(\pd_z^{-1}\pd_yv)(t,z,y) c_y(t,y)\varphi_{c(t,y)}\,dzdy\,.$$
Thus we can conclude Lemma~\ref{lem:L2conserve} from \cite[Lemma~8.2]{Mi}
by a limiting argument.
\end{proof}
\par

Now we are in position to prove Lemma~\ref{lem:M4-bound}.
\begin{proof}[Proof of Lemma~\ref{lem:M4-bound}]
Remark~\ref{rem:decomp} and Proposition~\ref{prop:continuation} tell us
that we can apply Lemma~\ref{lem:L2conserve} for $t\in[0,T]$
if $\bM_{c,x}(T)$ and $\bM_2(T)$ are sufficiently small.
\par
Since we have for $j$, $k\ge0$ and $z\in\R$,
\begin{equation}
  \label{eq:exp-localized}
  \pd_z^j\pd_c^k\varphi_c(z)\lesssim e^{-2\a|z|}\,,\quad
\int_{-\infty}^z\pd_c^j\varphi_c(z_1)dz_1\lesssim \min(1,e^{2\a z})\,,
\end{equation}
it follows that
\begin{equation}
  \label{eq:l2pf1}
\begin{split}
& \sup_{[0,T]}\left|\int_0^t\int_{\R^2}
(\ell_{11}+\ell_{12})v\,dzdyds\right|
\\  \lesssim & \bigl(\|c_t-6c_yx_y\|_{L^2((0,T)\times \R)}
+\|x_t-2c-3(x_y)^2\|_{L^2((0,T)\times \R)}
\\ & +\|x_{yy}\|_{L^2((0,T)\times \R)}\bigr)
(\|v_1\|_{L^2(0,T;W(t))}+\|v_2\|_{L^2(0,T;X)})\,,
\end{split}  
\end{equation}
\begin{equation}
  \label{eq:l2pf2}
  \begin{multlined}
\sup_{[0,T]}\left|\int_0^t\int_{\R^2}
c_y(s,y)\pd_c\varphi_{c(s,y)}(\pd_z^{-1}\pd_yv)(s,z,y)\,dzdyds\right|
\\  \lesssim \|c_y\|_{L^2((0,T)\times \R)}
(\|\pd_z^{-1}\pd_yv_1\|_{L^2(0,T;W(t))}
+\|\pd_z^{-1}\pd_yv_2\|_{L^2(0,T;X)})\,,   
  \end{multlined}
\end{equation}
\begin{equation}
  \label{eq:l2pf3}
\sup_{[0,T]}\left|\int_0^t\int_{\R^2}
 \varphi_{c(s,y)}'(z)v^2(s,z,y)\,dzdyds\right| 
\lesssim (\|v_1\|_{L^2(0,T;W(t))}+\|v_2\|_{L^2(0,T;X)})^2\,.
\end{equation}
In view of the definition of $\tpsi$,
\begin{equation}
\label{eq:psinorm}
\begin{split}
& \|\tpsi_{c(t,y)}\|_X \lesssim \|\tc\|_Ye^{-\a(3t+L)}\,,\\
& \|\tpsi_{c(t,y)}\|_{L^2(\R^2)}=2\sqrt{2}\|\sqrt{c}-\sqrt{2}\|_{L^2(\R)}
\|\psi\|_{L^2(\R)}\lesssim\|\tc\|_Y\,.  
\end{split}
\end{equation}
By \eqref{eq:exp-localized} and \eqref{eq:psinorm},
\begin{equation}
  \label{eq:l2pf4}
\begin{split}
& \sup_{[0,T]}\left|\int_0^t\int_{\R^2}
\varphi_{c(s,y)}'(z)\tpsi_{c(s,y)}(z)v(s,z,y)\,dzdyds\right|
\\ \lesssim & \|\tpsi_{c(t,y)}\|_{L^2(0,T;X)}
\|e^{-\a|z|}v(t)\|_{L^2(0,T;L^2(\R^2))}
\\ \lesssim & e^{-\a L}\sup_{t\in[0,T]}\|\tc(t)\|_Y
 (\|v_1\|_{L^2(0,T;W(t))}+\|v_2\|_{L^2(0,T;X)})\,,
\end{split}  
\end{equation}
\begin{equation}
  \label{eq:l2pf5}
\begin{split}
& \sup_{[0,T]}\left|\int_0^t\int_{\R^2}
(\ell_{11}+\ell_{12})\tpsi_{c(s,y)}(z)\,dzdyds\right|
\\ \le & \sup_{t\in[0,T]}\|e^{-\a z}(\ell_{11}+\ell_{12})\|_{L^2_{yz}}
\|\tpsi_{c(t,y)}\|_{L^1(0,T;X)}
\\ \lesssim & e^{-\a L}
\sup_{t\in[0,T]}\left\{\|\tc\|_Y
\left(\|c_t-6c_yx_y\|_{L^2}+\|x_t-2c-3(x_y)^2\|_{L^2}+\|x_{yy}\|_{L^2}\right)
\right\}\,.
\end{split}  
\end{equation}
By integration by parts, we have
\begin{align*}
& \int_{\R^2}(\ell_{21}+\ell_{22})\tpsi_{c(t,y)}(z)\,dzdy
\\=&\int_{\R^2}\left(c_t(t,y)\tpsi_{c(t,y)}(z)\pd_c\tpsi_{c(t,y)}(z)
+3\varphi_{c(t,y)}'(z)\tpsi_{c(t,y)}^2(z)\right)\,dzdy\,,  
\end{align*}
and it follows that 
\begin{equation}
\label{eq:l2pf6}
\begin{split}
& \sup_{t\in[0,T]}\left|
\int_0^t\int_{\R^2}(\ell_{21}+\ell_{22})\tpsi_{c(s,y)}(s,z,y)dzdyds
-\frac12\left[\int_{\R^2}\tpsi_{c(s,y)}^2(z)\,dzdy\right]_{s=0}^t\right|
\\ \le & 3\left\|\varphi_{c(t,y)}'(z)\tpsi_{c(t,y)}(z)\right\|_{L^1(0,T;L^1(\R^2))}
 \lesssim  e^{-\a L}\sup_{t\in[0,T]}\|\tc(t)\|_Y^2\,.
\end{split}
\end{equation}
By integration by parts,
\begin{align*}
& \int_{\R^2}(\ell_{13}+\ell_{23})\tpsi_{c(t,y)}(z)\,dzdy
\\=& -3\int_{\R^2} c_y^2(t,y)\pd_c\tpsi_{c(t,y)}(z)
\left\{\int_z^\infty\pd_c\varphi_{c(t,y)}(z_1)
-\pd_c\tpsi_{c(t,y)}(z_1)\,dz_1\right\}\,dzdy\,.
\end{align*}
Since $\int_z^\infty(\pd_c\varphi_c-\pd_c\tpsi_c)$ and
$\|\pd_c\tpsi_c\|_{L^1(\R)}$ are uniformly bounded for $c\in[1/2,3/2]$,
\begin{equation}
\label{eq:l2pf7}
\sup_{t\in[0,T]}
\left|\int_0^t\int_{\R^2}(\ell_{13}+\ell_{23})\tpsi_{c(s,y)}\,dzdyds\right|
\lesssim \|c_y\|_{L^2(0,T;Y)}^2\,.
\end{equation}
\par
Combining \eqref{eq:l2pf1}--\eqref{eq:l2pf3} and
\eqref{eq:l2pf4}--\eqref{eq:l2pf7} with Lemmas~\ref{lem:cx_t-bound} and
\ref{lem:L2conserve},
we see that for $t\in(0,T]$,
\begin{equation}
  \label{eq:cons1}
\begin{split}
& \Bigl[Q(s,v)+8\|\psi\|_{L^2}^2\|\sqrt{c(s)}-\sqrt{2}\|_{L^2(\R)}^2
\Bigr]_{s=0}^{s=t}
\\ \lesssim & (\bM_{c,x}(T)+\bM_1(T)+\bM_2(T))^2\,.
\end{split}  
\end{equation}
Since $c(0,\cdot)=2$ and 
$$Q(t,v)=\|v(t)\|_{L^2(\R^2)}^2+O(\|\tc(t)\|_Y\|v(t)\|_{L^2(\R^2)})\,,$$
Lemma~\ref{lem:M4-bound} follows immediately from \eqref{eq:cons1}.
Thus we complete the proof.
\end{proof}
\bigskip

\section{Estimates for $v_1$}
\label{sec:v1}
In this section, we will give upper bounds of $\bM_1(\infty)$ and
$\bM_1'(\infty)$.
\begin{lemma}
  \label{lem:v1-a}
There exist positive constants $C$ and $\delta_5$ such that if
$\|v_0\|_{L^2}<\delta_5$, then $\bM_1(\infty) \le C\|v_0\|_{L^2}$
\end{lemma}
\begin{lemma}
  \label{lem:v1-b}
There exist positive constants $C$ and $\delta_5'$ such that if
$\||D_x|^{-1/2}v_0\|_{L^2}+
\linebreak \||D_x|^{1/2}v_0\|_{L^2}+\||D_x|^{-1/2}|D_y|^{1/2}v_0\|_{L^2}
<\delta_5'$, then
$$\bM_1'(\infty) \le  C(
\||D_x|^{-1/2}v_0\|_{L^2}+\||D_x|^{1/2}v_0\|_{L^2}+\||D_x|^{-1/2}|D_y|^{1/2}v_0\|_{L^2}
)\,.$$
\end{lemma}

\subsection{Virial estimates for $v_1$}
The virial identity for $L^2$-solutions of the KP-II equation
\eqref{eq:KPII} was shown in \cite{dBM}.
It ensures $v_1(t)\in L^2([0,\infty);L^2_{loc}(\R^2))$.
Let $\chi_{+,\eps}(x)=1+\tanh\eps x$, $\tilde{x}_1(t)$ be a $C^1$ function and 
$$I_{x_0}(t)=\int_{\R^2}\chi_{+,\eps}(x-\tilde{x}_1(t)-x_0,y)\tv_1^2(t,x,y)\,dxdy\,.$$
Then we have the following.
\begin{lemma}
  \label{lem:virial-0}
  Let $\tv_1(t)$ be a solution   of \eqref{eq:KPII} satisfying $\tv_1(0)=v_0\in L^2(\R^2)$.
  Then for any $c_1>0$, there exist positive constants $\eps_0$ and $\delta$
  such that if $\inf_t\tilde{x}_1'(t)\ge c_1$, $\eps\in(0,\eps_0)$ and
  $\|v_0\|_{L^2}<\delta$, then for any $x_0\in\R$,
$$I_{x_0}(t)
+\nu \int_0^t\int_{\R^2}
\chi_{+,\eps}'(x-\tilde{x}_1(s)-x_0)
\{(\pd_x\tv_1)^2+(\pd_x^{-1}\pd_y\tv_1)^2+\tv_1^2\}(s,x,y)\,dxdyds
\le I_{x_0}(0)\,,$$
where $\nu=\frac{1}{2}\min\{3,c_1\}$. Moreover,
\begin{equation}
  \label{eq:limIx0}
  \lim_{t\to\infty}I_{x_0}(t)=0\quad\text{for any $x_0\in\R$.}
\end{equation}
\end{lemma}
See e.g. \cite[Lemma~5.3]{MT} for a proof.
Lemma~\ref{lem:v1-a} follows from Lemma~\ref{lem:virial-0} and
the $L^2$-conservation law of the KP-II equation.

\subsection{The $L^3$-estimate of $v_1$}
In order to estimate the $L^3$-norm of $v_1$, we apply the 
small data scattering result for the KP-II equation by \cite{HHK}.
\par
For the sake of self-containedness, let us introduce some notations
in \cite{HHK}.
Let $\mathcal{Z}$ be a set of finite partitions $-\infty=t_0<t_1<\cdots<t_K=\infty$.
We denote by $V^p$ ($1\le p<\infty$)  the set of all functions
$v:\R\to L^2(\R^2)$ such that $\lim_{t\to\pm\infty}v(t)$ exist and
for which the norm
$$\|v\|_{V^p}=\left\{\sup_{\{t_k\}_{k=0}^K\in\mathcal{Z}}\sum_{k=1}^K
\|v(t_k)-v(t_{k-1})\|_{L^2(\R^2)}^p\right\}^{1/p}$$
is finite, where $v(-\infty):=\lim_{t\to-\infty}v(t)$ and $v(\infty):=0$.
We denote by $V^p_{-,rc}$ the closed subspace of every right-continuous
function $v\in V^p$ satisfying $\lim_{t\to-\infty}v(t)=0$.
Let $V^p_S:=e^{\cdot S}V^p$ and $V^p_{-rc,-,S}:=e^{\cdot S}V^p$
with $S=-\pd_x^3-3\pd_x^{-1}\pd_y^2$.
\par

Let $\chi\in C_0^\infty(-2,2)$ be an even nonnegative function such that
$\chi(\eta)=1$ for $\eta\in[-1,1]$.
Let $\bar{\chi}(t)=\chi(t)-\chi(2t)$ and $P_N$
be a projection defined by
$\widehat{P_Nu}(\tau,\xi,\eta)=\bar{\chi}(N^{-1}\xi)\hat{u}(\tau,\xi,\eta)$
for $N=2^n$ and $n\in\Z$. For $s\le0$, we denote by $\dot{Y}^s$
the closure of $C(\R;H^1(\R^2))\cap V^2_{-,rc}$ with respect to 
the norm
$$\|u\|_{\dot{Y}^s}=\left(\sum_N N^{2s}\|P_Nu\|_{V_{S}^2}^2\right)^{1/2}\,.$$
We denote by $\dot{Y}^s(0,T)$ the restriction of $\dot{Y}^s$ to the time interval
$[0,T]$ with the norm
\begin{align*}
\|u\|_{\dot{Y}^s(0,T)}=\inf\{\|\tu\|_{\dot{Y}^s}\mid
\text{$\tu\in \dot{Y}^s$, $\tu(t)=u(t)$ for $t\in[0,T]$}\}\,.
\end{align*}
\par

Proposition~3.1 and Theorem~3.2 in \cite{HHK} ensure that
higher order Sobolev norms of a solution to \eqref{eq:v1} remain 
small provided $v_0$ is small in the higher order Sobolev spaces.
Let $T\ge0$ and 
$$I_T(u_1,u_2)(t)=\int_0^t
\mathbf{1}_{[0,T]}(s) e^{(t-s)S}\pd_x(u_1u_2)(s)\,ds\,.$$
Then we have the following.
\begin{lemma}
  \label{lem:bilinear} Let $s\ge0$ and $u_1$, $u_2\in \dot{Y}^{-1/2}$.
Then there exists a positive constant $C$ such that for any
$T\in(0,\infty)$,
\begin{align}
& \label{eq:bilinear-1}
\left\||D_x|^sI_T(u_1,u_2)\right\|_{\dot{Y}^{-1/2}}
\le C\||D_x|^su_1\|_{\dot{Y}^{-1/2}}\|u_2\|_{\dot{Y}^{-1/2}}\,,
\\ & \label{eq:bilinear-2}
\left\|\la D_y\ra^s I_T(u_1,u_2)\right\|_{\dot{Y}^{-1/2}}
\le C\prod_{j=1,2}\|\la D_y\ra^s u_j\|_{\dot{Y}^{-1/2}}\,.
\end{align}
\end{lemma}
\begin{proof}
We have \eqref{eq:bilinear-1} in exactly the same way
as the proof of \cite[Theorem~3.2]{HHK}.
Note that \eqref{eq:bilinear-1} and \eqref{eq:bilinear-2} are the same
with \cite[Corollary~3.4]{HHK} when $s=0$.
Using the fact that $1+\eta_3^2\lesssim (1+\eta_1^2)(1+\eta_2^2)$
for $\eta_1$, $\eta_2$ and $\eta_3$ satisfying $\eta_1+\eta_2+\eta_3=0$,
we can prove \eqref{eq:bilinear-2} in the same way as
Proposition~3.1 and Theorem~3.2 in \cite{HHK}.
\end{proof}

Thanks to Lemma~\ref{lem:bilinear}, we have the following.
\begin{proposition}
  \label{prop:v1-scattering}
There exists a positive constant $\delta_5'$ such that
if $$\||D_x|^{-1/2}v_0\|_{L^2}+\||D_x|^{-1/2}|D_y|^{1/2}v_0\|_{L^2}
\le \delta_5'\,,$$
then a solution $\tv_1$ of \eqref{eq:tv1} satisfies
\begin{equation}
  \label{eq:v1-scat1}
  \begin{split}
& \|\pd_x\tv_1\|_{\dot{Y}^{-1/2}} \lesssim\||D_x|^{1/2}v_0\|_{L^2}\,,\\    
& \|\la D_y\ra^{1/2}\tv_1\|_{\dot{Y}^{-1/2}} \lesssim
\||D_x|^{-1/2}v_0\|_{L^2}+\||D_x|^{-1/2}|D_y|^{1/2}v_0\|_{L^2}\,.
\end{split}
\end{equation}
\end{proposition}

\begin{proof}
Using the variation of constants formula, we have
$$\tv_1(t)=e^{tS}v_0-3I_T(\tv_1(s),\tv_1(s))\,ds
\quad \text{for $t\in[0,T]$.}$$
By Lemma~\ref{lem:bilinear} and the fact that
$\|e^{tS}v_0\|_{\dot{Y}^{-1/2}(0,T)}\lesssim\||D_x|^{-1/2}v_0\|_{L^2}$,
\begin{gather*}
\|\tv_1\|_{\dot{Y}^{-\frac12}(0,T)}\lesssim \||D_x|^{-1/2}v_0\|_{L^2}
+\|\tv_1\|_{\dot{Y}^{-\frac12}(0,T)}^2\,,
\\
\|\pd_x\tv_1\|_{\dot{Y}^{-\frac12}(0,T)}\lesssim \||D_x|^{1/2}v_0\|_{L^2}
+\|\pd_x\tv_1\|_{\dot{Y}^{-\frac12}(0,T)}\|\tv_1\|_{\dot{Y}^{-\frac12}(0,T)}\,,
\\
\|\la D_y\ra^{1/2}\tv_1\|_{\dot{Y}^{-\frac12}(0,T)}\lesssim 
\||D_x|^{-1/2}\la D_y\ra^{1/2}v_0\|_{L^2}
+\|\la D_y\ra^{1/2}\tv_1\|_{\dot{Y}^{-\frac12}(0,T)}^2\,.
\end{gather*}
If $\delta$ is sufficiently small, it follows from the above that
\begin{gather*}
\|\tv_1\|_{\dot{Y}^{-\frac12}(0,T)}\le
C_1\||D_x|^{-1/2}v_0\|_{L^2} +C_2\|\tv_1\|_{\dot{Y}^{-\frac12}(0,T)}^2\,,
\\
\|\la D_y\ra^{1/2}\tv_1\|_{\dot{Y}^{-\frac12}(0,T)}\le
C_1\||D_x|^{-1/2}\la D_y\ra^{1/2}v_0\|_{L^2}
+C_2\|\la D_y\ra^{1/2}\tv_1\|_{\dot{Y}^{-\frac12}(0,T)}^2\,,  
\\
\||D_x|^{1/2}\tv_1\|_{\dot{Y}^0(0,T)} \le C_1\||D_x|^{1/2}v_0\|_{L^2}
+\|\tv_1\|_{\dot{Y}^{-1/2}(0,T)}\||D_x|^{1/2}\tv_1\|_{\dot{Y}^0(0,T)}\,,
\end{gather*}
where $C_1$ and $C_2$ are positive constant independent of $T$.
Suppose $v_0\in H^2(\R^2)$. Then 
$$\|I_T(\tv_1,\tv_1)\|_{\dot{Y}^{-1/2}(0,T)}\,,\quad
\|I_T(\tv_1,\tv_1)\|_{\dot{Y}^0(0,T)}\,,\quad 
\|\la D_y\ra^{1/2}\tv_1\|_{\dot{Y}^{-\frac12}(0,T)}$$
are continuous in $T$ because $\tv_1\in C(\R;H^2(\R^2))$ and 
$$\pd_t\left(e^{-tS}I_T(\tv_1,\tv_1)(t)\right)
=\left\{
  \begin{aligned}
& e^{-tS}\pd_x\tv_1^2(t)\quad& \text{for $t\in[0,T]$,}\\
& 0 \quad & \text{otherwise.}
\end{aligned}\right.$$
Taking the limit $T\to\infty$, we have \eqref{eq:v1-scat1} for
any $v_0\in H^2(\R^2)$ satisfying the assumption in
Proposition~\ref{prop:v1-scattering}. For general $v_0$,
we have \eqref{eq:v1-scat1} by approximating $v_0$ by $H^3(\R^2)$ functions.
Thus we complete the proof.
\end{proof}

Proposition~\ref{prop:v1-scattering} implies the $L^3$-bound of $v_1$.
\begin{proof}[Proof of Lemma~\ref{lem:v1-b}]
By \eqref{eq:v1-scat1}, 
\begin{gather*}
\sup_{t\ge0}\||D_x|^{1/2}\tv_1(t)\|_{L^2}\lesssim \||D_x|^{1/2}v_0\|_{L^2}\,,\\
\sup_{t\ge0}\||D_x|^{-1/2}|D_y|^{1/2}\tv_1(t)\|_{L^2}\lesssim \||D_x|^{-1/2}\la D_y\ra v_0\|_{L^2}\,.
\end{gather*}
Using an isotropic Sobolev imbedding inequality
\begin{equation}
  \label{eq:aniso-L3}
\|u\|_{L^3(\R^2)}\lesssim \||D_x|^{1/2}u\|_{L^2(\R^2)}
+\||D_x|^{-1/2}|D_y|^{1/2}u\|_{L^2(\R^2)}\,,
\end{equation}
we have 
$$\|v_1(t)\|_{L^3}=\|\tv_1(t)\|_{L^3}\lesssim
 \||D_x|^{1/2}\tv_1(t)\|_{L^2}+\||D_x|^{-1/2}|D_y|^{1/2}\tv_1(t)\|_{L^2}\,.$$
Combining the above with \eqref{eq:v1-scat1},
we have Lemma~\ref{lem:v1-b}.
We remark that \eqref{eq:aniso-L3} follows by interpolating
the imbedding theorem $ Id: \dot{E}^1 \to L^6(\R^2)$
(see e.g. \cite[Lemma~2]{MST_KPI}) and $Id: \dot{E}^0\to L^2(\R^2)$, where
$\dot{E}^t$ is a Banach space with the norm 
$$\|u\|_{\dot{E}^t}= \left\|\left(\xi^2+\frac{\eta^2}{\xi^2}\right)^{t/2}
\hat{u}(\xi,\eta)\right\|_{L^2(\R^2)}\,.$$
\end{proof}

\bigskip

\section{Decay estimates in the exponentially weighted space}
\label{sec:X-norm}
In this section, we will estimate $\bM_2(T)$ following
the line of \cite[Chapter~8]{Mi}.
\begin{lemma}
  \label{lem:exp-bound}
Let $\eta_0$ and $\a$ be positive constants satisfying
$\nu_0<\a<2$. Suppose $\bM_1'(\infty)$ is sufficiently small.
Then there exist positive constants $\delta_6$ and $C$
 such that if
  $\bM_{c,x}(T)+\bM_1(T)+\bM_2(T)+\bM_v(T)\le \delta_6$,
\begin{equation}
  \label{eq:M4-bound}
\bM_2(T)\le C(\bM_{c,x}(T)+\bM_1(T))\,.
\end{equation}
\end{lemma}
Let $\chi\in C_0^\infty(-2,2)$ be an even nonnegative function such that
$\chi(\eta)=1$ for $\eta\in[-1,1]$.
Let $\chi_M(\eta)=\chi(\eta/M)$ and
$$P_{\le M}u:=\frac{1}{2\pi}\int_{\R^2}\chi_M(\eta)
\hat{u}(\xi,\eta)e^{i(x\xi+y\eta)}d\xi d\eta,\quad P_{\ge M}=I-P_{\le M}.$$
\par
To prove Lemma~\ref{lem:exp-bound}, we will use linear stability property
of line solitons (Proposition~\ref{prop:semigroup-est}) 
to the low frequency part $v_<(t):=P_{\le M}v_2(t)$ and
make use of a virial type 
estimate for the high frequency part $v_>(t):=P_{\ge M}v_2(t)$.

\subsection{Decay estimates for the low frequency part}
\label{subsec:lowbound}

\begin{lemma}
  \label{lem:nonresonant-ylow}
Let $\eta_0$ and $\a$ be positive constants satisfying $\nu_0<\a<2$.
Suppose that $v_2(t)$ is a solution of \eqref{eq:v2} satisfying $v_2(0)=0$.
Then there exist positive constants  $\delta_6$ and $C$ such that
if $\bM_1(T)+\bM_2(T)<\delta_6$ and $M\ge \eta_0$, then
\begin{equation}
  \label{eq:nonresonant-ylow}
\begin{split}
& \left\|P_1(0,2M)v_2\right\|_{L^\infty(0,T;X)}
+\left\|P_1(0,2M)v_2\right\|_{L^2(0,T;X)}
\\  \le  & C\left\{\bM_{c,x}(T)+\bM_1(T)+\bM_2(T)(\bM_2(T)+\bM_v(T))\right\}\,.
\end{split}  
\end{equation}
\end{lemma}

\begin{proof}[Proof of Lemma~\ref{lem:nonresonant-ylow}]
Let $\tv_2(t)=P_2(\eta_0,2M)v_2(t)$. 
Then 
\begin{equation}
  \label{eq:v1n}
\left\{
  \begin{aligned}
& \pd_t\tv_2=\mL\tv_2+P_2(\eta_0,2M)\{\ell+\pd_x(N_{2,1}+N_{2,2}+N_{2,2}'+N_{2,4})
+N_{2,3}\}\,,
\\ & \tv_2(0)=0\,,
  \end{aligned}\right.
\end{equation}
where $N_{2,2}'=\{2\tc(t,y)+6(\varphi(z)-\varphi_{c(t,y)}(z))\}v_2(t,z,y)$.
Hereafter we abbreviate $P_2(\eta_0,2M)$ as $P_2$.
\par
Applying Proposition~\ref{prop:semigroup-est} to \eqref{eq:v1n},
we have
\begin{equation}
  \label{eq:v1n-est}
  \begin{split}
& \|\tv_2(t)\|_X\lesssim 
\int_0^t e^{-b'(t-s)}(t-s)^{-3/4}\|e^{\a z}P_2N_{2,1}(s)\|_{L^1_zL^2_y}\,ds
\\ & \enskip +\int_0^t e^{-b'(t-s)}(t-s)^{-1/2}
(\|N_{2,2}(s)\|_X+\|N_{2,2}'(s)\|_X+\|N_{2,4}\|_X)\,ds
\\ &\qquad  +\int_0^t e^{-b(t-s)}(\|\ell(s)\|_X+\|N_{2,3}(s)\|_X)\,ds\,.
  \end{split}
\end{equation}
Since
$\|e^{\a z}P_2N_{2,1}\|_{L^1_zL^2_y}\lesssim  \sqrt{M}(\|v_1\|_{L^2}+\|v_2\|_{L^2})
\|v_2\|_X$ by  \cite[Claim~9.1]{Mi}), we have
\begin{equation}
  \label{eq:N2,1-est}
\begin{split}
&  \sup_{t\in[0,T]}\|e^{\a z}P_2N_{2,1}\|_{L^1_zL^2_y}
+\|e^{\a z}P_2N_{2,1}\|_{L^2(0,T;L^1_zL^2_y)}
\lesssim   \sqrt{M}(\bM_1(T)+\bM_v(T))\bM_2(T)\,.
\end{split}
\end{equation}
By the definitions,
\begin{gather*}
\|\ell_1\|_X\lesssim 
\|x_t-2c-3(x_y)^2\|_{L^2}+\|c_t-6c_yx_y\|_{L^2}+\|x_{yy}\|_{L^2}
+\|c_{yy}\|_{L^2}+\|c_y\|_{L^4}^2\,,
\end{gather*}
\begin{align*}
\|\ell_2\|_X\lesssim 
e^{-\a(3t+L)}\bigl(& \|c_t-6c_yx_y\|_{L^2}+\|x_t-2c-3(x_y)^2\|_{L^2}
+\|\tc\|_{L^2}+\|x_{yy}\|_{L^2}
\\ & \enskip +\|c_{yy}\|_{L^2}+\|c_y\|_{L^4}^2\bigr)\,,
\end{align*}
\begin{gather*}
\|N_{2,2}\|_X \lesssim
(\|x_t-2c-3(x_y)^2\|_{L^\infty}+\|\tc\|_{L^\infty})\|v_2\|_X\,,
\\  \|N_{2,2}'\|_X\lesssim \|\tc\|_{L^\infty}\|v_2(t)\|_X\,,\quad
\|N_{2,4}\|_X \lesssim \|v_1(t)\|_{W(t)}\,.
\end{gather*}
Hence it follows from Lemma~\ref{lem:cx_t-bound} and the definitions of
$\bM_{c,x}(T)$, $\bM_1(T)$ and $\bM_2(T)$ that
\begin{equation}
  \label{eq:ell-est}
\sup_{t\in[0,T]}\|\ell\|_X+\|\ell\|_{L^2(0,T;X)}
\lesssim \bM_{c,x}(T)+\bM_1(T)+\bM_2(T)^2\,,
\end{equation}
\begin{equation}
\label{eq:N2,2-est}
\sup_{t\in[0,T]}\|N_{2,2}\|_X+\|N_{2,2}\|_{L^2(0,T;X)}
 \lesssim  (\bM_{c,x}(T)+\bM_1(T)+\bM_2(T)^2)\bM_2(T)\,,
\end{equation}
\begin{equation}
  \label{eq:N2,2'-est}
\begin{split}
&\sup_{t\in[0,T]}(\|N_{2,2}'\|_X+\|N_{2,4}\|_X)
+\|N_{2,2}'\|_{L^2(0,T;X)}+\|N_{2,4}\|_{L^2(0,T;X)}
\\\lesssim & \bM_{c,x}(T)\bM_2(T)+\bM_1(T)\,.  
\end{split}
\end{equation}
Since $\|\pd_yP_2\|_{B(X)}\lesssim M$, we have
$\|P_2N_{2,3}\|_X\lesssim  M(\|x_y\|_{L^\infty}+\|x_{yy}\|_{L^\infty})\|v_2\|_X$ and
\begin{equation}
\label{eq:N2,3-est}
\sup_{t\in[0,T]}\|N_{2,3}\|_X+\|N_{2,3}\|_{L^2(0,T;X)}
\lesssim  \bM_{c,x}(T)\bM_2(T)\,.  
\end{equation}
Combining \eqref{eq:N2,1-est}--\eqref{eq:N2,3-est}
with \eqref{eq:v1n-est}, we have
$$\sup_{t\in[0,T]}\|\tv_2(t)\|_X+\|\tv_2(t)\|_{L^2(0,T;X)}
\lesssim \bM_{c,x}(T)+\bM_1(T)+(\bM_v(T)+\bM_2(T))\bM_2(T)\,.$$
\par
As long as $v_2(t)$ satisfies the orthogonality condition \eqref{eq:orth} and
$\tc(t,y)$ remains small, we have
$$\|\tv_2(t)\|_X\lesssim \|P_1(0,2M)v_2(t)\|_X\lesssim \|\tv_2(t)\|_X$$
in exactly the same way as the proof of Lemma~9.2 in  \cite{Mi}.
Thus we have \eqref{eq:nonresonant-ylow}.
This completes the proof of Lemma~\ref{lem:nonresonant-ylow}.
\end{proof}
\bigskip

\subsection{Virial estimates for $v_2$}
\label{subsec:virial}
Next, we prove a virial type estimate in the weighted space $X$
in order to estimate the high frequency part of $v_>$.
We need the smallness assumption of  $\sup_{t\ge0}\|v_1(t)\|_{L^3(\R^2)}$
to estimate the high frequency part $v_>(t)$.
\begin{lemma}
\label{lem:virial}
Let $\a\in(0,2)$ and $v_2(t)$ be a solution to \eqref{eq:v2} satisfying $v_2(0)=0$.
Suppose $\bM_1'(\infty)$ is sufficiently small.
Then there exist positive constants $\delta_6$, $M_1$ and $C$ such that
if $\bM_{c,x}(T)+\bM_1(T)+\bM_2(T)+\bM_v(T)<\delta_6$ and $M\ge M_1$,
then for $t\in[0,T]$,
\begin{equation*}
\|v_2(t)\|_X^2 \le C\int_0^t e^{-M\a(t-s)}\left(\|\ell(s)\|_X^2+\|P_{\le M}v_2(s)\|_X^2
+\|v_1(s)\|_{W(s)}^2
\right)\,ds\,.
\end{equation*}
\end{lemma}
\begin{proof}[Proof of Lemma~\ref{lem:virial}]
\par
Let $p(z)=e^{2\a z}$.
Multiplying \eqref{eq:v2} by $2p(z)v_2(t,z,y)$ and integrating
the resulting equation by parts, we have for $t\in[0,T]$,
\begin{equation}
  \label{eq:vir1}
  \begin{split}
&  \frac{d}{dt}\left(
\int_{\R^2} p(z)v_2(t,z,y)^2\,dzdy\right) 
+\int_{\R^2}p'(z)\left(\mathcal{E}(v_2)-4v_2^3\right)(t,z,y)\,dzdy
\\ =& \sum_{k=1}^5III_k(t)\,.
  \end{split}  
\end{equation}
where
\begin{align*}
III_1=&2 \int_{\R^2}p(z)\ell v_2(s,z,y)\,dzdyds\,,\\
III_2=&-\int_{\R^2}p'(z)\left((\tx_t(t,y)-3x_y(t,y)^2\right)v_2(t,z,y)^2\,dzdy\,,
\\  III_3=&
\int_{\R^2}\biggl\{p'''(z)+6p(z)^2
\left(\frac{\varphi_{c(t,y)}(z)-\psi_{c(t,y),L}(z+3t)}{p(z)}\right)_z
\biggr\}v_2(t,z,y)^2\,dzdy\,,
\\ III_4=& 12\int_{\R^2}p'(z)(v_1v_2^2)(t,z,y)\,dzdy
+12\int_{\R^2}p(z)(v_1v_2\pd_zv_2)(t,z,y)\,dzdy\,,
\\ III_5=& 12
\int_{\R^2} \pd_z\left(p(z)v_2(t,z,y)\right)
\left(\varphi_{c(t,y)}(z)-\psi_{c(t,y),L}(z+3t)\right)v_1(t,z,y)\,dzdy\,.
\end{align*}
\par
Obviously,
$$\left|III_1\right| \le
\int p'(z) v_2^2\,dzdy +\frac{1}{2\a}\int p(z) \ell^2 \,dzdy\,,$$
$$\left|III_3\right| \le (M-1) \int_{\R^2} p'(z) v_2(t,z,y)^2\,dzdy\,,$$
where
$$M=1+4\a^2+6\sup_{t,y,z}\frac{p^2(z)}{p'(z)}
\left|\left(\frac{\varphi_{c(t,y)}(z)-\psi_{c(t,y),L}(z+3t)}{p(z)}
\right)_z\right|\,,$$
and 
$$III_5 \lesssim \left(\int_{\R^2}p'(z)
\left\{(\pd_zv_2)^2+v_2^2\right\}(t,z,y)\,dzdy\right)^{1/2}\|v_1(t)\|_{W(t)}\,.
$$
Using Claim~\ref{cl:aniso} and the H\"older inequality, we have
\begin{equation*}
\left|\int p'(z)v_2(t,z,y)^3dzdy\right|\lesssim 
\|v_2(t)\|_{L^2}\int_{\R^2}p'(z) \mathcal{E}(v_2(t,z,y))dzdy\,,
\end{equation*}
\begin{equation*}
III_4 \lesssim \|v_1(t)\|_{L^3}
\int_{\R^2} p'(z)\left((\pd_zv_2)^2+(\pd_z^{-1}\pd_yv_2)^2+v_2^2\right)
(t,z,y)\,dzdy\,.
\end{equation*}
By Lemma~\ref{lem:cx_t-bound},
$$\left|III_2\right| \lesssim (\bM_{c,x}(T)+\bM_1(T)+\bM_2(T)^2)
\int_{\R^2} p'(z) v_2(t,z,y)^2
\,.dzdy$$
For $y$-high frequencies, the potential term can be absorbed into the
left hand side. Indeed, it follows from
Plancherel's theorem and the Schwarz inequality that
\begin{align*}
& \int_{\R^2}
p'(z)\left((\pd_zv_>)^2+(\pd_z^{-1}\pd_yv_>)^2\right)(t,z,y)\,dzdy
\\ =& 2\a\int_{\R^2}\left(|\xi+i\a|^2+\frac{\eta^2}{|\xi+i\a|^2}\right)
\left|\mF v_{>}(t,\xi+i\a,\eta)\right|^2\,d\xi d\eta
\\ \ge & 2M\int_{\R^2}p'(z)v_{>}(t,z,y)^2\,dzdy\,.
\end{align*}
Combining the above, we have for $t\in[0,T]$,
\begin{equation}
  \label{eq:vir2}
  \begin{split}
& \frac{d}{dt}\int_\R p(z)v_2(t,z,y)^2\,dzdy +M\a\int_\R p(z)v_2(t,z,y)^2\,dzdy 
\\ \le &
\frac{1}{2\a}\int_{\R^2}p(z)\ell^2\,dzdy
+M\a\int_{\R^2}p(z)(v_<)^2(s,z,y)\,dzdy+O\left(\|v_1(t)\|_{W(t)}^2\right)
  \end{split}
\end{equation}
provided $\delta_6$ is sufficiently small.
Lemma~\ref{lem:virial} follows immediately from \eqref{eq:vir2}.
 Thus we complete the proof.
\end{proof}

Now we are in position to prove Lemma~\ref{lem:exp-bound}.
\begin{proof}[Proof of Lemma~\ref{lem:exp-bound}]
Since $\chi_M(\eta)=0$ for $\eta\not\in[-2M,2M]$, we have
$$\|P_{\le M}v_2(t)\|_X\le \|P_1(0,2M)v_2(t)\|_X\,.$$ 
Combining Lemmas~\ref{lem:nonresonant-ylow} and \ref{lem:virial} with
\eqref{eq:ell-est} and the definition $\bM_1(T)$,
we have \eqref{eq:M4-bound} provided $\delta_6$ is sufficiently small.
This completes the proof of Lemma~\ref{lem:exp-bound}.
\end{proof}
\bigskip

\section{Proof of Theorem~\ref{thm:stability}}
\label{sec:thm1}
Now we are in position to complete the proof of Theorem~\ref{thm:stability}.
\begin{proof}[Proof of Theorem~\ref{thm:stability}]
Let $\delta_*=\min_{0\le i\le 6}\delta_i/2$.
Thanks to the scaling invariance of \eqref{eq:KPII}, we may assume $c_0=2$
without loss of generality. Since $\tc(0)=\tx(0)\equiv0$ in $Y$ and
$v_1(0)=v_0$ and $v_2(0)=0$, there exists a $T>0$ such that
\begin{equation}
  \label{eq:tot-1}
\bM_{tot}(T):=\bM_{c,x}(T)+\bM_1(T)+\bM_2(T)+\bM_v(T)\le \frac{\delta_*}{2}\,.
\end{equation}
By Proposition~\ref{prop:continuation}, we can extend the
decomposition \eqref{eq:decomp} satisfying \eqref{eq:decomp2} and
\eqref{eq:orth} beyond $t=T$. Let $T_1\in(0,\infty]$ be the maximal
time such that the decomposition \eqref{eq:decomp} satisfying
\eqref{eq:decomp2} and \eqref{eq:orth} exists for $t\in[0,T_1]$ and
$\bM_{tot}(T_1)\le \delta_*$.  Suppose $T_1<\infty$.  Then it follows
from Lemmas~\ref{lem:M1-bound}, \ref{lem:M4-bound}, \ref{lem:v1-a},
\ref{lem:v1-b} and \ref{lem:exp-bound} that
if $\||D_x|^{-1/2}v_0\|_{L^2}+\||D_x|^{1/2}v_0\|_{L^2}
+\||D_x|^{-1/2}|D_y|^{1/2}v_0\|_{L^2}$ is sufficiently small, then
\begin{align*}
& \bM_1(T)\lesssim \|v_0\|_{L^2}\,,\\
& \bM_{c,x}(T)\lesssim \|v_0\|_{L^2}+\bM_1(T)+\bM_2(T)^2
\lesssim \|v_0\|_{L^2}+\bM_2(T)^2\,,\\
& \bM_2(T)\lesssim \bM_{c,x}(T)+\bM_1(T)\lesssim \|v_0\|_{L^2}+\bM_2(T)^2\,,\\
& \bM_v(T)\lesssim \|v_0\|_{L^2}+\bM_{c,x}(T)+\bM_1(T)+\bM_2(T)
\lesssim \|v_0\|_{L^2}+\bM_2(T)\,,
\end{align*}
and $\bM_{tot}(T_1)\lesssim \|v_0\|_{L^2(\R^2)}+\bM_{tot}(T_1)^2$.
If $\|v_0\|_{L^2(\R^2)}$ is sufficiently small, we have
$$\bM_{tot}(T_1)\le \delta_*/2\,,$$
which contradicts to the maximality of $T_1$.
Thus we prove $T_1=\infty$ and
\begin{equation}
  \label{eq:tot-3}
  \bM_{tot}(\infty)\lesssim \|v_0\|_{L^2(\R^2)}\,.
\end{equation}
\par

By \eqref{eq:decomp}, \eqref{eq:psinorm} and \eqref{eq:tot-3},
\begin{align*}
\|u(t,x,y)-\varphi_{c(t,y)}(x-x(t,y))\|_{L^2(\R^2)}\le &
\|v(t)\|_{L^2(\R^2)}+\|\tpsi_{c(t,y)}\|_{L^2(\R^2)}
\\ \lesssim &  \bM_v(\infty)+\bM_{c,x}(\infty)\,.
\end{align*}
Since $H^k(\R)\subset Y$ for any $k\ge0$,
we see that \eqref{phase-sup} follows immediately
from \eqref{eq:tot-3} and Lemma~\ref{lem:cx_t-bound}.
Moreover, we have \eqref{phase2} because $c_y$,
$x_{yy}\in L^2(0,\infty;Y)$ and  $pd_tc_y$,  $\pd_tx_{yy}\in L^\infty(0,\infty;Y)$.
\par
Finally, we will prove \eqref{AS}.
Since $\|f\|_{L^\infty}\lesssim \|f\|_Y^{1/2}\|\pd_yf\|_Y^{1/2}$ 
for any $f\in Y$, we have from \eqref{phase-sup}
\begin{gather*}
\|x_t(t)-2c(t)\|_{L^\infty}+\|c(t)-c_0\|_{L^\infty}\lesssim \|v_0\|_{L^2}\,,
\end{gather*}
and for any $y\in\R$,
\begin{align*}
  x(t,y)=\int_0^t x_t(s,y)\,ds\ge (2c_0+O(\|v_0\|_{L^2})t\,.
\end{align*}
Here we use $x(0,\cdot)=0$.
Let $\tx_1(t)=c_0t$ and $x_0=R$. Then by Lemma~\ref{lem:virial-0}, 

\begin{equation}
  \label{eq:v1-loc-est}
  \lim_{t\to\infty}\|v_1(t,x+x(t,y),y)\|_{L^2(x>-c_0t/2,y\in\R)}=0\,.
\end{equation}
Dividing the integral interval $[0,t]$ of 
into $[0,t/2]$ and $[t/2,t]$ and
using \eqref{eq:v1n-est}--\eqref{eq:N2,3-est}, 
we have
$$\lim_{t\to\infty}\|v_2(t)\|_X=0\,.$$
Thus we complete the proof of Theorem~\ref{thm:stability}.
\end{proof}
\bigskip

\section{Proof of Theorem~\ref{thm:poly}}
\label{sec:poly}
If $v_0(x,y)$ is polynomially localized, then at $t=0$
we can decompose a perturbed line soliton into a sum of a locally
amplified line soliton and a remainder part $v_*(x,y)$ satisfying
$\int_\R v_*(x,y)\,dx=0$ for all $y\in\R$.
\begin{lemma}
  \label{lem:nonzeromean1}
Let $c_0>0$ and $s>1$ be constants.
There exists a positive constant $\eps_0$ such that if
$\eps:=\|\la x\ra^sv_0\|_{H^1(\R^2)}<\eps_0$, then there 
exists $c_1(y)\in H^1(\R)$ such that
\begin{align}
\label{eq:nonzero1a}
 & \int_\R\left(\varphi_{c_1(y)}(x)-\varphi_{c_0}(x)\right)\,dx
   =\int_\R v_0(x,y)\,dx\,,
\\ \label{eq:nonzero1b} 
& 
\left\|c_1(\cdot)-c_0\right\|_{L^2(\R)}
\lesssim \left\|\la x\ra^{s/2}v_0\right\|_{L^2(\R^2)}\,,
\quad 
\left\|\pd_yc_1(\cdot)\right\|_{H^1(\R)}
\lesssim \left\|\la x\ra^{s/2}v_0\right\|_{H^1(\R^2)}\,,
\\ \label{eq:nonzero1c}
& \|v_*\|_{L^2(\R^2)} \lesssim \|\la x\ra^{s/2}v_0\|_{L^2(\R^2)}\,,
\quad 
 \|\pd_x^{-1}v_*\|_{L^2}+\|v_*\|_{H^1(\R^2)} \lesssim
\|\la x\ra^sv_0\|_{H^1(\R^2)}\,,
\end{align}
where $v_*(x,y)=v_0(x,y)+\varphi_{c_0}(x)-\varphi_{c_1(y)}(x)$.
\end{lemma}
\begin{proof}
First, we will prove
\begin{equation}
  \label{eq:nonzerompf1}
\sup_{y\in\R}\left|\int_\R v_0(x,y)\,dx\right|\lesssim
 \|\la x\ra^{s/2}v_0\|_{L^2(\R^2)}+ \|\la x\ra^{s/2}\pd_yv_0\|_{L^2(\R^2)}\,.
\end{equation}
By the Schwarz inequality,
\begin{equation}
\label{eq:nonzerompf2}
  \left|\int_\R v_0(x,y)\,dx\right|\lesssim
\left\{\int_\R \la x\ra^s v_0(x,y)^2\,dx\right\}^{1/2}\,.
\end{equation}
Substituting
$\sup_yv_0^2(x,y)\le \int_\R \{(\pd_yv_0)^2+v_0^2\}(x,y)\,dy$
into the right hand side of \eqref{eq:nonzerompf2},
we have \eqref{eq:nonzerompf1}.
\par
Let
$$c_1(y)=\left\{\sqrt{c_0}+\frac{1}{2\sqrt{2}}\int_\R v_0(x,y)\,dx\right\}^2\,.$$
Then we have \eqref{eq:nonzero1a} and $\int_\R v_*(x,y)\,dx=0$ for
every $y\in\R$ because
\begin{equation}
  \label{eq:nonzerompf3}
\int_\R \{\varphi_{c_1(y)}(x)-\varphi_{c_0}(x)\}\,dx=2\sqrt{2}(\sqrt{c_1(y)}-\sqrt{c_0})\,.
\end{equation}
Moreover, it follows from \eqref{eq:nonzerompf1} that
$$\sup_{y\in\R}|c_1(y)-c_0|\lesssim 
 \|\la x\ra^{s/2}v_0\|_{L^2(\R^2)}+ \|\la x\ra^{s/2}\pd_yv_0\|_{L^2(\R^2)}\,.$$
By \eqref{eq:nonzero1a}, \eqref{eq:nonzerompf2} and \eqref{eq:nonzerompf3},
$$\|c_1(y)-c_0\|_{L^2(\R)}\lesssim 
 \left\|\int_\R v_0(x,y)\,dx\right\|_{L^2(\R)}
\lesssim  \|\la x\ra^{s/2}v_0\|_{L^2(\R^2)}\,.$$
Using Minkowski's inequality, we have for $j\ge0$,
\begin{align*}
\|\pd_x^j\varphi_{c_1(y)}-\pd_x^j\varphi_{c_0}\|_{L^2(\R^2)} \le &
\left\| \int_{c_0}^{c_1(y)}\left\|\pd_x^j\pd_c\varphi_c\right\|_{L^2_x(\R)}\,dc
\right\|_{L^2_y(\R)}
\\ \lesssim & \|c_1(y)-c_0\|_{L^2(\R)} \lesssim \|\la x\ra^{s/2}v_0\|_{L^2(\R^2)}\,,
\end{align*}
and 
$\|\pd_x^jv_*\|_{L^2(\R^2)} \lesssim \|\pd_x^jv_0\|_{L^2(\R^2)}
+\|\la x\ra^{s/2}v_0\|_{L^2(\R^2)}$.
Similarly, we have
\begin{gather*}
\|\pd_yc_1\|_{L^2(\R)}\lesssim \|\la x\ra^{s/2}\pd_yv_0\|_{L^2(\R^2)}\,,\\
\|\pd_yv_*\|_{L^2(\R^2)}\lesssim \|\pd_yv_0\|_{L^2(\R^2)}+\|\pd_yc_1\|_{L^2(\R)}
\lesssim \|\la x\ra^{s/2}\pd_yv_0\|_{L^2(\R^2)}\,.
\end{gather*}
\par
Since $\int_\R v_*(x,y)\,dx=0$, 
$$\pd_x^{-1}v_*(x,y)
=\int_{\pm\infty}^x\{v_0(x_1,y)+\varphi_{c_0}(x_1)-\varphi_{c_1(y)}(x_1)\}\,dx_1\,.$$
By the Schwarz inequality, we have for $\pm x>0$,
\begin{align*}
  |\pd_x^{-1}v_*(x,y)|
\lesssim & 
\left(\|\la x\ra^sv_0(\cdot,y)\|_{L^2(\R)}
+\|\la x\ra^s(\varphi_{c_1(y)}-\varphi_{c_0})\|_{L^2(\R)}\right)
\la x\ra^{-s+1/2}
\\ \lesssim & 
\left(\|\la x\ra^sv_0(\cdot,y)\|_{L^2(\R)}+|c_1(y)-c_0|\right)
\la x\ra^{-s+1/2}\,,
\end{align*}
and 
\begin{align*}
\|\pd_x^{-1}v_*\|_{L^2(\R^2)} \lesssim & \|\la x\ra^sv_0\|_{L^2}+\|c_1-c_0\|_{L^2(\R^2)}
 \lesssim  \|\la x\ra^s v_0\|_{L^2(\R^2)}\,.
\end{align*}
Thus we complete the proof.
\end{proof}

Now we are in position to prove Theorem~\ref{thm:poly}.
\begin{proof}[Proof of Theorem~\ref{thm:poly}]
To prove Theorem~\ref{thm:poly}, we modify the definitions
of $v_1(t,z,y)$, $v_2(t,z,y)$ $c(t,y)$ and $x(t,y)$ as follows.
Let $\tv_1$ be a solution of \eqref{eq:KPII} satisfying
$\tv_1(0,x,y)=v_*(0,x,y)$.
Then it follows from Lemmas~\ref{lem:v1-a} and \ref{lem:nonzeromean1} that
$\bM_1(\infty)\lesssim \|\la x\ra^{s/2}v_0\|_{L^2(\R^2)}$.
By \eqref{eq:nonzero1c},
\begin{align*}
& \||D_x|^{-1/2}v_*\|_{L^2(\R^2)}+\||D_x|^{1/2}v_*\|_{L^2(\R^2)}
+\||D_x|^{-1/2}|D_y|^{1/2}v_*\|_{L^2(\R^2)}
\\ \lesssim & \|v_*\|_{H^1(\R^2)}+\|\pd_x^{-1}v_*\|_{L^2(\R^2)}
\lesssim  \|\la x\ra^sv_0\|_{H^1(\R^2)}\,,
\end{align*}
and $\bM_1'(\infty)\lesssim \|\la x\ra^sv_0\|_{H^1(\R^2)}$
follows from Lemma~\ref{lem:v1-b}.
\par
Let $\tu(t,x,y)=u(t,x,y)-\tv_1(t,x,y)$. Then $\tu(0,x,y)=\varphi_{c_1(y)}(x)$.
By Lemma~\ref{lem:nonzeromean1},
$$\|u(0,x,y)-\varphi_{c_0}(x)\|_X \lesssim \|c_1(\cdot)-c_0\|_{L^2(\R)}
\lesssim \|\la x\ra^{s/2}v_0\|_{L^2(\R^2)}\,,$$
and  Lemma~\ref{lem:decomp} and Remark\ref{rem:decomp} imply that
there exist a $T>0$ and 
$(v_2(t),\tc(t),\tx(t))\in X\times Y\times Y$
satisfying 
\eqref{eq:decomp}, \eqref{eq:decomp2} and \eqref{eq:orth} for $t\in[0,T]$,
where $\tc(t,y)=c(t,y)-c_0$ and $\tx(t,y)=x(t,y)-2c_0t$. 
Clearly, we have
$$\|v_2(0)\|_{X\cap L^2(\R^2)}+\|\tc(0)\|_Y\lesssim
\|\la x\ra^{s/2}v_0\|_{L^2(\R^2)}\,,
\quad x(0,\cdot)=0\,,$$
and following the proof of Lemmas~\ref{lem:M1-bound}, \ref{lem:M4-bound}
and \ref{lem:exp-bound}, we can prove
\begin{gather*}
\bM_{c,x}(T)\lesssim \|\la x\ra^{s/2}v_0\|_{L^2(\R^2)}+\bM_1(T)+\bM_2(T)^2\,,
\\
\bM_v(T)\lesssim \|\la x\ra^{s/2}v_0\|_{L^2(\R^2)}+\bM_{c,x}(T)+\bM_1(T)+\bM_2(T)\,,
\\
\bM_2(T)\lesssim \|\la x\ra^{s/2}v_0\|_{L^2(\R^2)}+\bM_{c,x}(T)+\bM_1(T)\,.
\end{gather*}
Thus we can prove Theorem~\ref{thm:poly} in exactly the same way as
Theorem~\ref{thm:stability}.
\end{proof}
\bigskip

\appendix
\section{Proof of Claim~\ref{cl:[B3,pdy]}}
\label{sec:[B3,pdy]}

\begin{proof}[Proof of Claim~\ref{cl:[B3,pdy]}]
By Claims~B.1 and B.2 in \cite{Mi},
\begin{gather}
\label{eq:wS1-est}
\|\wS_1\|_{B(Y)}+\|\wS_1\|_{B(Y_1)}\lesssim 1\,,
\quad [\wS_1,\pd_y]=0\,,
\\ \label{eq:wS2-est}
\|\wS_2\|_{B(Y_1,Y)}\lesssim \|\tc\|_Y\,,\quad
\|\wS_2\|_{B(Y)}\lesssim \|\tc\|_{L^\infty}\,,
\quad \left\|[\pd_y,\bS_2]\right\|_{B(Y_1,Y)} \lesssim \|c_y\|_Y\,.
\end{gather}
Following the proof of Claims~B.3--B.5 in \cite{Mi},
we can show
\begin{gather}
\label{eq:S3p}
\|S^3_k[p](f)(t,\cdot)\|_Y\le Ce^{-a(3t+L)}\|e^{\a z}p\|_{L^2}\|\wP_1f\|_{Y}\,,
\quad [S^3_k[p],\pd_y]=0\,,
\\ \label{eq:S4p}
\|S^4_k[p](f)(t,\cdot)\|_{Y_1} \le Ce^{-a(3t+L)}\|e^{\a z}p\|_{L^2}
\|\tc(t)\|_Y\|f\|_{L^2}\,,
\\ \label{eq:Sv2}
\|S^5_k(f)(t,\cdot)\|_{Y_1}+\|S^6_k(f)\|_{Y_1}\le 
C\|v_2(t,\cdot)\|_X\|f\|_{L^2}\,,
\end{gather}
in exactly the same way. 
By \eqref{eq:wS1-est} and \eqref{eq:S3p}, we have  $[\pd_y,B_4]=0$.
\par

Applying \eqref{eq:S3p}, \eqref{eq:S4p} with
$p(z)=\pd_z^j\psi(z)$ ($j\ge0$) and using \eqref{eq:Sv2}
and Claim~\ref{cl:(1+wT2)^{-1}},
we have
\begin{equation}
\label{eq:bS3}
  \begin{split}
& \|\wS_3\|_{B(Y)}+\|\bS_3\|_{B(Y))}\lesssim e^{-\a(3t+L)}\,,\\
& \|\wS_4\|_{B(Y,Y_1)}+\|\bS_4\|_{B(Y,Y_1)} \lesssim \|\tc(t)\|_Y e^{-\a(3t+L)}\,,
  \end{split}
\end{equation}
\begin{gather}
 \label{eq:bS5}
\|\wS_5\|_{L^2(0,T;B(Y,Y_1))}+\|\bS_5\|_{L^2(0,T;B(Y,Y_1))}\lesssim \|v_2(t)\|_X\,.
\end{gather}
In view of \eqref{eq:def-B3},
\begin{equation}
  \label{eq:b3-pd1}
[\pd_y,B_3]=[\pd_y,\wC_1]+\sum_{j=1,2}\pd_y^2[\pd_y,\bS_j]-
\sum_{j=3,4,5}[\pd_y,\bS_j]\,.  
\end{equation}
We will estimate each term of the right hand side
following the proof of \cite[Claim~7.1]{Mi}.
By \cite[Claims~B.7]{Mi},
\begin{equation}
  \label{eq:[T,pd]}
\|[\pd_y,\wC_k]\|_{B(Y,Y_1)}\lesssim \|c_y\|_Y \quad\text{for $k=1$, $2$.}  
\end{equation}
Applying \cite[Claims~B.1--B.7]{Mi} to 
$[\pd_y,\bS_j]=\{[\pd_y,\wS_j]+\bS_j[\wC_2,\pd_y]\}(I+\wC_2)^{-1}$,
we have
\begin{equation}
\label{eq:[pd,bS]}
\|[\pd_y,\bS_j]\|_{B(Y,Y_1)}\lesssim \|c_y\|_Y 
\quad\text{for $1\le j\le 4$,}\,.
\end{equation}
By \eqref{eq:Sv2} and the fact that $\pd_y$ is bounded on $Y$ and $Y_1$,
\begin{equation}
  \label{eq:[pd,bS5]}
  \|[\pd_y,\bS_5]\|_{B(Y,Y_1)}\lesssim \|v_2\|_X\,.
\end{equation}
Combining \eqref{eq:b3-pd1}--\eqref{eq:[pd,bS5]},
we obtain the first two estimates of Claim~\ref{cl:[B3,pdy]}.
Thus we complete the proof.
\end{proof}

Finally, we will  estimate the operator norm of $S^7_1[q_c]$.
\begin{claim}
  \label{cl:S7}
There exist positive constants $C$ and $\delta$ such that\newline if
$\sup_{t\in[0,T]}\|\tc(t)\|_{L^\infty}\le \delta$, then 
\begin{equation}
  \label{eq:S7-est}
\|S^7[q_c](f)(t,\cdot)\|_{Y_1}\le C\|v_1(t,\cdot)\|_{W(t)}
\left\|e^{\a|\cdot|}\sup_{c\in[2-\delta,2+\delta]}q_c\right\|_{L^2(\R)}\|f\|_{L^2(\R)}\,.
  \end{equation}
\end{claim}
\begin{proof}
Applying the Schwarz inequality to the right hand side of
\begin{align*}
\|S^7_1[q_c](f)(t,y)\|_{Y_1}
=& \frac{1}{2\sqrt{2\pi}}\left\|
\int_{\R^2} v_1(t,z,y)f(y)q_{c(t,y)}(z)e^{-iy\eta}\,dzdy
\right\|_{L^\infty[-\eta_0,\eta_0]}\,,
\end{align*}
we have \eqref{eq:S7-est}.
\end{proof}

Using Lemma~\ref{lem:cx_t-bound}, we can prove the following commutator
estimate in the same way as Claim~\ref{cl:[B3,pdy]}.
\begin{claim}
  \label{cl:B3-comm}
There exist positive constants $C$ and $\delta$ such that if
$\bM_{c,x}(T)\le\delta$, then
\begin{gather*}
\left\|\left[\pd_t,B_3\right]\right\|_{B(L^2(0,T;Y),L^1(0,T;Y))}
\le C (e^{-\a L}+\bM_{c,x}(T)+\bM_1(T)+\bM_2(T))\,.
\end{gather*}
\end{claim}

\section{Estimates of $R^k$}
\label{sec:Rk}
\begin{claim}
  \label{cl:R1-R2}
There exist positive constants $\delta$ and $C$ such that if
$\bM_{c,x}(T)\le \delta$, then
\begin{equation*}
\|R^2_k\|_{L^2(0,T;Y)} \le C\bM_{c,x}(T)^2\,.
\end{equation*}
\end{claim}
\begin{proof}
By \cite[Claims~B.1 and B.2]{Mi}, 
\begin{align*}
\|R^2_k\|_Y \lesssim & 
\|\tc\|_{L^\infty}(\|x_{yy}\|_Y+\|c_{yy}\|_Y)
+(1+\|\tc\|_{L^\infty})\|c_y\|_{L^\infty}\|c_y\|_Y\,.
\end{align*}
Since $Y\subset H^1(\R)$, we have  Claim~\ref{cl:R1-R2}.
\end{proof}

\begin{claim}
  \label{cl:R3}
There exist  positive constants $\delta$ and $C$ such that
if $\bM_1(T)\le \delta$, then
$\|R^3_k(t,\cdot)\|_Y\le Ce^{-\a(3t+L)}\bM_{c,x}(T)^2$  for $t\in[0,T]$.
\end{claim}

\begin{claim}
\label{cl:akbound}
  There exist  positive constants $C$ and $L_0$ such that
if $L\ge L_0$, then
$$
\|\widetilde{\mathcal{A}}_1(t)\|_{B(Y)}
+\|\widetilde{\mathcal{A}}_1(t)\|_{B(Y_1)}
+\|A_1(t)\|_{B(Y)}\le Ce^{-\a(3t+L)}
\quad\text{for every $t\ge0$.}$$
\end{claim}

Claims~\ref{cl:R3}--\ref{cl:akbound}
can be shown in exactly the same way as \cite[Claims~D.2 and D.3]{Mi}.

\begin{claim}
  \label{cl:R4-R5}
Suppose $\a\in(0,1)$ and $\bM_1(T)\le \delta$ If $\delta$ is sufficiently small,
then there exists a positive constant $C$ such that
\begin{equation}
  \label{eq:R4k}
  \begin{split}
& \sup_{t\in[0,T]}\|R^4_k(t)\|_{Y_1}+\|R^4_k\|_{L^1(0,T;Y_1)}
\le C(\bM_{c,x}(T)+\bM_1(T)+\bM_2(T))\bM_2(T)\,,
  \end{split}
\end{equation}
\begin{align}
\label{eq:R5k}
& \sup_{t\in[0,T]}\|R^5_k(t)\|_{Y_1}+\|R^5_k\|_{L^2(0,T;Y_1)}
\le  C\bM_{c,x}(T)\bM_2(T)\,,
\\ &  \label{eq:R6'}
\|R^6_k\|_{Y_1}\le Ce^{-\a(3t+L)}\bM_{c,x}(T)\bM_2(T)\,.
\end{align}
\end{claim}
\begin{proof}
Following the proof of Claim~D.5 in \cite{Mi}, we have
\begin{gather*}
\|II^1_k(t,\cdot)\|_{Z_1}\lesssim (\|c_y(t)\|_Y+\|c_{yy}\|_Y+\|c_y(t)\|_{L^4}^2)
\|v_2(t)\|_X\,,\\   
\|II^2_k(t,\cdot)\|_{Z_1} \lesssim 
(\|e^{-\a|z|/2}v_1(t)\|_{L^2}+\|v_2(t)\|_X)\|v_2(t)\|_X\,,\\
\|II^3_{k1}(t,\cdot)\|_{Z_1}\lesssim \|x_{yy}(t)\|_Y\|v_2(t)\|_X\,,
\quad \|II^3_{k2}(t,\cdot)\|_{Z_1}\lesssim \|x_y(t)\|_Y\|v_2(t)\|_X\,,
\\
\|R^6_k\|_{Y_1}\lesssim \|v_2(t)\|_X\|\tpsi_{c(t,y)}\|_X
\lesssim  e^{-\a(3t+L)}\|\tc(t)\|_{L^2(\R)}\|v_2(t)\|_X\,.
\end{gather*}
Claim~\ref{cl:R4-R5} follows immediately from the above.
\end{proof}
\begin{claim}
\label{cl:R6}
There exist positive constants $\delta$ and  $C$ such that
if $\bM_{c,x}(T) \le \delta$, then
\begin{align}
\label{eq:R6-est2}
&   \sup_{t\in[0,T]}\|\wP_1R^7_1\|_Y+
\|\wP_1R^{71}_1\|_{L^1(0,T;Y)}\le C\bM_{c,x}(T)^2\,,\\
\label{eq:R6-est3}
&  \sup_{t\in[0,T]}\|\wP_1R^7_2\|_Y+\|\wP_1R^7_2\|_{L^2(0,T;Y)}\le C\bM_{c,x}(T)^2\,.
\end{align}
\end{claim}

\begin{proof}
Since $\|f\|_{L^\infty}\lesssim \|f\|_Y^{1/2}\|f_y\|_Y^{1/2}$ for $f\in Y$,
it follows from \cite[(D.11),(D.15)]{Mi} that
\begin{align*}
& \left\|\left(\frac{c}{2}\right)^{1/2}c_y-b_y\right\|_{L^2}
\\ \lesssim & 
\left\|\left(\frac{c}{2}\right)^{1/2}-1\right\|_{L^\infty}\|c_y\|_Y
+\|b_y-c_y\|_Y
\\ \lesssim & \|\tc\|_Y^{1/2}\|c_y\|_Y^{3/2},
\end{align*}
and
\begin{align*}
& \left\|\left(\frac{c}{2}\right)^{3/2}-1-\frac{3}{4}b\right\|_{L^\infty}
\\ \lesssim & 
\left\|\left(\frac{c}{2}\right)^{3/2}-1-\frac{3}{4}b\right\|_{L^2}^{1/2}
\left\|\left(\frac{c}{2}\right)^{1/2}c_y-b_y\right\|_{L^2}^{1/2}
\\ \lesssim & \|\tc\|_Y\|c_y\|_Y\,.
\end{align*}
Combining the above with \cite[(D.11), (D.13)]{Mi}, we have
\begin{multline*}
\|\wP_1R^7_1\|_Y
\lesssim  
\left\|\left(\frac{c}{2}\right)^{3/2}-1-\frac{3}{4}b\right\|_{L^\infty}
\|x_{yy}\|_{L^2}\\
+\|x_y\|_{L^\infty}\left\|b_y-\left(\frac{c}{2}\right)^{1/2}c_y\right\|_Y
+\|c_y\|_Y^2
\\ \lesssim  \|x_{yy}\|_Y\|c_y\|_Y \|\tc\|_Y
+\|x_y\|_Y^{1/2}\|x_{yy}\|_Y^{1/2}\|c_y\|_Y^{3/2} \|\tc\|_Y^{1/2}+\|c_y\|_Y^2\,.
\end{multline*}
Hence by the definition of $\bM_{c,x}(T)$, we have \eqref{eq:R6-est2}.
We can prove \eqref{eq:R6-est3} using \cite[Claim~D.6]{Mi} and \eqref{eq:wpxy}
in a similar way. Thus we complete the proof.
\end{proof}

\begin{claim}
  \label{cl:R8-R11}
There exist positive constants $C$ and $\delta$ such that
if $\bM_{c,x}(T)+\bM_2(T)<\delta$, then
\begin{align}
\label{eq:R8-est}
& \sup_{t\in[0,T]}\|R^8(t)\|_Y+\|R^8\|_{L^2(0,T;Y)} \le C\bM_{c,x}(T)^2\,,
\\ \label{eq:R9-est}
& \sup_{t\in[0,T]}\|R^9(t)\|_Y+\|R^9\|_{L^1(0,T;Y)} \le 
C\bM_{c,x}(T)(\bM_{c,x}(T)+\bM_2(T))\,,
\\ \label{eq:R10-est}
&  \sup_{t\in[0,T]}\|R^{10}(t)\|_Y+\|R^{10}\|_{L^2(0,T;Y)} \le C\bM_{c,x}(T)^2\,,
\\ \label{eq:R11-est} 
& \sup_{t\in[0,T]}\|R^{11}(t)\|_Y+\|R^{11}\|_{L^1(0,T;Y)} \le C\bM_{c,x}(T)^2\,.
\end{align}
\end{claim}
\begin{proof}
By \eqref{eq:H^s-Y} and the fact that $\|b\|_Y\lesssim \|\tc\|_Y$,
$$\|(I+\mathcal{C}_2)(c_yx_y)-(bx_y)_y\|_Y
\lesssim (\|\tc\|_Y+\|x_y\|_Y)(\|c_y\|_Y+\|x_{yy}\|_Y)\,,$$
whence
\begin{equation}
  \label{eq:pfR8-1}
\|(I+\mathcal{C}_2)(c_yx_y)-(bx_y)_y\|_{L^2(0,T;Y)\cap L^\infty(0,T;Y)}
\lesssim \bM_{c,x}(T)^2\,.  
\end{equation}
Eq.~\eqref{eq:R8-est} follows from \eqref{eq:pfR8-1} and
\cite[(C.1),(C.2)]{Mi}.
Eq.~\eqref{eq:R9-est} follows from \eqref{eq:pfR8-1}, \eqref{eq:bS3} and
\eqref{eq:bS5}.

By \cite[Claims~B.1 and (D.11)]{Mi}, we have $\|\wS_0\|_{B(Y)}\lesssim 1$ and
\begin{equation}
  \label{eq:pfR10}
\|R^{10}\|_Y \lesssim \|c_y\|_Y\|\tc\|_{L^\infty}\,.  
\end{equation}
By Claim~\ref{cl:akbound} and \cite[(D.10)]{Mi},
\begin{equation}
  \label{eq:pfR11}
\|R^{11}\|_Y \lesssim e^{-\a(3t+L)}\|\tc\|_{L^\infty}\|\tc\|_Y\,.  
\end{equation}
The estimates \eqref{eq:R10-est} and \eqref{eq:R11-est} follows immediately
from \eqref{eq:pfR10} and \eqref{eq:pfR11}.
\end{proof}

\begin{claim}
\label{cl:Rv1}
There exist positive constants $C$ and $\delta$ such that
if $\bM_{c,x}(T)\le \delta$, then
\begin{gather}
\label{eq:Rv1-11-est}
\|R^{v_1}_{11}\|_{L^1(0,T;Y_1)}\le C\bM_1(T)(\bM_{c,x}(T)+\bM_1(T))\,,\\
  \label{eq:Rv1-12-est}
\|R^{v_1}_2\|_{L^2(0,T;Y)}+\|R^{v_1}_{12}\|_{L^2(0,T;Y)}\le C\bM_1(T)\,.
\end{gather}
\end{claim}
\begin{proof}[Proof of Claim~\ref{cl:Rv1}]
By the assumption, there exists $\delta'\in(0,2)$ such that
$c(t,y)\in[2-\delta',2+\delta']$ for $t\in[0,T]$ and $y\in\R$.
Since $\psi$ has a compact support,
\begin{equation}
  \label{eq:II6-13-est}
\begin{split}
\|II^6_{13}(t,\eta)\|_{L^\infty(-\eta_0,\eta_0)}
\lesssim & \|v_1(t)\|_{L^2(\R^2)}\|\tc\|_Y
\sup_{\substack{\eta\in[-\eta_0,\eta_0]\\c\in[2-\delta',2+\delta']}}
\|\psi(\cdot+3t)\pd_zg^*(\cdot,\eta,c)\|_{L^2(\R)}
\\ \lesssim & e^{-\a(3t+L)}\|\tc(t)\|_Y\|v_1(t)\|_{L^2(\R^2)}\,.
\end{split}  
\end{equation}
By the Schwarz inequality,
\begin{equation}
  \label{eq:II6-111-est}
  \begin{split}
\|II^6_{111}(t,\eta)\|_{L^\infty(-\eta_0,\eta_0)}
\lesssim  &
\|v_1(t)\|_{W(t)}^2+\|\pd_x^{-1}\pd_yv_1(t)\|_{W(t)}\|c_y(t)\|_Y
\\ & +\|v_1(t)\|_{W(t)}(\|x_{yy}(t)\|_Y+\|(c_yx_y)(t)\|_{L^2(\R)})\,.    
  \end{split}
\end{equation}
Combining \eqref{eq:II6-13-est} and \eqref{eq:II6-111-est}, we have
\eqref{eq:Rv1-11-est}.
\par

Next, we will prove \eqref{eq:Rv1-12-est}.
We decompose $II^6_{112}$ as $II^6_{1121}+II^6_{1122}$, where
\begin{align*}
II^6_{1121}(t,\eta)=& -\frac32\int_{\R^2}(\pd_z^{-1}\pd_yv_1)(t,z,y)
\varphi(z)e^{-iy\eta}\,dzdy
\\=& -\frac{3\sqrt{2\pi}}{2}\int_\R \varphi(z)
\mF_y(\pd_z^{-1}\pd_yv_1)(t,z,\eta)\,dz\,,  
\end{align*}
\begin{align*}
II^6_{1122}(t,\eta)= & -\frac32\int_{\R^2}(\pd_z^{-1}\pd_yv_1)(t,z,y)
\tc(t,y)\delta\varphi_{c(t,y)}(z)e^{-iy\eta}\,dzdy
\\ & +3\int_{\R^2}v_1(t,z,y)x_y(t,y)\varphi_{c(t,y)}(z)e^{-iy\eta}\,dzdy\,.  
\end{align*}
By the the Schwarz inequality and Plancherel's theorem,
\begin{equation}
  \label{eq:1121-est}
  \begin{split}
& \|II^6_{1121}(t,\cdot)\|_{L^2(-\eta_0,\eta_0)}
\\ \lesssim & \left(\int_{-\eta_0}^{\eta_0}\int_\R e^{-2\a|z|}
|\mF_y(\pd_z^{-1}\pd_yv_1)(t,z,\eta)|^2\,dzd\eta\right)^{1/2}
\|e^{\a|\cdot|}\varphi\|_{L^2(\R)}
\\ \lesssim & \|v_1\|_{W(t)}\,,    
  \end{split}
\end{equation}
and
\begin{equation}
  \label{1122-est}
\|II^6_{1122}(t,\eta)\|_{L^\infty(-\eta_0,\eta_0)}  \lesssim 
(\|v_1\|_{W(t)}+\|\pd_z^{-1}\pd_yv_1\|_{W(t)})(\|\tc(t)\|_Y+\|x_y(t)\|_Y)\,.
\end{equation}
Similarly, we have 
\begin{equation}
  \label{eq:II62-est}
\|II^6_2(t,\cdot)\|_{L^2(-\eta_0,\eta_0)}+
\|II^6_{12}(t,\cdot)\|_{L^2(-\eta_0,\eta_0)}\lesssim \|v_1(t)\|_{W(t)}\,.
\end{equation}
Since $Y_1\subset Y$, we have \eqref{eq:Rv1-12-est} from
\eqref{eq:1121-est}--\eqref{eq:II62-est}.
Thus we complete the proof.
\end{proof}

Finally, we will estimate $k(t,y)$.
\begin{claim}
  \label{cl:k-est}
There exist positive constants $C$ and $\delta$ such that if
$\bM_{c,x}(T)\le \delta$, then
\begin{equation}
  \label{eq:k-est}
\sup_{t\in[0,T]}\|k(t,\cdot)\|_Y+\|k\|_{L^2(0,T;Y)} \le C\bM_1(T)\,.
\end{equation}
Moreover,
\begin{equation}
\label{eq:k-lim}
\lim_{t\to\infty}\|k(t,\cdot)\|_Y=0\,.
\end{equation}
\end{claim}
\begin{proof}

Let $\delta\varphi_c=(\varphi_c-\varphi)/\tc$ and
\begin{gather*}
k_1(t,y)=\frac{1}{4\pi}\int_{-\eta_0}^{\eta_0}\int_{\R^2}
v_1(t,z,y_1)\varphi(z)e^{i(y-y_1)\eta}\,dzdy_1d\eta\,,
\\
k_2(t,y)=\frac{1}{4\pi}\int_{-\eta_0}^{\eta_0}\int_{\R^2}\tc(t,y_1)
v_1(t,z,y_1)\delta\varphi_{c(t,y_1)}(z)e^{i(y-y_1)\eta}\,dzdy_1d\eta\,.  
\end{gather*}
By the definitions, we have $k=k_1+k_2$.
Using Plancherel's theorem and Minkowski's inequality, we have
\begin{equation}
  \label{eq:k1-est}
\begin{split}
  \|k_1(t,\cdot)\|_Y=& \frac{1}{2\sqrt{2\pi}}
\left\|\int_\R(\mF_yv_1)(t,z,\cdot)\varphi(z)\,dz\right\|_{L^2(-\eta_0,\eta_0)}
\\ \le & 
\frac{1}{2\sqrt{2\pi}}
\int_\R \left\|(\mF_yv_1)(t,z,\cdot)\right\|_{L^2(-\eta_0,\eta_0)}
\varphi(z)\,dz
\\ \le & \|e^{-\a|\cdot|}v_1(t,\cdot)\|_{L^2(\R^2)}\|e^{\a|\cdot|}\varphi\|_{L^2(\R)}
\lesssim \|v_1(t)\|_{W(t)}\,.
\end{split}  
\end{equation}
If $\bM_{c,x}(T)\le \delta$ and $\delta$ is sufficiently small, then
there exists $\delta'\in(0,2-\alpha)$ such that
$|c(t,y)-2|\le \delta'$ for every $t\in[0,T]$ and $y\in\R$ and
\begin{equation}
  \label{eq:k2-est}
\begin{split}
\|k_2(t,\cdot)\|_{Y_1}= & \frac{1}{2\sqrt{2\pi}}
\left\|\int_\R v_1(t,z,y)\tc(t,y)\delta\varphi_{c(t,y)}(z)e^{-iy\eta}\,dz
\right\|_{L^\infty(-\eta_0,\eta_0)}
\\ \lesssim &  \|v_1(t)\|_{W(t)}\|\tc(t)\|_Y
\quad\text{for $t\in[0,T]$.}
\end{split}  
\end{equation}
Since $Y_1\subset Y$, we see that \eqref{eq:k-est} follows from
\eqref{eq:k1-est} and \eqref{eq:k2-est}.
Moreover, we have \eqref{eq:k-lim} combining \eqref{eq:k1-est} and
\eqref{eq:k2-est} with \eqref{eq:limIx0}. Thus we complete the proof.
\end{proof}

\end{document}